\documentclass[reqno]{amsart}
\usepackage{amssymb, latexsym}
\usepackage{hyperref}

\usepackage{backref}

\allowdisplaybreaks

\def\XXint#1#2#3{{\setbox0=\hbox{$#1{#2#3}{\int}$} \vcenter{\hbox{$#2#3$}}\kern-.5\wd0}}
 
\renewcommand{\phi}{\varphi}
\renewcommand{\P}{\mathbb{P}}%projection

\newcommand{\nm}[1]{\left\|#1\right\|} % norm
\newcommand{\abs}[1]{\left|#1\right|} % absolute value
\newcommand{\f}{\frac}%%I use this as an abbreviation for fraction
\renewcommand{\bar}[1]{\overline{#1}}

\newcommand{\lec}{\lesssim}

\newcommand{\paren}[1]{\left(#1\right)}
\newcommand{\bracket}[1]{\left[#1\right]}

\newcommand{\HH}{\mathbb{H}}
\newcommand{\af}{\mathfrak{a}}

\renewcommand{\a}{\alpha} %variable in lagrangian coordinates
\renewcommand{\aa}{{\alpha '}} %variable in Riemannian coordinates

\renewcommand{\AA}{\mathcal{A}} %for the Riemannian coordinates
\newcommand{\AAt}{\mathcal{A}_t}%for the Riemannian coordinates
\let\Re=\undefined\DeclareMathOperator*{\Re}{Re}
\let\Im=\undefined\DeclareMathOperator*{\Im}{Im}

\theoremstyle{plain}
\newtheorem{theorem}{Theorem}
\newtheorem{proposition}[theorem]{Proposition}
\newtheorem{lemma}[theorem]{Lemma}
\newtheorem{corollary}[theorem]{Corollary}
\theoremstyle{definition}
\newtheorem{definition}[theorem]{Definition}
\newtheorem{remark}[theorem]{Remark}

\newcounter{smalllist}

\numberwithin{equation}{section} \numberwithin{theorem}{section}
\begin{document}

\title[ Existence of two dimensional water waves with angled crests]{A blow-up criteria and the existence of 2d gravity water waves with angled crests}
\author{Sijue Wu
}
\address{University of Michigan, Ann Arbor, MI}

\thanks{Financial support in part by NSF grants DMS-1101434, DMS-1361791}

\begin{abstract}
We consider the two dimensional gravity water wave equation in the regime that includes free surfaces with angled crests. We assume that  the fluid is inviscid, incompressible and irrotational, the air density is zero, and we neglect the surface tension.  In \cite{kw} it was shown that in this regime, only a degenerate Taylor inequality $-\frac{\partial P}{\partial\bold{n}}\ge 0$  holds, with degeneracy  at the singularities;  an energy functional $\frak E$ was constructed 
and an aprori estimate was proved. In this paper  we show that a (generalized) solution of the water wave equation with smooth data will remain smooth so long as $\frak E(t)$ remains finite; and for any data satisfying $\frak E(0)<\infty$, the  equation is solvable locally in time, for a period depending only on $\frak E(0)$.

\end{abstract}

\maketitle

\section{Introduction}

A class of water wave problems concerns the
motion of the 
interface separating an inviscid, incompressible, irrotational fluid,
under the influence of gravity, 
from a region of zero density (i.e. air) in 
$n$-dimensional space. It is assumed that the fluid region is below the
air region. Assume that
the density  of the fluid is $1$, the gravitational field is
$-{\bold k}$, where ${\bold k}$ is the unit vector pointing in the  upward
vertical direction, and at  
 time $t\ge 0$, the free interface is $\Sigma(t)$, and the fluid
occupies  region
$\Omega(t)$. When surface tension is
zero, the motion of the fluid is  described by 
\begin{equation}\label{euler}
\begin{cases}   \ \bold v_t + \bold v\cdot \nabla \bold v = -\bold k-\nabla P
\qquad  \text{on } \Omega(t),\ t\ge 0,
\\
\ \text{div}\,\bold v=0 , \qquad \text{curl}\,\bold v=0, \qquad  \text{on }
\Omega(t),\ t\ge 
0,
\\  
\ P=0, \qquad\text{on }
\Sigma(t) \\ 
\ (1, \bold v) \text{ 
is tangent to   
the free surface } (t, \Sigma(t)),
\end{cases}
\end{equation}
where $ \bold v$ is the fluid velocity, $P$ is the fluid
pressure. 
There is an important condition for these problems:
\begin{equation}\label{taylor}
-\frac{\partial P}{\partial\bold n}\ge 0
\end{equation}
point-wise on the interface, where $\bold n$ is the outward unit normal to the  interface 
$\Sigma(t)$ \cite{ta}.
It is well known that when surface tension is neglected and the Taylor sign condition \eqref{taylor} fails, the water wave motion can be subject to the Taylor instability \cite{ ta, bi, bhl}.
In  \cite{wu1, wu2}, we showed that  for dimensions $n\ge 2$, the strong Taylor stability criterion 
 \begin{equation}\label{taylor-s}
-\frac{\partial P}{\partial\bold n}\ge c_0>0
\end{equation}
always holds for the
 infinite depth water wave problem \eqref{euler}, as long as the interface is non-self-intersecting and smooth; and the initial value problem of
the  water
wave system \eqref{euler} is uniquely 
solvable  locally in   
time in 
Sobolev spaces $H^s$, $s\ge 4$ for arbitrary given data.  
Earlier work include Nalimov
\cite{na},   
Yosihara \cite{yo} and Craig \cite{cr} on local existence and uniqueness for small and smooth data for the 2d water wave equation \eqref{euler}. 
There have been much work recently, local wellposedness for water waves with additional effects such as surface tension, bottom and vorticity have been proved, c.f. \cite{am, cl, cs, ig1, la, li, ot, sz, zz};
 local wellposedness of \eqref{euler} in low regularity Sobolev spaces where the interfaces are  in $C^{3/2}$ has been obtained, c.f. \cite{abz, abz14}. 
In all of these work, the strong Taylor stability criterion \eqref{taylor-s} is  assumed.\footnote{When there is surface tension, or vorticity, or a bottom,   \eqref{taylor-s} doesn't always hold.} In addition, in the last few years,  almost global and global wellposedness for the water wave equation \eqref{euler} in both two and three dimensional spaces for small, smooth and localized initial data have been proved, c.f. \cite{wu3, wu4, gms, ip, ad}.

In \cite{kw}, we studied the 2d water wave equation \eqref{euler} in the regime that includes free interfaces with angled crests. We constructed an energy functional $\frak E(t)$ in this framework and proved an a priori estimate. In this paper we introduce a notion of generalized solutions of \eqref{euler} --  a generalized solution is classical provided the interface is non-self-intersecting;\footnote{By non-self-intersecting we mean it is a Jordan curve.} we prove  a blow-up criteria that states that for smooth initial data, a unique generalized solution of the 2d water wave equation exists and remains smooth so long as $\frak E(t)$ remains finite; and we show that for data satisfying $\frak E(0)<\infty$, a generalized solution of the 2d water wave equation \eqref{euler} exists for a time period depending only on $\frak E(0)$; if in addition the initial interface is chord-arc,\footnote{A  curve is chord-arc if the arc-length and the chord length between any two points on the curve are comparable.}
 there is a $T>0$, depending only on $\frak E(0)$ and the chord-arc constant, so that the interface remains chord-arc and a classical solution of  the 2d water wave equation \eqref{euler} exists for time $t\in [0, T]$.
The (generalized) solution is constructed by mollifying the initial data and by showing that the sequence of (generalized) solutions for the mollified data converges to a (generalized) solution for the given data.

The rest of the paper is organized as follows: in section~\ref{prelim}, we state and refine the earlier results this paper is built upon, this includes the local wellposedness result for
Sobolev data in \cite{wu1},  and the energy functional $\frak E$ constructed and the a priori estimate proved in \cite{kw}, in the context of generalized solutions; the notion of generalized solutions will be introduced in \S\ref{surface-equation} and \S\ref{general-soln}. In section~\ref{main} we present the main results: a blow-up criteria via the energy functional $\frak E$ and the local existence of water waves with angled crests. We prove the blow-up criteria in sections~\ref{proof1}  and  the local existence in section ~\ref{proof2}. The majority of the notation are introduced in \S\ref{notation1}, with the rest throughout the paper. Some basic preparatory results in analysis are given in Appendix~\ref{ineq}; various identities that are useful for the paper are derived in Appendix~\ref{iden}. Finally in Appendix~\ref{quantities}, we  list the quantities which have been shown in \cite{kw}  are controlled by $\frak E$.

\section{Preliminaries}\label{prelim}

\subsection{Notation and convention}\label{notation1}
We consider solutions of the water wave equation \eqref{euler} in the setting where the fluid domain $\Omega(t)$ is simply connected, with the free interface $\Sigma(t):=\partial\Omega(t)$ a Jordan curve,\footnote{That is, $\Sigma(t)$ is homeomorphic to the line $\mathbb R$.}
 $${\bold v}(z, t)\to 0,\qquad\text{as } |z|\to\infty$$
 and the interface $\Sigma(t)$ tending to horizontal lines at infinity.\footnote{The problem with velocity $\bold v(z,t)\to (c,0)$ as $|z|\to\infty$ can be reduced to the one with $\bold v\to 0$ at infinity by studying  the solutions in a moving frame. $\Sigma(t)$ may tend to two different lines at $+\infty$ and $-\infty$.} 
 
 We use the following notations and conventions:  $[A, B]:=AB-BA$ is the commutator of operators $A$ and $B$.  $H^s(\mathbb R)$ is the Sobolev space with norm $\|f\|_{H^s}:=(\int (1+|\xi|^2)^s|\hat f(\xi)|^2\,d\xi)^{1/2}$, $\dot H^{1/2}$ is the Sobolev space with norm $\|f\|_{\dot H^{1/2}}:= (\int |\xi| |\hat f(\xi)|^2\,d\xi)^{1/2}$, $L^p=L^p(\mathbb R)$ is the $L^p$ space with $\|f\|_{L^p}:=(\int|f(x)|^p\,dx)^{1/p}$ for $1\le p<\infty$ and $\|f\|_{L^\infty}:=\text{ess sup }|f(x)|$. 
We write $f(t):=f(\cdot, t)$, with $\|f(t)\|_{H^s}$ being the Sobolev norm,  $\|f(t)\|_{L^p}$ being the $L^p$ norm of    $f(t)$ in the spatial variable. 
When not specified, all the $H^s$ and $L^p$ norms are in terms of the spatial variables.
 Compositions are always in terms of the spatial variables and we write for $f=f(\cdot, t)$, $g=g(\cdot, t)$, $f(g(\cdot,t),t):=f\circ g(\cdot, t):=U_gf(\cdot,t)$. 
We identify $(x,y)$ with the complex number $x+iy$; $\Re z$, $\Im z$ are the real and imaginary parts of $z$; $\bar z=\Re z-i\Im z$ is the complex conjugate of $z$. $\overline \Omega$ is the closure of the domain $\Omega$, $\partial\Omega$ is the boundary of $\Omega$, $P_-:=\{z\in \mathbb C: \Im z<0\}$ is the lower half plane.
We write 
\begin{equation}\label{eq:comm}
[f,g; h]:=\frac1{\pi i}\int\frac{(f(x)-f(y))(g(x)-g(y))}{(x-y)^2}h(y)\,dy.
\end{equation}

We use $c$, $C$  to denote universal constants and $c(a,b)$, $C(a)$, $M(a)$ etc. to denote constants that depends on $a, b$ and respectively $a$ etc. Constants appearing in different contexts need not be the same. We write $f\lesssim g$ if there is a universal constant $c$, such that $f\le cg$.  RHS, LHS are the short codes for the "right hand side" and the "left hand side".

\subsection{The equation for the free surface in Lagrangian and Riemann mapping variables}\label{surface-equation}

Let the free interface $\Sigma(t): z=z(\alpha, t)$, $\alpha\in\mathbb R$ be given by Lagrangian parameter $\alpha$, so $z_t(\alpha, t)={\bold v}(z(\alpha,t);t)$ is the velocity  of the fluid particles on the interface, $z_{tt}(\alpha,t)={\bold v_t + (\bold v\cdot \nabla) \bold v}(z(\alpha,t); t)$ is the acceleration; 
notice that $P=0$ on $\Sigma(t)$ implies that $\nabla P$ is normal to $\Sigma(t)$, therefore $\nabla P=-i\frak a z_\alpha$, where 
\begin{equation}\label{frak-a}
\frak a =-\frac1{|z_\alpha|}\frac{\partial P}{\partial {\bold n}};
\end{equation}
 and the first and third equation of \eqref{euler} gives 
\begin{equation}\label{interface-l}
z_{tt}+i=i\frak a z_\alpha.
\end{equation}
The second equation of \eqref{euler}: $\text{div } \bold v=\text{curl } \bold v=0$ implies that $\bar {\bold v}$ is holomorphic in the fluid domain $\Omega(t)$, hence $\bar z_t$ is the boundary value of a holomorphic function in $\Omega(t)$. By Proposition~\ref{prop:hilbe}  the second equation of \eqref{euler} is equivalent to $\bar z_t=\mathfrak H {\bar z_t}$, where $\frak H$ is the Hilbert transform associated with the fluid domain $\Omega(t)$. So the motion of the fluid interface $\Sigma(t): z=z(\alpha,t)$ is given by 
\begin{equation}\label{interface-e}
\begin{cases}
z_{tt}+i=i\frak a z_\alpha\\
\bar z_t=\frak H \bar z_t.
\end{cases}
\end{equation}
\eqref{interface-e} is a fully nonlinear equation.  In \cite{wu1},  Riemann mapping was introduced to analyze the quasi-linear structure of \eqref{interface-e}.

Let $\Phi(\cdot, t): \overline{\Omega(t)}\to \overline P_-$ be the Riemann mapping taking $\overline{\Omega(t)}$ to the closure of the lower half plane $\overline P_-$, 
satisfying $\lim_{z\to\infty}\Phi_z(z,t)=1$. Let 
\begin{equation}\label{h}
h(\alpha,t):=\Phi(z(\alpha,t),t),
\end{equation}
so $h:\mathbb R\to\mathbb R$ is a homeomorphism. Let $h^{-1}$ be defined by 
$$h(h^{-1}(\alpha',t),t)=\alpha',\quad \alpha'\in \mathbb R;$$
and 
\begin{equation}\label{1001}
Z(\alpha',t):=z\circ h^{-1}(\alpha',t),\quad Z_t(\alpha',t):=z_t\circ h^{-1}(\alpha',t),\quad Z_{tt}(\alpha',t):=z_{tt}\circ h^{-1}(\alpha',t)
\end{equation}
be the reparametrization of the position, velocity and acceleration of the interface in the Riemann mapping variable $\alpha'$. Let
\begin{equation}\label{1002}
Z_{,\alpha'}(\alpha', t):=\partial_{\alpha'}Z(\alpha', t),\quad Z_{t,\alpha'}(\alpha', t):=\partial_{\alpha'}Z_t(\alpha',t), \quad Z_{tt,\alpha'}(\alpha', t):=\partial_{\alpha'}Z_{tt}(\alpha',t),
\end{equation}
etc. We note that $\Phi^{-1}(\alpha', t)=Z(\alpha', t)$, so $(\Phi^{-1})_{z'}(\alpha',t)=Z_{,\alpha'}(\alpha',t)$, and by Proposition~\ref{prop:hilbe},  
\begin{equation}\label{interface-holo}
Z_{,\alpha'}-1=\mathbb H(Z_{,\alpha'}-1),\qquad \frac1{Z_{,\alpha'}}-1=\mathbb H(\frac1{Z_{,\alpha'}}-1).
\end{equation}
Observe that ${\bar {\bold v}}\circ \Phi^{-1}: P_-\to \mathbb C$ is holomorphic in the lower half plane $P_-$ with  ${\bar {\bold v}}\circ \Phi^{-1}(\alpha', t)={\bar Z}_t(\alpha',t)$. Precomposing \eqref{interface-l}  with $h^{-1}$ and applying Proposition~\ref{prop:hilbe} to $
{\bar {\bold v}}\circ \Phi^{-1}$ on $P_-$ gives the free surface equation in the Riemann mapping variable:
\begin{equation}\label{interface-r}
\begin{cases}
Z_{tt}+i=i\mathcal AZ_{,\alpha'}\\
\bar{Z}_t=\mathbb H \bar{Z}_t
\end{cases}
\end{equation}
where $\mathcal A\circ h=\frak a h_\alpha$ and $\mathbb H$ is the Hilbert transform associated with the lower half plane $P_-$:
$$\mathbb H f(\alpha')=\frac1{\pi i}\text{pv.}\int\frac1{\alpha'-\beta'}\,f(\beta')\,d\beta'.$$
From the chain rule, we know for $f=f(\cdot, t)$, $U_h^{-1}\partial_tU_h f= (\partial_t+b\partial_{\alpha'})f$ where $$b:=h_t\circ h^{-1};$$ so $Z_{tt}=(\partial_t+b\partial_{\alpha'})Z_t=(\partial_t+b\partial_{\alpha'})^2Z$. Let $A_1:=\mathcal A |Z_{,\alpha'}|^2$. Multiply $\bar Z_{,\alpha'}$ to the first equation of \eqref{interface-r} yields
\begin{equation}\label{interface-a1}
\bar{Z}_{,\alpha'}(Z_{tt}+i)=iA_1.
\end{equation}
In \cite{wu1}, it was shown  that systems \eqref{euler}, \eqref{interface-e} and \eqref{interface-r}-\eqref{interface-holo} with $b$, $A_1$ given by \eqref{b}, \eqref{a1} are equivalent in the regime of nonself-intersecting interfaces $z=z(\cdot,t)$.\footnote{When $\Sigma(t): Z=Z(\cdot,t)$ becomes self-intersecting, it is not physical to assume $P\equiv 0$ on $\Sigma(t)$. So in general we do not consider beyond the regime of non-self-intersecting interfaces. }

However  the system \eqref{interface-r}-\eqref{interface-holo} is well defined even if $Z=Z(\cdot,t)$ is self-intersecting. In constructing the approximating sequence of solutions from the mollified data, it is convenient to allow self-intersecting solutions of \eqref{interface-r}-\eqref{interface-holo}. In this context, $Z$ and $z$, $Z_t$, $z_t$ etc. are related via \eqref{1001} and \eqref{1002} through a homeomorphism $h=h(\cdot,t):\mathbb R\to\mathbb R$, and from \eqref{interface-r}-\eqref{interface-holo} we can show that $h$, $A_1$ satisfy \eqref{b}-\eqref{a1}, see Appendix~\ref{basic-iden}. For not necessarily non-self-intersecting solutions $Z$ of \eqref{interface-r}-\eqref{interface-holo} we will abuse terminologies by continue saying $Z$, $Z_t$ etc. are in the Riemann mapping variable, $z$, $z_t$ etc. are in the Lagrangian coordinates, $Z$, $Z_t$, $Z_{tt}$ are the interface, velocity and acceleration. 

 Let's consider the solution of \eqref{interface-r}-\eqref{interface-holo} in the "fluid domain".\footnote{It makes sense to talk about fluid domain only when $Z=Z(\cdot, t)$ is non-self-intersecting. Here we just abuse the terminology.}

\subsection{Generalized solutions of the water wave equation}\label{general-soln}\footnote{Here and in \S\ref{surface-equation} we give a generic discussion. The statements  are rigorous if the quantities involved are sufficiently regular.} Let $Z=Z(\cdot, t)$ be a solution of \eqref{interface-r}-\eqref{interface-holo}, let $F(\cdot, t): P_-\to \mathbb C$, $\Psi(\cdot, t): P_-\to \mathbb C$ be holomorphic functions, continuous on $\bar P_-$, such that
\begin{equation}\label{eq:270}
F(\alpha',t)=\bar Z_t(\alpha', t),\qquad \Psi(\alpha',t)=Z(\alpha',t),\qquad \Psi_{z'}(\alpha',t)=Z_{,\alpha'}(\alpha', t).
\end{equation}
By \eqref{c3} of Appendix~\ref{basic-iden} and \eqref{eq:270},
\begin{equation}\label{eq:271}
h_t\circ h^{-1}=\frac{Z_t}{Z_{,\alpha'}}-\frac{\Psi_t}{\Psi_{z'}}=\frac{\bar F}{\Psi_{z'}}-\frac{\Psi_t}{\Psi_{z'}} .
\end{equation}
Now $\bar z_t(\alpha,t)=\bar Z_t(h(\alpha,t),t)=F(h(\alpha,t),t)$, so
$$\bar z_{tt}=F_t\circ h+F_{z'}\circ h h_t=U_h\{F_t-\frac{\Psi_t}{\Psi_{z'}}F_{z'}+\frac{\bar F}{\Psi_{z'}}F_{z'} \}$$
therefore $\bar Z_{tt}$ is the trace  of the function $F_t-\frac{\Psi_t}{\Psi_{z'}}F_{z'}+\frac{\bar F}{\Psi_{z'}}F_{z'} $ on $\partial P_-$; $Z_{,\alpha'}(\bar Z_{tt}-i)$ is then the trace of the function $\Psi_{z'} 
F_t- {\Psi_t}F_{z'}+{\bar F}F_{z'} -i\Psi_{z'}$  on $\partial P_-$.  By \eqref{interface-a1}, 
\begin{equation}\label{eq:272}
\Psi_{z'} F_t- {\Psi_t}F_{z'}+{\bar F}F_{z'}-i\Psi_{z'}=iA_1,\qquad \text{on }\partial P_- .
\end{equation}
On the left hand side of \eqref{eq:272}, $\Psi_{z'} F_t- {\Psi_t}F_{z'}-i\Psi_{z'}$ is holomorphic on $P_-$,
while
${\bar F}F_{z'}=\partial_{z'}(\bar F F)$; we recall from complex analysis, $\partial_{z'}=\frac12(\partial_{x'}-i\partial_{y'})$. So there is a real valued function $\frak P:\P_-\to \mathbb R$, such that
\begin{equation}\label{eq:273}
\Psi_{z'} F_t- {\Psi_t}F_{z'}+{\bar F}F_{z'}-i\Psi_{z'}=-(\partial_{x'}-i\partial_{y'})\frak P,\qquad \text{on }P_-
\end{equation}
moreover by \eqref{eq:272}, because $iA_1$ is purely imaginary,
\begin{equation}\label{eq:274}
\frak P=0,\qquad \text{on }\partial P_-.
\end{equation}
We note that by applying $\partial_{x'}+i\partial_{y'}$ to both sides of \eqref{eq:273}, $\frak P$ satisfies
\begin{equation}\label{eq:275}
\Delta \frak P= -2|F_{z'}|^2\qquad\text{on }P_-.
\end{equation}

If in addition $\Sigma(t)=\{Z=Z(\alpha',t):=\Psi(\alpha',t)\ | \ \alpha'\in\mathbb R\}$  is a Jordan curve with $$\lim_{|\alpha'|\to\infty} Z_{,\alpha'}(\alpha',t)=1,$$ let $\Omega(t)$ be the domain bounded by $Z=Z(\cdot,t)$ from the above, then $Z=Z(\alpha',t) $, $\alpha'\in\mathbb R$ winds the boundary of $\Omega(t)$ exactly once. By the argument principle, $\Psi: \bar P_-\to \bar \Omega(t)$ is one-to-one and onto,  $\Psi^{-1}:\Omega(t)\to P_-$
exists and is a holomorphic function. In this case, it is easy to check by the chain rule that equation \eqref{eq:273} is equivalent to
\begin{equation}\label{eq;276}
(F\circ \Psi^{-1})_t+\bar F\circ \Psi^{-1}(F\circ \Psi^{-1})_{z}+(\partial_x-i\partial_y)(\frak P\circ \Psi^{-1})=i,\qquad \text{on }\Omega(t)
\end{equation}
This is the Euler equation, i.e. the first equation of \eqref{euler} in complex form. Therefore  $\bar {\bold v}=F\circ \Psi^{-1}$, $P=\frak P\circ \Psi^{-1}$ is a solution of the water wave equation \eqref{euler}, with $\Sigma(t): Z=Z(\cdot,t)$ the boundary of the fluid domain $\Omega(t)$.

In what follows we give the local wellposedness result of \cite{wu1} and the a priori estimate of \cite{kw} for solutions of \eqref{interface-r}-\eqref{interface-holo}. 

\subsection{Local wellposedness in Sobolev spaces}
In \cite{wu1}  we derived 
a quasi-linearization of \eqref{interface-r}-\eqref{interface-holo}, the system (4.6)-(4.7) of \cite{wu1} by taking one derivative to $t$ to equation \eqref{interface-l}
and analyzed the quantities $b$ and $A_1$;\footnote{ \cite{wu6} has a slightly different and shorter derivation. \cite{kw} has the derivation in a periodic setting.  The reader may want to consult \cite{kw, wu6} for 
the derivations. The identities in Appendix~\ref{basic-iden} provide yet another derivation of  the quasi-linearization and \eqref{b}, \eqref{a1} from \eqref{interface-r}-\eqref{interface-holo}, without assuming $Z=Z(\cdot,t)$ being non-self-intersecting. We note that \eqref{taylor-formula} only makes sense for non-self-intersecting interfaces.} and via $A_1$, we showed that the strong Taylor inequality \eqref{taylor-s} always holds for smooth nonself-intersecting interfaces. In addition, we proved that the Cauchy problem of the system (4.6)-(4.7) of \cite{wu1} is  locally well-posed in Sobolev spaces. 
\begin{proposition}[Lemma 3.1 and (4.7) of \cite{wu1}, Proposition 2.2 and (2.18) of \cite{wu6}]\label{prop:a1}  
We have 1. \begin{equation}\label{b}
b:=h_t\circ h^{-1}=\Re \big([Z_t,\mathbb H](\frac1{Z_{,\alpha'}}-1)\big)+2\Re Z_t.
\end{equation}
2. \begin{equation}\label{a1}
A_1=1-\Im [Z_t,\mathbb H]{\bar Z}_{t,\alpha'}=1+\frac1{2\pi }\int \frac{|Z_t(\alpha', t)-Z_t(\beta', t)|^2}{(\alpha'-\beta')^2}\,d\beta'\ge 1.
\end{equation}
3. \begin{equation}\label{taylor-formula}
-\frac{\partial P}{\partial\bold n}\big |_{Z=Z(\cdot,t)}= \frac{A_1}{|Z_{,\alpha'}|};
\end{equation}
in particular if the interface $\Sigma(t)\in C^{1,\gamma}$ for some $\gamma>0$, then the strong Taylor sign condition \eqref{taylor-s} holds.
\end{proposition}

\begin{remark}
By \eqref{taylor-formula}, the Taylor sign condition \eqref{taylor} always holds. Assume  $\Sigma(t)$ is non-self-intersecting with angled crests, assume the interior angle at a crest is $\nu$. Around the crest, we know the Riemann mapping $\Phi^{-1}$ (we move the singular point to the origin) behaves like 
$$\Phi^{-1}(z')\approx (z')^r,\qquad \text{with } \nu=r\pi$$
so $Z_{,\alpha'}\approx (\alpha')^{r-1}$. From \eqref{interface-a1} and the fact $A_1\ge 1$,  the interior angle at the crest must be $\le \pi$ if the acceleration $|Z_{tt}|\ne\infty$,
and $-\frac{\partial P}{\partial\bold n}=0$  at the singularities where the interior angles are $<\pi$,\footnote{
$\mathcal A:=\frac{A_1}{|Z_{\alpha'}|^2}$ equals zero alongside $-\frac{\partial P}{\partial\bold n}$.}
 cf. \cite{kw}, \S3.
\end{remark}

 Let $h(\alpha,0)=\alpha$ for $\alpha\in\mathbb R$; let the initial interface $Z(\cdot,0):=Z(0)$, the initial velocity $Z_t(\cdot, 0):=Z_t(0)$  be given 
such that $Z(0)$ satisfy \eqref{interface-holo} and 
$Z_t(0)$ satisfy $\bar Z_t(0)=\mathbb H \bar Z_t(0)$; let $A_1$ be given by \eqref{a1}, the initial acceleration $Z_{tt}(0)$ satisfy
\eqref{interface-a1}, 
and 
$a_0=\frac{A_1(\cdot,0)}{|Z_{,\alpha'}(\cdot,0)|^2}$. By Theorem 5.11 of \cite{wu1} and a refinement of the argument in \S 6 of \cite{wu1},  the following local existence result holds.

\begin{proposition}[local existence in Sobolev spaces, cf. Theorem 5.11, \S6 of \cite{wu1}]\label{prop:local-s} Let $s\ge 4$. Assume that 
$Z_t(0)\in H^{s+1/2}(\mathbb R)$, $Z_{tt}(0)\in H^s(\mathbb R)$ and $a_0\ge c_0>0$ for some constant $c_0>0$.\footnote{Let $s\ge 4$. as a consequence of $(Z_t(0), Z_{tt}(0))\in H^{s+1/2}(\mathbb R)\times H^s(\mathbb R)$ and $a_0\ge c_0>0$, $Z_{,\alpha'}(0)-1\in H^s(\mathbb R)$. In general by \eqref{interface-a1}, $(Z_t, Z_{tt})\in H^{s+1/2}\times H^s$ implies $\frac1{Z_{,\alpha'}}-1\in H^s$, and $(Z_t,  \frac1{Z_{,\alpha'}}-1)\in H^{s+1/2}\times H^s$ implies $Z_{tt}\in H^s$. }    Then there is $T>0$,\footnote{$T$ depends only on $c_0$, $\|Z_{t}(0)\|_{H^{s+1/2}}$ and $\|Z_{tt}(0)\|_{H^s}$.} such that on $[0, T]$, the initial value problem of \eqref{interface-r}-\eqref{interface-holo}-\eqref{b}-\eqref{a1} has a unique solution 
$Z=Z(\cdot, t)$,  satisfying
 $(Z_{tt}, Z_t)\in C^
l([0, T], H^{s-l}(\mathbb R)\times H^{s+1/2-l}(\mathbb R))$,  and $Z_{,\alpha'}-1\in C^l([0, T], H^{s-l}(\mathbb R))$, for $l=0,1$.

Moreover if $T^*$ is the supremum over all such times $T$, then either $T^*=\infty$, or $T^*<\infty$, but
\begin{equation}\label{eq:1}
\sup_{[0, T^*)}(\|Z_{tt}(t)\|_{H^3}+\| Z_t(t)\|_{H^{3+1/2}})=\infty. \end{equation}
\end{proposition}

\begin{proof}  Notice that the system (4.6)-(4.7) of \cite{wu1} is a system for the horizontal velocity $w=\Re Z_{t}$ and horizontal acceleration $u=\Re Z_{tt}$, the interface doesn't appear explicitly; it is well-defined even if the interface $Z=Z(\cdot,t)$ is self-intersecting.  The first part of Proposition~\ref{prop:local-s}
follows from Theorem 5.11,  and the argument  from the second half of page 70 to the first half of page 71 of \S6 of \cite{wu1}.

Now assume $T^*<\infty$, and 
\begin{equation}\label{eq:10}
\sup_{[0, T^*)}(\|Z_{tt}(t)\|_{H^3}+\| Z_t(t)\|_{H^{3+1/2}}):=M_0<\infty.
\end{equation}
We want to show that the solution $Z$ of the system \eqref{interface-r}-\eqref{interface-holo}-\eqref{b}-\eqref{a1} can be extended beyond $T^*$ by a time $T'>0$ that depends only on $M_0$, $c_0$, $\|Z_t(0)\|_{H^s}$ and $\|Z_{tt}(0)\|_{H^s}$, contradicting with the maximality of $T^*$.

Let $T<T^*$ be arbitrary chosen. Let $a=a(\cdot, t)$, $b=b(\cdot,t)$ be given by (4.7) of \cite{wu1}, and let $h=h(\cdot, t)$ satisfy 
\begin{equation}\label{eq:3}
\begin{cases}
\frac{dh}{dt}=b(h,t)\\
h(\alpha, 0)=\alpha.
\end{cases}
\end{equation}
By Theorem 5.11 of \cite{wu1} and the argument in \S 6 of \cite{wu1},  we know $b\in C([0, T], H^{s+1/2}(\mathbb R))$ with $\|b(t)\|_{H^{s+1/2}}\le c(\|Z_{t}(t)\|_{H^{s+1/2}}, \|Z_{tt}(t)\|_{H^s})$, and $h(\cdot, t): \mathbb R\to \mathbb R$ is a diffeomorphism with $h(\alpha,t)-\alpha\in C([0, T], H^{s+1/2})$. Moreover  $Z(\alpha',t):=z\circ h^{-1}(\alpha',t)$ satisfies \eqref{interface-a1}, and for $t\in [0, T]$, 
\begin{equation}\label{eq:2}
\|Z_{tt}(t)\|_{H^s}+\|Z_t(t)\|_{H^{s+1/2}}\le d_0 e^{Kt} (\|Z_{tt}(0)\|_{H^s}+\|Z_t(0)\|_{H^{s+1/2}}),
\end{equation}
where $K=K(M(T), \frak a(T), s)$, $d_0=d(M(T), \frak a(T), s)$ are constants depending on $$\frak a (T):=\inf_{\mathbb R\times [0, T]} a(\alpha', t),\qquad 
M(T)=\sup_{[0, T]}(  \|Z_{tt}(t)\|_{H^3}+\|Z_t(t)\|_{H^{3+1/2}} ),$$
and $K(M(T), \frak a(T), s)\to \infty$, $d(M(T), \frak a(T), s)\to\infty$ as $M(T)\to\infty$, $\frak a(T)\to 0$.
We want to show that $\frak a(T)\ge \frac1{C(M_0, c_0)}$ for some constant $C(M_0, c_0)>0$.

By \eqref{interface-a1}, 
$$a(\alpha',t):=\frac{|Z_{tt}(t)+i|^2}{A_1(t)}=\frac{A_1(t)}{|Z_{,\alpha'}(t)|^2},$$
so it suffices to show that there is a constant $c(M_0, c_0)$, such that $\|Z_{,\alpha'}(t)\|_{L^\infty}\le c(M_0, c_0)$ for all $t\in [0, T]$. From the assumption
 $a_0=\frac{A_1(\cdot,0)}{|Z_{,\alpha'}(\cdot,0)|^2}\ge c_0$,  $|Z_{,\alpha'}(\cdot,0)|^2\le \frac{A_1(\cdot,0)}{c_0}$. Applying the Hardy's inequality Proposion~\ref{hardy-inequality} and Cauchy-Schwarz on \eqref{a1} yields $$\|Z_{,\alpha'}(0)\|^2_{L^\infty}\le \frac{\|A_1(0)\|_{L^\infty}}{c_0}\lesssim \frac{1+\|Z_{t,\alpha'}(0)\|_{L^2}^2}{c_0}.$$ We calculate $\|Z_{,\alpha'}(t)\|_{L^\infty}$ by the fundamental theorem of calculus. 

Differentiating \eqref{eq:3} gives
\begin{equation}\label{eq:4}
\begin{cases}
\frac{dh_{\alpha}}{dt}=b_{\alpha'}(h,t)h_{\alpha}\\
h_{\alpha}(\alpha, 0)=1.
\end{cases}
\end{equation}
So on $[0, T]$, 
$$e^{-\int_0^t\|b_{\alpha'}(\tau)\|_{L^\infty(\mathbb R)}\,d\tau }\le h_\alpha(\alpha,t)\le e^{\int_0^t\|b_{\alpha'}(\tau)\|_{L^\infty(\mathbb R)}\,d\tau};$$
and by Sobolev embedding, $\|b_{\alpha'}(\tau)\|_{L^\infty(\mathbb R)}\lesssim \|b(\tau)\|_{H^2(\mathbb R)} \le c(\|Z_t(\tau)\|_{H^{2}}, \|Z_{tt}(\tau)\|_{H^2})$. Because $h(\alpha,0)=\alpha$, $z(\alpha,0)=Z(\alpha,0)$ and 
$$z_\alpha(\alpha,t)=Z_\alpha(\alpha,0)+\int_0^t z_{t\alpha}(\alpha,\tau)\,d\tau.$$
By the chain rule $z_{t\alpha}=Z_{t,\alpha'}\circ h h_\alpha$, $z_\alpha=Z_{,\alpha'}\circ h h_\alpha$; so for $t\in [0, T]$, 
$$\|Z_{,\alpha'}(t)\|_{L^\infty}\le (\|Z_{,\alpha'}(0)\|_{L^\infty}+\int_0^t\|Z_{t,\alpha'}(\tau)\|_{L^\infty} \|{h_\alpha}(\tau)\|_{L^\infty}  \,d\tau)\|\frac1{h_\alpha(t)}\|_{L^\infty}\le C(M_0, c_0)$$
for some constant $C(M_0, c_0)$ depending on $M_0$, $c_0$ and $T^*$, and 
\begin{equation}\label{eq:5}
\frak a(T)=\inf_{\mathbb R\times [0, T]} a(\alpha', t)=\inf_{\mathbb R\times [0, T]}\frac{A_1}{|Z_{,\alpha'}|^2} \ge \frac1{C(M_0, c_0)}.
\end{equation}

Now \eqref{eq:10}, \eqref{eq:2} and \eqref{eq:5} gives
\begin{equation*}
\begin{aligned}
\|Z_{tt}(T)\|_{H^s}+\|Z_t(T)\|_{H^{s+1/2}}&\le c(M_0, c_0, \|Z_{tt}(0)\|_{H^s}, \|Z_t(0)\|_{H^{s+1/2}}),\\
\text{and }\quad a(\cdot, T)&\ge \frac1{C(M_0,  c_0)}>0.
\end{aligned}
\end{equation*}
So by the first part of Proposition~\ref{prop:local-s}, the solution $Z$ can be extended onto $[T, T+T']$, for some $T'>0$ depending only on $M_0, c_0$ and $\|Z_{tt}(0)\|_{H^s}, \|Z_t(0)\|_{H^{s+1/2}}$. This contradicts with the definition of $T^*$, so either $T^*=\infty$ or \eqref{eq:1} holds.

\end{proof}

 Let $D_\alpha:=\frac 1{z_\alpha}\partial_\alpha $ and $D_{\alpha'}:=\frac 1{Z_{,\alpha'}}\partial_{\alpha'}$. By \eqref{interface-holo} and the basic fact that product of holomorphic functions is holomorphic, if $g$ is the boundary value of a holomorphic function on $P_-$, then $D_{\alpha'}g$ is also the boundary value of a holomorphic function on $P_-$. 
 Notice that 
 for any function $f$, 
$$(D_\alpha f)\circ h^{-1}=D_{\alpha'} (f\circ h^{-1}).$$

\subsection{An a priori estimate for water waves with angled crests}\label{a priori}

In \cite{kw}, we studied the water wave equation \eqref{euler} in the regime that includes interfaces with angled crests in a symmetric periodic setting, we constructed an energy functional for this regime and proved an a priori estimate. 
The same analysis applies to the whole line setting. The main difference is that in the whole line case, we do not
need to consider the means of the various quantities; and in the proof of the a priori estimate, the argument in the footnote 21 of \cite{kw} works, so we do not need the Peter-Paul trick. Hence in the whole line case, that part of the proof is simpler. Additionally, with the minor modifications given in Appendix~\ref{basic-iden}, 
the argument in \cite{kw} applies more generally to solutions  of \eqref{interface-r}-\eqref{interface-holo}-\eqref{b}-\eqref{a1},  without any non-self-intersecting assumptions, and the characterization of the energy given in \S10 of \cite{kw} also holds.
In this subsection we present the results of \cite{kw} in the whole line setting for solutions of \eqref{interface-r}-\eqref{interface-holo}-\eqref{b}-\eqref{a1}, we will only show how to handle the differences between the symmetric periodic and the whole line cases.

Let 
\begin{equation}
  \label{eq:ea}
  E_a(t) = \|(\partial_t D_\alpha^2 \bar{z}_t) \circ h^{-1}(t)\|_{L^2(1/A_1)}^2 +  \|\frac{1}{Z_{,\alpha'}}D_{\alpha'}^2 \bar{Z}_t(t)\|_{\dot{H}^{1/2}}^2 + \| D_{\alpha'}^2 \bar{Z}_t(t)\|_{L^2(1/A_1)}^2.
\end{equation}
and
\begin{equation}
  \label{eq:eb}
  E_b(t) = \|\partial_t D_\alpha \bar{z}_t(t)\|_{L^2(\frac{1}{\frak a})}^2 +  \| D_{\alpha'} \bar{Z}_t(t)\|_{\dot{H}^{1/2}}^2 + \|\bar{Z}_{t,\alpha'}(t)\|_{L^2}.
 \end{equation}
Let
\begin{equation}\label{energy}
\frak E(t)= E_a(t)+E_b(t)+ \|\bar Z_{tt}(t)-i\|_{L^\infty}
\end{equation}
Notice that we replaced the third term $|\bar{z}_{tt}(\alpha_0, t)-i|$ in the energy of \cite{kw} by $\|\bar Z_{tt}(t)-i\|_{L^\infty}$. In \S10 of \cite{kw}, we showed that the regime $\frak E<\infty$ includes interfaces with angled crests with interior angles  $<\frac \pi 2$, in particular, the self-similar solutions constructed in \cite{wu5} has finite energy $\frak E$.

\begin{theorem}[cf. Theorem 2 of \cite{kw}]\label{prop:a priori} 
Let 
$Z=Z(\cdot,t)$, $t\in [0, T']$ be a solution of the system \eqref{interface-r}-\eqref{interface-holo}-\eqref{b}-\eqref{a1}, satisfying 
$(Z_{tt,\alpha'}, Z_{t,\alpha'})\in C^
l([0, T'], H^{s-l}(\mathbb R)\times H^{s+1/2-l}(\mathbb R))$,  $l=0,1$ for some $s\ge 3$ and $Z_{tt}\in C([0, T'], L^\infty(\mathbb R))$. Then there are  $T:=T(\frak E(0))>0$, $C=C(\frak E(0))>0$, depending only on $\frak E(0)$,\footnote{$T(e)$ is decreasing with respect to $e$, and $C(e)$ is increasing with respect to $e$.} such that 
\begin{equation}\label{a priori-e}
\sup_{[0, \min\{T, T'\}]}\frak E(t)\le C(\frak E(0))<\infty.
\end{equation}
\end{theorem}

\begin{proof}
Let $h(\alpha,0)=\alpha$, $\alpha\in\mathbb R$, $$\frak e(t)=E_a(t)+E_b(t)
$$
We only need to show how to handle the term $\|\bar Z_{tt}(t)-i\|_{L^\infty}$. The argument in \S4.4.3 of
\cite{kw}  shows that for each given $\alpha\in\mathbb R$,
\begin{equation}
\frac{d}{dt}|\bar z_{tt}(\alpha,t)-i|\le (\|\frac{\frak a_t}{\frak a}\|_{L^\infty}+\|D_{\alpha}\bar z_t\|_{L^\infty})
|\bar z_{tt}(\alpha,t)-i|.
\end{equation}
Notice from the estimate for $\|\frac{\frak a_t}{\frak a}\|_{L^\infty}$ in \cite{kw} and Sobolev embedding that in fact
$$\|\frac{\frak a_t}{\frak a}\|_{L^\infty}+\|D_{\alpha}\bar z_t\|_{L^\infty}\le c(\frak e),$$
where $c$ is a polynomial with  nonnegative universal coefficients. Therefore
$$\frac{d}{dt}|\bar z_{tt}(\alpha,t)-i|\le c(\frak e)
|\bar z_{tt}(\alpha,t)-i|.$$
By Gronwall,
$|\bar z_{tt}(\alpha,t)-i|\le |\bar z_{tt}(\alpha,0)-i|e^{\int_0^t c(\frak e(\tau))\,d\tau}$
hence
$$\|\bar z_{tt}(t)-i\|_{L^\infty}\le \|\bar z_{tt}(0)-i\|_{L^\infty}e^{\int_0^t c(\frak e(\tau))\,d\tau}.$$
Now let
$$\frak E_1(t)= \frak e(t)+\|\bar z_{tt}(0)-i\|_{L^\infty}e^{\int_0^t c(\frak e(\tau))\,d\tau}$$
so $\frak E(t)\le \frak E_1(t)$, and  $\frak E(0)= \frak E_1(0)$. By the whole line counterpart of Theorem 2 of \cite{kw}, $\frac{d}{dt}\frak e(t)\le p(\frak E(t))$ for some polynomial $p$ with nonnegative universal coefficients,  therefore
\begin{equation}
\frac{d}{dt}\frak E_1(t)\le p(\frak E_1(t))+C(\frak E_1(t)).
\end{equation}
Applying Gronwall again yields the conclusion of Theorem~\ref{prop:a priori}.

\end{proof}

Let \begin{equation}\label{energy1}
\begin{aligned}
\mathcal E(t)=\|\bar Z_{t,\alpha'}\|_{L^2}^2&+ \|D_{\alpha'}^2\bar Z_t\|_{L^2}^2+\|\partial_{\alpha'}\frac1{Z_{,\alpha'}}\|_{L^2}^2+\|D_{\alpha'}^2\frac1{Z_{,\alpha'}}\|_{L^2}^2\\&+\|\frac1{Z_{,\alpha'} }D_{\alpha'}^2\bar Z_t  \|_{\dot H^{1/2}}^2+\| D_{\alpha'}\bar Z_t  \|_{\dot H^{1/2}}^2+\|\frac1{Z_{,\alpha'}}\|_{L^\infty}^2.
\end{aligned}
\end{equation}
As was shown in \S10 of \cite{kw}, we have the following characterization of the energy $\frak E$. 

\begin{proposition}[A characterization of $\frak E$ via $\mathcal E$, cf. \S10 of \cite{kw}] \label{prop:energy-eq}
There are polynomials $C_1$ and $C_2$,
with nonnegative universal coefficients, such that for solutions $Z$ of \eqref{interface-r}-\eqref{interface-holo}, 
\begin{equation}\label{energy-equiv}
\mathcal E(t)\le C_1(\frak E(t)),\qquad\text{and}\quad \frak E(t)\le C_2(\mathcal E(t)).
\end{equation}

\end{proposition}

\subsection{A description of the class $\mathcal E<\infty$ in the fluid domain}
We give here an equivalent description of the class $\mathcal E<\infty$ for solutions $Z$ of \eqref{interface-r}-\eqref{interface-holo} in the "fluid domain".

Let $1< p\le \infty$, and 
\begin{equation}\label{poisson}
K_y(x)=\frac{-y}{\pi(x^2+y^2)},\qquad y<0
\end{equation}
be the Poisson kernel. 
We know for any holomorphic function $G$ on $P_-$, 
$$\sup_{y<0}\|G(x+iy)\|_{L^p(\mathbb R,dx)}<\infty$$
if and only if there exists $g\in L^p(\mathbb R)$ such that $G(x+iy)=K_y\ast g(x)$. In this case,
$\sup_{y<0}\|G(x+iy)\|_{L^p(\mathbb R,dx)}=\|g\|_{L^p}$.  Moreover, if $g\in L^p(\mathbb R)$, $1<p<\infty$,  $\lim_{y\to 0-} K_y\ast g(x)=g(x)$  in $L^p(\mathbb R)$ and if $g\in L^\infty\cap C(\mathbb R)$,  $\lim_{y\to 0-} K_y\ast g(x)=g(x)$ for all $x\in\mathbb R$.

Let $Z=Z(\cdot, t)$ be a solution of \eqref{interface-r}-\eqref{interface-holo}, let $\Psi$, $F$ be holomorphic functions on $P_-$, continuous on $\bar P_-$, such that
$$Z(\alpha', t)=\Psi(\alpha', t), \quad \bar Z_t(\alpha', t)=F(\alpha', t).$$ 
Notice that all the quantities in \eqref{energy1} are boundary values of some holomorphic functions on $P_-$. Let $z'=x'+iy'$, where $x', y'\in\mathbb R$. $\mathcal E(t)<\infty$ is equivalent to\footnote{It is clear $\mathcal E(t)=\mathcal E_1(t)$ for smooth $Z=Z(\cdot,t)$. Otherwise this equivalence is understood at a formal level, and is made rigorous according to the circumstances.}
\begin{equation}\label{domain-energy}
\begin{aligned}
&\mathcal E_1(t):=\sup_{y'<0}\|F_{z'}(t)\|_{L^2(\mathbb R,dx')}^2+\sup_{y'<0}\|\frac1{\Psi_{z'}}\partial_{z'}\big(\frac1{\Psi_{z'} }F_{z'}\big)(t)\|_{L^2(\mathbb R,dx')}^2\\&+\sup_{y'<0}\| \partial_{z'}\big(\frac1{\Psi_{z'} }\big)(t)  \|_{L^2(\mathbb R,dx')}^2+\sup_{y'<0}\|\frac1{\Psi_{z'} }(t)  \|_{L^\infty(\mathbb R,dx')}^2
\\&+\sup_{y'<0}\|\frac1{\{\Psi_{z'}\}^2}\partial_{z'}\big(\frac1{\Psi_{z'} }F_{z'}\big)(t)\|_{\dot H^{1/2}(\mathbb R,dx')}^2+\sup_{y'<0}\|\frac1{\Psi_{z'} }F_{z'}(t)\|_{\dot H^{1/2}(\mathbb R,dx')}^2
\\&+\sup_{y'<0}\|  \frac1{\Psi_{z'} }\partial_{z'}\paren{\frac1{\Psi_{z'} }\partial_{z'}\big(\frac1{\Psi_{z'} }\big)}(t)  \|_{L^2(\mathbb R,dx')}^2<\infty.
\end{aligned}
\end{equation}

\section{The main results}\label{main}
We are now ready to state the main results of the paper. For simplicity we present and prove the results in the whole line setting. The same results hold for the symmetric periodic setting as studied in \cite{kw} and the proofs are similar, except for some minor modifications. 

Let $h(\alpha,0)=\alpha$ for $\alpha\in\mathbb R$; let the initial interface $Z(\cdot,0):=Z(0)$, the initial velocity $Z_t(\cdot, 0):=Z_t(0)$  be given 
such that $Z(0)$ satisfy \eqref{interface-holo} and 
$Z_t(0)$ satisfy $\bar Z_t(0)=\mathbb H \bar Z_t(0)$; let $A_1$ be given by \eqref{a1}, the initial acceleration $Z_{tt}(0)$ satisfy
\eqref{interface-a1}.

\begin{theorem}[A blow-up criteria via $\frak E$]\label{blow-up}
Let $s\ge 4$. Assume $Z_{,\alpha'}(0)\in L^\infty (\mathbb R)$, $Z_t(0)\in H^{s+1/2}(\mathbb R)$ and $Z_{tt}(0)\in H^s(\mathbb R)$. Then there is $T>0$, such that on $[0, T]$, the initial value problem of \eqref{interface-r}-\eqref{interface-holo} has a unique solution 
$Z=Z(\cdot, t)$,  satisfying
 $(Z_{tt}, Z_t)\in C^
l([0, T], H^{s-l}(\mathbb R)\times H^{s+1/2-l}(\mathbb R))$ for $l=0,1$, and $Z_{,\alpha'}-1\in C([0, T], H^s(\mathbb R))$.

Moreover if $T^*$ is the supremum over all such times $T$, then either $T^*=\infty$, or $T^*<\infty$, but
\begin{equation}\label{eq:30}
\sup_{[0, T^*)}\frak E(t)=\infty 
\end{equation}

\end{theorem}

\begin{remark}
1. Assume  $Z_{,\alpha'}(0)\in  L^\infty(\mathbb R)$. We note that by the definition $\mathcal A:=\frac{A_1}{|Z_{,\alpha'}|^2}$, $a_0=\frac{A_1(\cdot,0)}{|Z_{,\alpha'}(\cdot, 0)|^2}\ge c_0>0$ for some constant $c_0>0$. So the first part of Theorem~\ref{blow-up} is the local wellposedness in Sobolev spaces as stated in Proposition~\ref{prop:local-s}. The novelty of Theorem~\ref{blow-up} is the new blow up
criteria via the energy functional $\frak E$. 

2. Notice that $\sup_{[0, T^*)}\frak E(t)<\infty$ if and only if $\sup_{[0, T^*)} \mathcal E(t)<\infty$, by Proposition~\ref{prop:energy-eq}.

\end{remark}

  By the discussion of \S\ref{general-soln},  a solution of \eqref{interface-r}-\eqref{interface-holo} is a solution of the water wave equation \eqref{euler} if and only if $\Sigma(t)=\{Z=Z(\alpha',t) \ |\ \alpha'\in\mathbb R\}$ is Jordan. So we can modify the statement of Theorem~\ref{blow-up} to give a blow-up criteria for the water wave equation \eqref{euler}. For the first half of the statements in Corollary~\ref{blow-up1}, see Theorem 6.1 of \cite{wu1}.

\begin{corollary}[A blow-up criteria via $\frak E$]\label{blow-up1}
Let $s\ge 4$. Assume in addition $Z=Z(\cdot,0)$ is non-self-intersecting. Then there is $T>0$, such that on $[0, T]$, the initial value problem of \eqref{euler} has a unique solution, with the properties that the interface
$Z=Z(\cdot, t)$ is nonself-intersecting and 
 $(Z_{tt}, Z_t)\in C^
l([0, T], H^{s-l}(\mathbb R)\times H^{s+1/2-l}(\mathbb R))$ for $l=0,1$, and $Z_{,\alpha'}-1\in C([0, T], H^s(\mathbb R))$.

Moreover if $T^*$ is the supremum over all such times $T$, then either $T^*=\infty$, or $T^*<\infty$, but
\begin{equation}\label{eq:30'}
\sup_{[0, T^*)}\frak E(t)=\infty, \qquad\text{or }\quad Z=Z(\cdot, t) \ \text{becomes self-intersecting at } t=T^*
\end{equation}

\end{corollary}

\subsection{The initial data}\label{id}\footnote{We only need to assume that $F(\cdot, 0), \Psi(\cdot,0)$ are holomorphic on $P_-$ and continuous on $\bar P_-$, satisfying $\lim_{z'\to\infty} \Psi_{z'}(z',0)=1$, $\Psi_{z'}(z',0)\ne 0$ on $P_-$, \eqref{domain-energy} at $t=0$ and \eqref{iid}. We give the initial data as is to put it in the context of the water waves \eqref{euler}.}
Let $\Omega(0)$ be the initial fluid domain, with the interface $\Sigma(0):=\partial\Omega(0)$ being a Jordan curve that tends to horizontal lines at infinity, and let $\Phi(\cdot, 0):\Omega(0)\to P_-$ 
be the Riemann Mapping such that $\lim_{z\to\infty} \Phi_z(z,0)=1$. We know $\Phi(\cdot,0):\overline{\Omega(0)}\to \bar P_-$ is a homeomorphism. Let $\Psi(\cdot, 0):=\Phi^{-1}(\cdot,0)$, and $Z(\alpha',0):=\Psi(\alpha', 0)$, so $Z=Z(\alpha',0):\mathbb R\to\Sigma(0)$  is the parametrization of $\Sigma(0)$ in the Riemann Mapping variable. Let $\bold v(\cdot, 0):\Omega(0)\to \mathbb C$ be the initial velocity field, and $F(z', 0)=\bar{\bold v}(\Psi(z', 0), 0)$. Assume $\bar{\bold v}(\cdot, 0)$ is holomorphic on $\Omega(0)$, so $F(\cdot, 0)$ is holomorphic on $P_-$.  Assume $F(\cdot,0)$, $\Psi(\cdot, 0)$ satisfy \eqref{domain-energy} at $t=0$.
In addition, assume \footnote{Let $Z_{tt}(0)$ be given by \eqref{interface-a1}. Under the assumption \eqref{domain-energy} at $t=0$, this is equivalent to assuming $\|Z_t(0)\|_{L^2}+\|Z_{tt}(0)\|_{L^2}<\infty$.} 
\begin{equation}\label{iid}
c_0:=\sup_{y'<0}\|F(x'+iy', 0)\|_{L^2(\mathbb R, dx')}+\sup_{y'<0}\|\frac1{\Psi_{z'}(x'+iy',0)}-1\|_{L^2(\mathbb R,dx')}<\infty.
\end{equation}

\begin{theorem}[Local existence in the $\frak E<\infty$ regime]\label{th:local}
1. There exists $T_0>0$, depending only on $\mathcal E_1(0)$, such that on $[0,T_0]$, the initial value problem of the water wave equation \eqref{euler} has a generalized solution $(F, \Psi, \frak P)$ in the sense of \eqref{eq:273}-\eqref{eq:274}, with the properties that $F(\cdot, t),\Psi(\cdot, t)$ are holomorphic on $P_-$ for each fixed $t\in [0, T_0]$, $F, \Psi, \frac1{\Psi_{z'}}, \frak P$ are continuous on $\bar P_-\times [0, T_0]$,   $F, \Psi$ are continuous differentiable on $P_-\times [0, T_0]$ and $\frak P$ is continuous differentiable with respect to the spatial variables on $P_-\times [0, T_0]$; during this time,  
$\mathcal E_1(t)<\infty$ and 
\begin{equation}\label{iidt}
\sup_{y'<0}\|F(x'+iy', t)\|_{L^2(\mathbb R, dx')}+\sup_{y'<0}\|\frac1{\Psi_{z'}(x'+iy',t)}-1\|_{L^2(\mathbb R,dx')}<\infty.
\end{equation}
The generalized solution gives rise to a solution $(\bar{\bold v}, P)=(F\circ \Psi^{-1}, \frak P\circ \Psi^{-1})$ of the water wave equation \eqref{euler} so long as $\Sigma(t)=\{Z=\Psi(\alpha',t)\ | \ \alpha'\in \mathbb R\}$ is a Jordan curve.

2. If in addition, the initial interface is chord-arc, that is, $Z_{,\alpha'}(\cdot,0)\in L^1_{loc}(\mathbb R)$ and there is $0<\delta<1$, such that
$$\delta \int_{\alpha'}^{\beta'} |Z_{,\alpha'}(\gamma,0)|\,d\gamma\le |Z(\alpha', 0)-Z(\beta', 0)|\le \int_{\alpha'}^{\beta'} |Z_{,\alpha'}(\gamma,0)|\,d\gamma,\quad \forall -\infty<\alpha'<  \beta'<\infty.$$
Then there is $T_0>0, T_1>0$, $T_0, T_1$ depend only on $\mathcal E_1(0)$, such that on $[0,  \min\{T_0, \frac{\delta}{T_1}\}]$, the initial value problem of the water wave equation \eqref{euler} has a solution, satisfying  $\mathcal E_1(t)<\infty$ and \eqref{iidt}, and the interface $Z=Z(\cdot, t)$ is chord-arc.
\end{theorem}

\section{The proof of Theorem ~\ref{blow-up}}\label{proof1}

We only need to prove the second part, the blow-up criteria of Theorem~\ref{blow-up}. We assume $T^*<\infty$,  for otherwise we are done. 

Let $Z=Z(\cdot,t)$, $t\in [0, T^*)$ be a solution of \eqref{interface-r}-\eqref{interface-holo}:
\begin{equation}\label{interface-e-1}
Z_{tt}+i=i\mathcal A Z_{,\alpha'},
\end{equation}
with constraint 
\begin{equation}\label{interface-e-2}
\begin{cases}
\bar Z_t=\mathbb H \bar Z_t,\\
Z_{,\alpha'}-1=\mathbb H(Z_{,\alpha'}-1),\qquad \frac1{Z_{,\alpha'}}-1=\mathbb H(\frac1{Z_{,\alpha'}}-1);
\end{cases}
\end{equation}
satisfying  $(Z_{tt}, Z_t)\in C^
l([0, T^*), H^{s-l}(\mathbb R)\times H^{s+1/2-l}(\mathbb R))$ for $l=0,1$, and $Z_{,\alpha'}-1\in C([0, T^*), H^s(\mathbb R))$.  
Precompose \eqref{interface-e-1} with $h$ gives
\begin{equation}\label{interface-e2}
z_{tt}+i=i\frak a z_{\alpha}
\end{equation}
where $\frak a h_\alpha:=\mathcal A\circ h$. 
Differentiating  \eqref{interface-e2}  with respect to $t$ yields
\begin{equation}\label{quasi-l}
{\bar z}_{ttt}+i\frak a {\bar z}_{t\alpha}=-i\frak a_t {\bar z}_{\alpha}=\frac{\frak a_t}{\frak a} ({\bar z}_{tt}-i)\\
\end{equation}
Precompose \eqref{quasi-l} with $h^{-1}$. This gives the corresponding equation in the Riemann mapping variable:
\begin{equation}\label{quasi-r}
{\bar Z}_{ttt}+i\mathcal A {\bar Z}_{t,\alpha'}=\frac{\frak a_t}{\frak a}\circ h^{-1} ({\bar Z}_{tt}-i)\\
\end{equation}
We know ${\bar Z}_{ttt}=(\partial_t+b\partial_{\alpha'})^2\bar Z_t$ and ${\bar Z}_{tt}=(\partial_t+b\partial_{\alpha'})\bar Z_t$, where $b:=h_t\circ h^{-1}$. The analysis in Appendix~\ref{basic-iden} shows that $b$ and $A_1:=\mathcal A|Z_{,\alpha'}|^2$ are as given in \eqref{b}, \eqref{a1}, and 
\begin{equation}\label{at}
\frac{\frak a_t}{\frak a}\circ h^{-1}=\frac{-\Im( 2[Z_t,\mathbb H]{\bar Z}_{tt,\alpha'}+2[Z_{tt},\mathbb H]\partial_{\alpha'} \bar Z_t-
[Z_t, Z_t; D_{\alpha'} \bar Z_t])}{A_1}.
\end{equation}
where 
\begin{equation}\label{zzz}
[Z_t, Z_t; D_{\alpha'} \bar Z_t]:=\frac1{\pi i}\int\frac{(Z_t(\alpha',t)-Z_t(\beta',t))^2}{(\alpha'-\beta')^2} D_{\beta'} \bar Z_t(\beta',t)\,d\beta'.
\end{equation}
\eqref{quasi-l}-\eqref{interface-e-2} or equivalently \eqref{quasi-r}-\eqref{interface-e-2} with $b$, $A_1$ and $ \frac{\frak a_t}{\frak a}\circ h^{-1}$ given by 
\eqref{b}, \eqref{a1} and \eqref{at}
 is a quasilinear equation of the hyperbolic type in the regime of smooth interfaces, with the right hand side consisting of lower order terms. \footnote{\eqref{quasi-r} is equivalent to the quasi-linear system (4.6)-(4.7) of \cite{wu1}. The only difference is that \eqref{quasi-r} is in terms of $Z_t$ and $Z_{tt}$ and (4.6)-(4.7) of \cite{wu1} is in terms of the real components $\Re Z_t$ and $\Re Z_{tt}$.} 
However in the regime that includes interfaces with angled crests, since $\mathcal A$ and $-\frac{\partial P}{\partial\bold n}$ equal to zero at the crests where the interior angles are  $<\pi$,  the left hand side of \eqref{quasi-r} (or \eqref{quasi-l}) is degenerate hyperbolic.

We have the following basic energy inequality.

\begin{lemma}[Basic energy inequality]\label{basic-e}
Assume $\theta=\theta(\alpha,t)$, $\alpha\in \mathbb R$, $t\in [0, T)$ is smooth, decays fast at the spatial infinity and satisfies $(I-\mathbb H)(\theta\circ h^{-1})=0$ and 
\begin{equation}\label{eq:40}
\partial_t^2\theta+i\frak a\partial_\alpha \theta=G_\theta.
\end{equation}
Let
\begin{equation}\label{eq:41}
E_\theta(t):=\int\frac1{\frak a}|\theta_t|^2\,d\alpha+ i\int\partial_\alpha\theta \bar\theta\,d\alpha+\int\frac1{\frak a}|\theta|^2\,d\alpha
\end{equation}
Then
\begin{equation}\label{eq:42}
\frac d{dt} E_\theta(t)\le \big(\nm{\frac{\frak a_t}{\frak a}}_{L^\infty}+1\big)E_\theta(t)+2 E_\theta(t)^{1/2}\paren{\int\frac{|G_\theta|^2}{\frak a}\,d\alpha}^{1/2}.
\end{equation}
\end{lemma}
\begin{remark} 
Since  $\mathcal A\circ h:=\frak a h_\alpha$, 
 upon changing to the Riemann mapping variable, 
$$E_\theta(t)= \int \frac1{\mathcal A}( |\theta_t\circ h^{-1}|^2+|\theta\circ h^{-1}|^2)\,d\alpha'+ i\int\partial_{\alpha'}(\theta\circ h^{-1}) \bar\theta\circ h^{-1}\,d\alpha'  $$
By $\theta\circ h^{-1}=\mathbb H(\theta\circ h^{-1})$ and \eqref{def-hhalf},
$$i\int\partial_\alpha\theta \bar\theta\,d\alpha=i\int\partial_{\alpha'}(\theta\circ h^{-1}) \bar\theta\circ h^{-1}\,d\alpha'=\|\theta\circ h^{-1}\|_{\dot{H}^{1/2}}^2\ge 0.$$
\end{remark}
\begin{proof}
We have \footnote{Some variants of the proof have been given in \cite{wu3} and \cite{kw}. We prove \eqref{eq:42} nevertheless. }
\begin{equation}\label{eq:43}
\begin{aligned}
\frac d{dt} E_\theta(t)&=2\Re\int\frac1{\frak a}\theta_{tt}\bar \theta_t\,d\alpha-\int\frac{\frak a_t}{\frak a^2}|\theta_t|^2\,d\alpha+i\int\partial_\alpha\theta_t \bar\theta\,d\alpha+i\int\partial_\alpha\theta \bar\theta_t\,d\alpha
\\&\qquad\qquad\qquad\qquad+2\Re\int\frac1{\frak a}\theta_{t}\bar \theta\,d\alpha-\int\frac{\frak a_t}{\frak a^2}|\theta|^2\,d\alpha
\\&= 2\Re \int \frac1{\frak a} (\theta_{tt}+i\frak a\partial_\alpha\theta)\bar\theta_t\,d\alpha-\int\frac{\frak a_t}{\frak a^2}(|\theta_t|^2+|\theta|^2)\,d\alpha+2\Re\int\frac1{\frak a}\theta_{t}\bar \theta\,d\alpha
\end{aligned}
\end{equation}
Here in the second step we used integration by parts on the third term.  \eqref{eq:40}, Cauchy-Schwarz and the fact that $i\int\partial_\alpha\theta \bar\theta\,d\alpha\ge 0$ gives \eqref{eq:42}.

\end{proof}

Apply $D_\alpha (\frac{\partial_\alpha}{h_\alpha})^{k-1}$, $k=2,3$ to  \eqref{quasi-l}, then commute
$D_\alpha (\frac{\partial_\alpha}{h_\alpha})^{k-1}$ with $\partial_t^2+i\frak a\partial_\alpha$ yields
\begin{equation}\label{eq:44}
(\partial_t^2+i\frak a\partial_\alpha)D_\alpha(\frac{\partial_\alpha}{h_\alpha})^{k-1}\bar z_t=D_\alpha(\frac{\partial_\alpha}{h_\alpha})^{k-1}(-i\frak a_t\bar z_\alpha)+ [\partial_t^2+i\frak a\partial_\alpha, D_\alpha(\frac{\partial_\alpha}{h_\alpha})^{k-1}]\bar z_t
\end{equation}
Let
\begin{equation}
E_{k}(t):=E_{D_\alpha(\frac{\partial_\alpha}{h_\alpha})^{k-1}\bar z_t}(t).
\end{equation}
Because $\mathcal A=\frac{A_1}{|Z_{,\alpha'}|^2}$ and $U_h^{-1}D_\alpha U_h=D_{\alpha'}=\frac1{Z_{,\alpha'}}{\partial_{\alpha'}}$,  
\begin{equation}
E_k(t)=\int \frac1{A_1}(|\partial_{\alpha'}^k\bar Z_t|^2+|Z_{,\alpha'}U_h^{-1}\partial_t U_h \frac1{Z_{,\alpha'}}\partial_{\alpha'}^k\bar Z_t|^2)\,d\alpha'+\nm{\frac1{Z_{,\alpha'}}\partial_{\alpha'}^k\bar Z_t}_{\dot H^{1/2}}^2
\end{equation}
We prove Theorem~\ref{blow-up} via the following two Propositions.

\begin{proposition}\label{step1}
There exists a polynomial $p_1=p_1(x)$ with universal coefficients such that 
\begin{equation}
\frac d{dt}E_2(t)\le p_1(\frak E(t))E_2(t).
\end{equation}
\end{proposition}

\begin{proposition}\label{step2}
There exist polynomials $p_2=p_2(x,y)$ and $p_3=p_3(x,y)$ with universal coefficients such that 
\begin{equation}
\frac d{dt}E_3(t)\le p_2(\frak E(t), E_2(t))E_3(t)+p_3(\frak E(t), E_2(t)).
\end{equation}
\end{proposition}

Propositions~\ref{step1} and ~\ref{step2} give that
\begin{equation}\label{step1-2}
\begin{aligned}
E_2(t)&\le E_2(0)e^{\int_0^t p_1(\frak E(s))\,ds};\qquad\text{and }\\
E_3(t)&\le (E_3(0)+\int_0^t p_3(\frak E(s), E_2(s))\,ds)e^{\int_0^t p_2(\frak E(s), E_2(s))\,ds},
\end{aligned}
\end{equation}
so for $T^*<\infty$, $E_2(0)+E_3(0)<\infty$ and  $\sup_{[0, T^*)}\frak E(t)<\infty$ implies $\sup_{[0, T^*)}(E_2(t)+E_3(t))<\infty$. In \S\ref{proof-prop1} and \S\ref{proof-prop2} we will prove 
Propositions~\ref{step1} and ~\ref{step2}. We will complete the proof of Theorem~\ref{blow-up} in \S\ref{complete1} by 
showing that $\sup_{[0, T^*)} (\|Z_t(t)\|_{H^{3+1/2}}+\|Z_{tt}(t)\|_{H^3})$ is controlled by $\sup_{[0, T^*)}(E_2(t)+E_3(t))$ and the initial data.

\subsection{ The proof of Proposition~\ref{step1}}\label{proof-prop1}
\begin{proof}
We prove Proposition~\ref{step1} by applying the basic energy inequality, Lemma~\ref{basic-e} to $D_\alpha(\frac{\partial_\alpha}{h_\alpha})\bar z_t$ of \eqref{eq:44}, notice that $(I-\mathbb H)(U_{h}^{-1}D_\alpha(\frac{\partial_\alpha}{h_\alpha})\bar z_t)=(I-\mathbb H)D_{\alpha'}\bar Z_{t,\alpha'}=0$. 
Using \eqref{eq:c12} \eqref{eq:c5} and \eqref{eq:c10}, we expand the right hand side of \eqref{eq:44}:
\begin{equation}\label{eq:45}
\begin{aligned}
G_2&:=D_\alpha\frac{\partial_\alpha}{h_\alpha}(-i\frak a_t\bar z_\alpha)+[\partial_t^2+i\frak a\partial_\alpha, D_\alpha \frac{\partial_\alpha}{h_\alpha}] \bar z_t \\&=
D_\alpha\frac{\partial_\alpha}{h_\alpha}(-i\frak a_t\bar z_\alpha)-2(D_\alpha z_{tt}D_\alpha\frac{\partial_\alpha}{h_\alpha}\bar z_t+D_\alpha z_t\partial_t D_\alpha\frac{\partial_\alpha}{h_\alpha}\bar z_t)\\&-
D_\alpha\partial_tU_h\{(h_t\circ h^{-1})_{\alpha'}\bar Z_{t,\alpha'}\}-D_\alpha U_h \{(h_t\circ h^{-1})_{\alpha'}\bar Z_{tt,\alpha'}\}-iD_\alpha U_h\{\mathcal A_{\alpha'} \bar Z_{t,\alpha'}\}
\end{aligned}
\end{equation}
 We can control $\nm {\frac {\frak a_t}{ \frak a}}_{L^\infty}$ by a polynomial of $\frak E$, see Appendix~\ref{quantities}. What remains to be shown is that
\begin{equation}\label{eq:46}
\int\frac{|G_2|^2}{\frak a}\,d\alpha\le C(\frak E) E_2,
\end{equation}
for some polynomial $C(\frak E)$. Changing to the Riemann mapping variables and using $\mathcal A=\frac{A_1}{|Z_{,\alpha'}|^2}$, $A_1\ge 1$,
\begin{equation}\label{eq:54}
\int\frac{|G_2|^2}{\frak a}\,d\alpha=\int\frac{|G_2|^2}{\frak ah_\alpha}h_\alpha\,d\alpha =\int\frac{|Z_{,\alpha'}U_h^{-1}G_2|^2}{A_1}\,d\alpha'\le \int |Z_{,\alpha'}U_h^{-1}G_2|^2\,d\alpha'.
\end{equation}
So it suffices to show that 
$$\int |Z_{,\alpha'}U_h^{-1}G_2|^2\,d\alpha'\le C(\frak E) E_2.$$

Let
\begin{align}\label{eq:55}
G_{2,0}:&= D_\alpha\frac{\partial_\alpha}{h_\alpha}(-i\frak a_t\bar z_\alpha);\\ \label{eq:56}
G_{2,1}:&=-2(D_\alpha z_{tt}D_\alpha\frac{\partial_\alpha}{h_\alpha}\bar z_t+D_\alpha z_t\partial_t D_\alpha\frac{\partial_\alpha}{h_\alpha}\bar z_t);\qquad\qquad\text{and}
\end{align}
\begin{equation}\label{eq:57}
\begin{aligned} 
G_{2,2}:=-&
D_\alpha\partial_tU_h\{(h_t\circ h^{-1})_{\alpha'}\bar Z_{t,\alpha'}\}-D_\alpha U_h \{(h_t\circ h^{-1})_{\alpha'}\bar Z_{tt,\alpha'}\}\\ &-iD_\alpha U_h\{\mathcal A_{\alpha'} \bar Z_{t,\alpha'}\},
\end{aligned}
\end{equation}
so $G_2=G_{2,0}+G_{2,1}+G_{2,2}$. We know by $\bar z_{tt}-i=-i\frak a \bar z_\alpha$ \eqref{interface-e2} and $U_h^{-1}D_\alpha U_h=D_{\alpha'}:=\frac{\partial_{\alpha'}}{Z_{,\alpha'}}$, 
\begin{align}\label{eq:551}
Z_{,\alpha'}U_h^{-1}G_{2,0}&= \partial_{\alpha'}^2(\frac{\frak a_t}{\frak a}\circ h^{-1} (\bar Z_{tt}-i));\\ \label{eq:561}
Z_{,\alpha'}U_h^{-1}G_{2,1}&=-2(D_{\alpha'} Z_{tt}\partial_{\alpha'}^2\bar Z_t+D_{\alpha'} Z_t \big(Z_{,\alpha'}U_h^{-1}\partial_t U_h \frac1{Z_{,\alpha'}}\partial_{\alpha'}^2\bar Z_t)\big);
\end{align}
\begin{equation}\label{eq:571}
\begin{aligned} 
Z_{,\alpha'}U_h^{-1}G_{2,2}=-&
\partial_{\alpha'}U_h^{-1}\partial_tU_h\{(h_t\circ h^{-1})_{\alpha'}\bar Z_{t,\alpha'}\}-\partial_{\alpha'}  \{(h_t\circ h^{-1})_{\alpha'}\bar Z_{tt,\alpha'}\}\\ &-i\partial_{\alpha'} \{\mathcal A_{\alpha'} \bar Z_{t,\alpha'}\}.
\end{aligned}
\end{equation}

\subsubsection*{Step 1: Quantities controlled by $E_2$ and a polynomial of $\frak E$.}

By the definition of $E_2$, and the fact that $\|A_1\|_{L^\infty}\le C(\frak E)$ (cf. Appendix~\ref{quantities}), we know
\begin{align}\label{eq:47}
\int\frac{|D_\alpha\frac{\partial_\alpha}{h_\alpha}\bar z_t|^2 }{\frak a}\,d\alpha, \quad 
\int\frac{|\partial_tD_\alpha\frac{\partial_\alpha}{h_\alpha}\bar z_t|^2 }{\frak a}\,d\alpha &\le E_2\\ \label{eq:48}
\|\partial_{\alpha'}^2\bar Z_t\|_{L^2}^2,\quad   \nm{Z_{,\alpha'}U_h^{-1}\partial_t U_h \frac1{Z_{,\alpha'}}\partial_{\alpha'}^2\bar Z_t}_{L^2}^2,\quad \nm{\frac1{Z_{,\alpha'}}\partial_{\alpha'}^2\bar Z_t}_{\dot H^{1/2}}^2& \le C(\frak E) E_2.
\end{align}
We commute $Z_{,\alpha'}$ with $U_h^{-1}\partial_t U_h$ in the second quantity of \eqref{eq:48}
\begin{equation}
 Z_{,\alpha'}U_h^{-1}\partial_t U_h \frac1{Z_{,\alpha'}}\partial_{\alpha'}^2\bar Z_t= U_h^{-1}\partial_t U_h\partial_{\alpha'}^2\bar Z_t+[Z_{,\alpha'}, U_h^{-1}\partial_t U_h] \frac1{Z_{,\alpha'}}\partial_{\alpha'}^2\bar Z_t
\end{equation}
By \eqref{eq:c13} and Appendix~\ref{quantities}, 
\begin{equation}\label{eq:49}
\abs{\nm{U_h^{-1}\partial_t U_h\partial_{\alpha'}^2\bar Z_t}_{L^2}-\nm{Z_{,\alpha'}U_h^{-1}\partial_t U_h \frac1{Z_{,\alpha'}}\partial_{\alpha'}^2\bar Z_t}_{L^2}}
\le C(\frak E)\|\partial_{\alpha'}^2\bar Z_t\|_{L^2},
\end{equation}
so
\begin{equation}\label{eq:50}
\nm{U_h^{-1}\partial_t U_h\partial_{\alpha'}^2\bar Z_t}_{L^2}^2\le C(\frak E)E_2
\end{equation}

\subsubsection*{Step 2. Controlling $G_{2,1}$}
By \eqref{eq:561}, Appendix~\ref{quantities} and \eqref{eq:48},
\begin{equation}\label{eq:51}
\int {|Z_{,\alpha'}U_h^{-1}G_{2,1}|^2}\,d\alpha
\le C(\frak E)E_2. 
\end{equation}

\subsubsection*{Step 3. Controlling $G_{2,2}$}

We expand further the terms in $Z_{,\alpha'}U_h^{-1}G_{2,2}$ by the product rule, 
\begin{equation}\label{eq:52}
\begin{aligned}
\partial_{\alpha'}U_h^{-1}\partial_tU_h&\{(h_t\circ h^{-1})_{\alpha'}\bar Z_{t,\alpha'}\}=(h_t\circ h^{-1})_{\alpha'}\partial_{\alpha'}U_h^{-1}\partial_tU_h\bar Z_{t,\alpha'}\\&+\{U_h^{-1}\partial_tU_h(h_t\circ h^{-1})_{\alpha'}\}\partial_{\alpha'}\bar Z_{t,\alpha'}+
\{\partial_{\alpha'} (h_t\circ h^{-1})_{\alpha'}\}U_h^{-1}\partial_tU_h\bar Z_{t,\alpha'}\\&+\{\partial_{\alpha'}U_h^{-1}\partial_tU_h(h_t\circ h^{-1})_{\alpha'}\}\bar Z_{t,\alpha'};
\end{aligned}
\end{equation}
\begin{equation}\label{eq:53}
\begin{aligned}
\partial_{\alpha'} \{(h_t\circ h^{-1})_{\alpha'}\bar Z_{tt,\alpha'}\}&=\{\partial_{\alpha'} (h_t\circ h^{-1})_{\alpha'}\}\bar Z_{tt,\alpha'}+ (h_t\circ h^{-1})_{\alpha'} \partial_{\alpha'} \bar Z_{tt,\alpha'};
\\
\partial_{\alpha'}\{\mathcal A_{\alpha'} \bar Z_{t,\alpha'}\}&=(\partial_{\alpha'}\mathcal A_{\alpha'}) \bar Z_{t,\alpha'}+  \mathcal A_{\alpha'} \partial_{\alpha'}\bar Z_{t,\alpha'}
\end{aligned}
\end{equation}
\subsubsection*{Step 3.1. The  quantity $\partial_{\alpha'}^k(h_t\circ h^{-1})$.} 
By equation \eqref{c4} in Appendix~\ref{basic-iden}, 
\begin{equation}\label{b1}
h_t\circ h^{-1}(\alpha',t)= \frac{Z_t(\alpha', t)}{Z_{,\alpha'}(\alpha', t)} +\Xi(\alpha',t).
\end{equation}
where $(I-\mathbb H) \Xi(\cdot, t)=0$. 
Differentiating with respect to $\alpha'$  yields
\begin{equation}\label{eq:70}
(h_t\circ h^{-1})_{\alpha'}=\frac{Z_{t,\alpha'}}{Z_{,\alpha'}}+Z_{t}\partial_{\alpha'}\frac{1}{Z_{,\alpha'}}+\partial_{\alpha'}\Xi.
\end{equation}
Rewrite $\frac{ Z_{t,\alpha'}}{ Z_{,\alpha'}}=2\Re \frac{ Z_{t,\alpha'}}{ Z_{,\alpha'}}-\frac{\bar Z_{t,\alpha'}}{\bar Z_{,\alpha'}}$ and move $2\Re \frac{ Z_{t,\alpha'}}{ Z_{,\alpha'}}$ to the left, we obtain
\begin{equation}\label{eq:71}
(h_t\circ h^{-1})_{\alpha'}-2\Re \frac{Z_{t,\alpha'}}{Z_{,\alpha'}}=- \frac{\bar Z_{t,\alpha'}}{\bar Z_{,\alpha'}}+Z_{t}\partial_{\alpha'}\frac{1}{Z_{,\alpha'}}+\partial_{\alpha'}\Xi;
\end{equation}
differentiating \eqref{eq:70}  with respect to $\alpha'$ and using the fact $ \frac{\partial_{\alpha'}^2 Z_{t}}{ Z_{,\alpha'}}  =2\Re \frac{\partial_{\alpha'}^2 Z_{t}}{Z_{,\alpha'}}-\frac{\partial_{\alpha'}^2\bar Z_{t}}{\bar Z_{,\alpha'}}$ gives
\begin{equation}\label{eq:72}
\begin{aligned}
\partial_{\alpha'}(h_t\circ h^{-1})_{\alpha'}-2\Re \frac{\partial_{\alpha'}^2Z_{t}}{Z_{,\alpha'}}=2 Z_{t,\alpha'}\partial_{\alpha'}\frac{1}{Z_{,\alpha'}}- \frac{\partial_{\alpha'}^2\bar Z_{t}}{\bar Z_{,\alpha'}}+Z_{t}\partial_{\alpha'}^2\frac{1}{Z_{,\alpha'}}+\partial_{\alpha'}^2\Xi.
\end{aligned}
\end{equation}
Notice that $(I-\mathbb H)\partial_{\alpha'}^k\Xi=0$, $k=1,2$.  Apply $ (I-\mathbb H)$ to both sides of \eqref{eq:71} and \eqref{eq:72}, then take the real parts. Rewrite the last two terms on the right hand sides as commutators via the fact that $(I-\mathbb H)\partial_{\alpha'}^k \bar Z_t=0$ and $(I-\mathbb H)\partial_{\alpha'}^k\frac1{Z_{,\alpha'}}=0$, $k=1,2$.\footnote{If $(I-\mathbb H)g=0$, then $(I-\mathbb H)(fg)=[f,\mathbb H]g$.} We get
\begin{equation}\label{eq:73}
(h_t\circ h^{-1})_{\alpha'}-2\Re \frac{Z_{t,\alpha'}}{Z_{,\alpha'}}=\Re\{- [\frac{1}{\bar Z_{,\alpha'}}, \mathbb H]\bar Z_{t,\alpha'}+[Z_{t},\mathbb H]\partial_{\alpha'}\frac{1}{Z_{,\alpha'}}\}
\end{equation}
and
\begin{equation}\label{eq:74}
\begin{aligned}
\partial_{\alpha'}(h_t\circ h^{-1})_{\alpha'}-2\Re \frac{\partial_{\alpha'}^2Z_{t}}{Z_{,\alpha'}}
&=\Re\{2 (I-\mathbb H)(Z_{t,\alpha'}\partial_{\alpha'}\frac{1}{Z_{,\alpha'}})    
\\&- [\frac{1}{\bar Z_{,\alpha'}}, \mathbb H]\partial_{\alpha'}^2\bar Z_{t}+[Z_{t},\mathbb H]\partial_{\alpha'}^2\frac{1}{Z_{,\alpha'}}
 \}.
\end{aligned}
\end{equation}
From \eqref{eq:74}, by H\"older's inequality, \eqref{3.16} and \eqref{3.21},
\begin{equation}\label{eq:78}
\nm{\partial_{\alpha'}(h_t\circ h^{-1})_{\alpha'}-2\Re \frac{\partial_{\alpha'}^2Z_{t}}{Z_{,\alpha'}}}_{L^2}\lesssim \|Z_{t,\alpha'}\|_{L^\infty}\nm{\partial_{\alpha'}\frac1{Z_{,\alpha'}}}_{L^2}
\end{equation}
\subsubsection*{Step 3.2. The estimates for the  quantities involving $\bar Z_t$}
Commuting $\partial_{\alpha'}$ with $U_h^{-1}\partial_tU_h$ and using \eqref{eq:20} gives
\begin{equation}
\begin{aligned}
\partial_{\alpha'}U_h^{-1}\partial_tU_h\bar Z_{t,\alpha'}&=U_h^{-1}\partial_tU_h\partial_{\alpha'}^2\bar Z_{t}+[\partial_{\alpha'},U_h^{-1}\partial_tU_h]\bar Z_{t,\alpha'}\\&= U_h^{-1}\partial_tU_h\partial_{\alpha'}^2\bar Z_{t}+(h_t\circ h^{-1})_{\alpha'}\partial_{\alpha'}^2\bar Z_{t},
\end{aligned}
\end{equation}
so by \eqref{eq:48}, \eqref{eq:50} and Appendix~\ref{quantities},
\begin{equation}\label{eq:60}
\|\partial_{\alpha'}U_h^{-1}\partial_tU_h\bar Z_{t,\alpha'}\|_{L^2}^2\le C(\frak E)E_2.
\end{equation}

We estimate $\|Z_{t,\alpha}\|_{L^\infty}$ by \eqref{eq:sobolev}, Appendix~\ref{quantities} and \eqref{eq:48},
\begin{equation}\label{eq:61}
\|Z_{t,\alpha'}\|_{L^\infty}^2\le 2\|Z_{t,\alpha'}\|_{L^2}\|\partial_{\alpha'}^2Z_{t}\|_{L^2}\le C(\frak E)E_2^{1/2}.
\end{equation}

We compute $\partial_{\alpha'}^2\bar Z_{tt}$ by  \eqref{eq:c11},
\begin{equation}\label{eq:75}
\begin{aligned}
\partial_{\alpha'}^2\bar Z_{tt}-U_h^{-1}\partial_t U_h\partial_{\alpha'}^2\bar Z_t
&= [\partial_{\alpha'}^2, U_h^{-1}\partial_t U_h]\bar Z_t
\\&=2(h_t\circ h^{-1})_{\alpha'} \partial_{\alpha'}^2\bar Z_t +\partial_{\alpha'}(h_t\circ h^{-1})_{\alpha'} \bar Z_{t,\alpha'},
\end{aligned}
\end{equation}
where by \eqref{eq:78}, \eqref{eq:48}, \eqref{eq:61} and Appendix~\ref{quantities},
\begin{equation}\label{eq:99}
\begin{aligned}
\|\partial_{\alpha'}(h_t\circ h^{-1})_{\alpha'} \bar Z_{t,\alpha'} \|_{L^2}&\lesssim \|D_{\alpha'}Z_t\|_{L^\infty}\|\partial_{\alpha'}^2 \bar Z_t\|_{L^2}+\|Z_{t,\alpha'}\|_{L^\infty}^2\nm{\partial_{\alpha'}\frac1{Z_{,\alpha'}}}_{L^2}
\\&\lesssim C(\frak E)E_2^{1/2}.
\end{aligned}
\end{equation}
Therefore \eqref{eq:75}, \eqref{eq:99}, \eqref{eq:50}, \eqref{eq:48} and Appendix~\ref{quantities} gives that 
\begin{equation}\label{eq:79}
\|\partial_{\alpha'}^2\bar Z_{tt}\|^2_{L^2}\lesssim C(\frak E)E_2.
\end{equation}
As a consequence of  \eqref{eq:sobolev}, \eqref{eq:79} and Appendix~\ref{quantities},
\begin{equation}\label{eq:80}
\|\partial_{\alpha'}\bar Z_{tt}\|^2_{L^\infty}\le 2\|\partial_{\alpha'}\bar Z_{tt}\|_{L^2}\|\partial_{\alpha'}^2\bar Z_{tt}\|_{L^2}    \lesssim C(\frak E)E_2^{1/2}.
\end{equation}

We compute $U_h^{-1}\partial_tU_h\bar Z_{t,\alpha'}$ by commuting $U_h^{-1}\partial_tU_h$ with $\partial_{\alpha'}$ and using \eqref{eq:20},
\begin{equation}\label{eq:83}
U_h^{-1}\partial_tU_h\bar Z_{t,\alpha'}=\partial_{\alpha'}\bar Z_{tt}+[U_h^{-1}\partial_tU_h, \partial_{\alpha'}]\bar Z_t=\bar Z_{tt,\alpha'}-(h_t\circ h^{-1})_{\alpha'} \bar Z_{t,\alpha'};
\end{equation}
\eqref{eq:80}, \eqref{eq:61} and Appendix~\ref{quantities} imply that
\begin{equation}\label{eq:81}
\|U_h^{-1}\partial_tU_h\bar Z_{t,\alpha'}\|_{L^\infty}^2\lesssim C(\frak E)E_2^{1/2}.
\end{equation}
\subsubsection*{Step 3.3. The estimate for the terms involving $\partial_{\alpha'}^k(h_t\circ h^{-1})$ in \eqref{eq:52} and \eqref{eq:53}}
By Steps 3.1 and 3.2, we can give the estimates for some of the terms in \eqref{eq:52} and \eqref{eq:53}. First, 
because $\|(h_t\circ h^{-1})_{\alpha'}\|_{L^\infty}\le C(\frak E)$ (cf. Appendix~\ref{quantities}) and \eqref{eq:60}, 
\begin{equation}\label{eq:86}
\|(h_t\circ h^{-1})_{\alpha'}\partial_{\alpha'}U_h^{-1}\partial_tU_h\bar Z_{t,\alpha'}\|_{L^2}^2\le C(\frak E)E_2;
\end{equation}
and from \eqref{eq:79},
\begin{equation}\label{eq:85}
\|(h_t\circ h^{-1})_{\alpha'}\partial_{\alpha'}\bar Z_{tt,\alpha'}\|_{L^2}^2\le C(\frak E)E_2.
\end{equation}
From \eqref{eq:78}, \eqref{eq:48}, \eqref{eq:61}, \eqref{eq:80} and Appendix~\ref{quantities},
\begin{equation}\label{eq:82}
\begin{aligned}
\|\partial_{\alpha'}(h_t\circ h^{-1})_{\alpha'} \bar Z_{tt,\alpha'} \|_{L^2}&\lesssim \|D_{\alpha'}Z_{tt}\|_{L^\infty}\|\partial_{\alpha'}^2 \bar Z_t\|_{L^2}\\&+\|Z_{t,\alpha'}\|_{L^\infty}\|Z_{tt,\alpha'}\|_{L^\infty}\nm{\partial_{\alpha'}\frac1{Z_{,\alpha'}}}_{L^2}
\lesssim C(\frak E)E_2^{1/2};
\end{aligned}
\end{equation}
additionally from \eqref{eq:83},
\begin{equation}\label{eq:84}
\begin{aligned}
\|\partial_{\alpha'}&(h_t\circ h^{-1})_{\alpha'}U_h^{-1}\partial_t U_h \bar Z_{t,\alpha'} \|_{L^2}\lesssim( \|D_{\alpha'}Z_{tt}\|_{L^\infty}+C(\frak E) \|D_{\alpha'}Z_{t}\|_{L^\infty})  \|\partial_{\alpha'}^2 \bar Z_t\|_{L^2}\\&+(\|Z_{tt,\alpha'}\|_{L^\infty}+C(\frak E)\|Z_{t,\alpha'}\|_{L^\infty})\|Z_{t,\alpha'}\|_{L^\infty}\nm{\partial_{\alpha'}\frac1{Z_{,\alpha'}}}_{L^2}
\lesssim C(\frak E)E_2^{1/2}.
\end{aligned}
\end{equation}
\subsubsection*{Step 3.4. The terms involving $\partial_{\alpha'}^kU_h^{-1}\partial_t U_h(h_t\circ h^{-1})_{\alpha'}$, $k=0,1$}
We first consider $U_h^{-1}\partial_t U_h(h_t\circ h^{-1})_{\alpha'}$. Applying $U_h^{-1}\partial_t U_h$  to \eqref{eq:73} gives
\begin{equation}\label{eq:88}
\begin{aligned}
U_h^{-1}\partial_t U_h(h_t\circ h^{-1})_{\alpha'}&=2\Re U_h^{-1}\partial_t U_h\frac{Z_{t,\alpha'}}{Z_{,\alpha'}}\\&+\Re\{- U_h^{-1}\partial_t U_h[\frac{1}{\bar Z_{,\alpha'}}, \mathbb H]\bar Z_{t,\alpha'}+U_h^{-1}\partial_t U_h[Z_{t},\mathbb H]\partial_{\alpha'}\frac{1}{Z_{,\alpha'}}\};
\end{aligned}
\end{equation}
we know
\begin{equation}\label{eq:91}
U_h^{-1}\partial_t U_h\frac{Z_{t,\alpha'}}{Z_{,\alpha'}}=U_h^{-1}\partial_t \frac{z_{t\alpha}}{z_{\alpha}}=
 \frac{Z_{tt,\alpha'}}{Z_{,\alpha'}}- \big(\frac{Z_{t,\alpha'}}{Z_{,\alpha'}}\big)^2=D_{\alpha'}Z_{tt}-(D_{\alpha'}Z_{t})^2.
\end{equation}
We compute the last two terms on the RHS of \eqref{eq:88} by \eqref{eq:c14}, 
\begin{equation}\label{eq:87}
\begin{aligned}
U_h^{-1}\partial_t U_h[\frac{1}{\bar Z_{,\alpha'}}, \mathbb H]\bar Z_{t,\alpha'}&=[U_h^{-1}\partial_t U_h\frac{1}{\bar Z_{,\alpha'}}, \mathbb H]\bar Z_{t,\alpha'}\\&+[\frac{1}{\bar Z_{,\alpha'}}, \mathbb H](\partial_{\alpha'}\bar Z_{tt})-[\frac{1}{\bar Z_{,\alpha'}}, h_t\circ h^{-1}; \bar Z_{t,\alpha'}];
\end{aligned}
\end{equation}
and
\begin{equation}\label{eq:89}
\begin{aligned}
U_h^{-1}\partial_t U_h[ Z_{t}, \mathbb H]\partial_{\alpha'}\frac{1}{ Z_{,\alpha'}}&=[Z_{tt}, \mathbb H]\partial_{\alpha'}\frac{1}{ Z_{,\alpha'}}\\&+[Z_t, \mathbb H](\partial_{\alpha'}U_h^{-1}\partial_t U_h \frac{1}{ Z_{,\alpha'}}  )-[Z_t, h_t\circ h^{-1};  \partial_{\alpha'}  \frac{1}{Z_{,\alpha'}}].
\end{aligned}
\end{equation}
Now by the product rule,
\begin{equation}\label{eq:90}
\begin{aligned}
\partial_{\alpha'}U_h^{-1}\partial_t U_h \frac{1}{ Z_{,\alpha'}}&=\partial_{\alpha'} \{\frac{1}{ Z_{,\alpha'}}((h_t\circ h^{-1})_{\alpha'}-D_{\alpha'}Z_t)\}\\&=\big(\partial_{\alpha'} \frac{1}{ Z_{,\alpha'}}\big)((h_t\circ h^{-1})_{\alpha'}-D_{\alpha'}Z_t)+(D_{\alpha'}(h_t\circ h^{-1})_{\alpha'}-D_{\alpha'}^2Z_t);
\end{aligned}
\end{equation}
commuting $U_h^{-1}\partial_t U_h$ with $\partial_{\alpha'}$ and using    \eqref{eq:20} gives
\begin{equation}\label{eq:98}
\begin{aligned}
U_h^{-1}\partial_t U_h\partial_{\alpha'} \frac{1}{ Z_{,\alpha'}}&=\partial_{\alpha'}U_h^{-1}\partial_t U_h \frac{1}{ Z_{,\alpha'}}+[U_h^{-1}\partial_t U_h, \partial_{\alpha'}] \frac{1}{ Z_{,\alpha'}}\\&
=\partial_{\alpha'}U_h^{-1}\partial_t U_h \frac{1}{ Z_{,\alpha'}}- (h_t\circ h^{-1})_{\alpha'}\partial_{\alpha'}\frac{1}{ Z_{,\alpha'}}.
\end{aligned}
\end{equation}
Applying Appendix~\ref{quantities} yields
\begin{equation}\label{eq:97}
\|\partial_{\alpha'}U_h^{-1}\partial_t U_h \frac{1}{ Z_{,\alpha'}}\|_{L^2}+\| U_h^{-1}\partial_t U_h\partial_{\alpha'} \frac{1}{ Z_{,\alpha'}}\|_{L^2}\le C(\frak E);
\end{equation}
and from \eqref{eq:88}, by \eqref{eq:91}, \eqref{eq:87}, \eqref{eq:89}, \eqref{eq:97} and \eqref{eq:b13}, \eqref{eq:b15} and Appendix~\ref{quantities},
\begin{equation}\label{eq:92}
\|U_h^{-1}\partial_t U_h(h_t\circ h^{-1})_{\alpha'}\|_{L^\infty}\lesssim C(\frak E).
\end{equation}

We analyze $\partial_{\alpha'}U_h^{-1}\partial_t U_h(h_t\circ h^{-1})_{\alpha'}$ similarly. Commuting $\partial_{\alpha'}$ with $U_h^{-1}\partial_t U_h$ and using \eqref{eq:20} gives
\begin{equation}\label{eq:93}
\begin{aligned}
\partial_{\alpha'}U_h^{-1}\partial_t U_h(h_t\circ h^{-1})_{\alpha'}&=[\partial_{\alpha'}, U_h^{-1}\partial_t U_h](h_t\circ h^{-1})_{\alpha'}+U_h^{-1}\partial_t U_h\partial_{\alpha'}(h_t\circ h^{-1})_{\alpha'}
\\&=(h_t\circ h^{-1})_{\alpha'}  \partial_{\alpha'}(h_t\circ h^{-1})_{\alpha'}+U_h^{-1}\partial_t U_h\partial_{\alpha'}(h_t\circ h^{-1})_{\alpha'}.
\end{aligned}
\end{equation}
We compute the second term on the RHS of \eqref{eq:93} via \eqref{eq:74}:
\begin{equation}\label{eq:94}
\begin{aligned}
U_h^{-1}\partial_t U_h\partial_{\alpha'}(h_t\circ h^{-1})_{\alpha'}&-2\Re U_h^{-1}\partial_t U_h\frac{\partial_{\alpha'}^2Z_{t}}{Z_{,\alpha'}}
=\Re\{2 U_h^{-1}\partial_t U_h(I-\mathbb H)(Z_{t,\alpha'}\partial_{\alpha'}\frac{1}{Z_{,\alpha'}})    
\\&-U_h^{-1}\partial_t U_h [\frac{1}{\bar Z_{,\alpha'}}, \mathbb H]\partial_{\alpha'}^2\bar Z_{t}+U_h^{-1}\partial_t U_h[Z_{t},\mathbb H]\partial_{\alpha'}^2\frac{1}{Z_{,\alpha'}}
 \};
\end{aligned}
\end{equation}
commuting $U_h^{-1}\partial_t U_h$ with  $\frac{\partial_{\alpha'}}{Z_{,\alpha'}}:=D_{\alpha'}$ and  using\eqref{eq:c1} gives
\begin{equation}\label{eq:95}
\begin{aligned}
U_h^{-1}\partial_t U_h\frac{\partial_{\alpha'}^2Z_{t}}{Z_{,\alpha'}}&=D_{\alpha'}U_h^{-1}\partial_t U_h Z_{t,\alpha'} +[U_h^{-1}\partial_t U_h,  D_{\alpha'}] Z_{t,\alpha'} \\&=\frac1{Z_{,\alpha'}}\partial_{\alpha'}U_h^{-1}\partial_t U_h Z_{t,\alpha'}- \frac1{Z_{,\alpha'}}(D_{\alpha'}Z_t) (\partial_{\alpha'}^2Z_{t});
\end{aligned}
\end{equation}
for the first term on the RHS of \eqref{eq:94}, we commute $ U_h^{-1}\partial_t U_h$ with $(I-\mathbb H)$ and use \eqref{eq:c21},
\begin{equation}\label{eq:96}
\begin{aligned}
U_h^{-1}&\partial_t U_h(I-\mathbb H)(Z_{t,\alpha'}\partial_{\alpha'}\frac{1}{Z_{,\alpha'}}) =(I-\mathbb H)U_h^{-1}\partial_t U_h(Z_{t,\alpha'}\partial_{\alpha'}\frac{1}{Z_{,\alpha'}})\\&-[U_h^{-1}\partial_t U_h, \mathbb H](Z_{t,\alpha'}\partial_{\alpha'}\frac{1}{Z_{,\alpha'}})\\&=
(I-\mathbb H)U_h^{-1}\partial_t U_h(Z_{t,\alpha'}\partial_{\alpha'}\frac{1}{Z_{,\alpha'}})-[h_t\circ h^{-1},\mathbb H]\partial_{\alpha'}(Z_{t,\alpha'}\partial_{\alpha'}\frac{1}{Z_{,\alpha'}}); 
\end{aligned}
\end{equation}
we use product rule to expand further the terms on the RHS of \eqref{eq:96}. By \eqref{eq:81}, \eqref{eq:97}, \eqref{eq:61}, Appendix~\ref{quantities} and \eqref{3.20},
\begin{equation}
\nm{U_h^{-1}\partial_t U_h(I-\mathbb H)(Z_{t,\alpha'}\partial_{\alpha'}\frac{1}{Z_{,\alpha'}})}_{L^2}^2\le C(\frak E)E_2^{1/2}.
\end{equation}
We use \eqref{eq:c14} to compute the last two terms on the RHS of \eqref{eq:94}, then use \eqref{3.20}, \eqref{3.21} and \eqref{eq:97}, \eqref{eq:81}, \eqref{eq:61}, \eqref{eq:80} and Appendix~\ref{quantities} to do the estimates, we get
\begin{equation}\label{eq:100}
\nm{U_h^{-1}\partial_t U_h [\frac{1}{\bar Z_{,\alpha'}}, \mathbb H]\partial_{\alpha'}^2\bar Z_{t}}_{L^2}^2+\nm{U_h^{-1}\partial_t U_h[Z_{t},\mathbb H]\partial_{\alpha'}^2\frac{1}{Z_{,\alpha'}}        }_{L^2}^2\le C(\frak E)E_2^{1/2},
\end{equation}
therefore
\begin{equation}\label{eq:101}
\nm{U_h^{-1}\partial_t U_h\partial_{\alpha'}(h_t\circ h^{-1})_{\alpha'}-2\Re U_h^{-1}\partial_t U_h\frac{\partial_{\alpha'}^2Z_{t}}{Z_{,\alpha'}}
}_{L^2}^2\le C(\frak E)E_2^{1/2}.
\end{equation}

We now conclude the estimates for the two terms involving $U_h^{-1}\partial_tU_h(h_t\circ h^{-1})_{\alpha'}$ in \eqref{eq:52}.  By \eqref{eq:93}, \eqref{eq:78},  \eqref{eq:61}, \eqref{eq:101} and \eqref{eq:48} and Appendix~\ref{quantities},
\begin{equation}\label{eq:102}
\nm{\{\partial_{\alpha'}U_h^{-1}\partial_tU_h(h_t\circ h^{-1})_{\alpha'}\}\bar Z_{t,\alpha'}}_{L^2}^2\le C(\frak E)E_2;
\end{equation}
by \eqref{eq:92} and \eqref{eq:48},
\begin{equation}\label{eq:110}
\nm{   \{U_h^{-1}\partial_tU_h(h_t\circ h^{-1})_{\alpha'}\}\partial_{\alpha'}\bar Z_{t,\alpha'}     }_{L^2}^2\le C(\frak E)E_2.
\end{equation}
 
 Finally we estimate the $L^2$ norms of the two terms on the RHS of the second equation in \eqref{eq:53}.  
   \subsubsection*{Step 3.5. The $L^2$ norm of $\partial_{\alpha'}(\mathcal A_{\alpha'}\bar Z_{t,\alpha'})$.}
 We begin with the first equation of \eqref{interface-r} $\mathcal A:=\frac{Z_{tt}+i}{iZ_{,\alpha'}}$. Differentiating with respect to $\alpha'$ gives
 \begin{equation}\label{eq:0103}
 \partial_{\alpha'}\mathcal A=-i D_{\alpha'} Z_{tt}-i (Z_{tt}+i)\partial_{\alpha'}\frac{1}{Z_{,\alpha'}}
 \end{equation}
 By Appendix~\ref{quantities},
 \begin{equation}\label{eq:0104}
 \|\partial_{\alpha'}\mathcal A\|_{L^\infty}\le C(\frak E),
 \end{equation}
therefore by \eqref{eq:48},
  \begin{equation}\label{eq:106}
  \|\mathcal A_{\alpha'}\partial_{\alpha'}^2\bar Z_t\|_{L^2}^2
  \le C(\frak E)E_2.
  \end{equation}

  We now consider the term $(\partial_{\alpha'}\mathcal A_{\alpha'})\bar Z_{t,\alpha'}$ in \eqref{eq:53}. We calculate $\partial_{\alpha'}^2\mathcal A$ by differentiating the equation $i\mathcal A=\frac{Z_{tt}+i}{Z_{,\alpha'}}$ \eqref{interface-r} twice:
  \begin{equation}\label{eq:108}
  i\partial_{\alpha'}^2\mathcal A=\frac{\partial_{\alpha'}^2Z_{tt}}{Z_{,\alpha'}}+2\partial_{\alpha'} Z_{tt}\partial_{\alpha'}\frac{1}{Z_{,\alpha'}}+(Z_{tt}+i)\partial_{\alpha'}^2\frac{1}{Z_{,\alpha'}}.
  \end{equation}
  Applying $(I-\mathbb H)$ then taking the imaginary parts gives
    \begin{equation}\label{eq:109}
  \partial_{\alpha'}^2\mathcal A=\Im(I-\mathbb H)(\frac{\partial_{\alpha'}^2Z_{tt}}{Z_{,\alpha'}})+2\Im(I-\mathbb H)(\partial_{\alpha'} Z_{tt}\partial_{\alpha'}\frac{1}{Z_{,\alpha'}})+\Im(I-\mathbb H)((Z_{tt}+i)\partial_{\alpha'}^2\frac{1}{Z_{,\alpha'}});
  \end{equation}
  we rewrite the first term on the right by commuting out $\frac1{Z_{,\alpha'}}$
  \begin{equation}\label{eq:107}
  (I-\mathbb H)(\frac{\partial_{\alpha'}^2Z_{tt}}{Z_{,\alpha'}})=\frac{1}{Z_{,\alpha'}}(I-\mathbb H)({\partial_{\alpha'}^2Z_{tt}})+[\frac{1}{Z_{,\alpha'}}, \mathbb H]({\partial_{\alpha'}^2Z_{tt}});
\end{equation}
  using $(I-\mathbb H) \partial_{\alpha'}^2\frac{1}{Z_{,\alpha'}}=0$  we rewrite the third term on the right of \eqref{eq:109} as a commutator 
  \begin{equation}\label{eq:111}
  (I-\mathbb H)((Z_{tt}+i)\partial_{\alpha'}^2\frac{1}{Z_{,\alpha'}})=[Z_{tt}, \mathbb H]\partial_{\alpha'}^2\frac{1}{Z_{,\alpha'}}
  \end{equation}
  so
  \begin{equation}\label{eq:112}
  \begin{aligned}
   \partial_{\alpha'}^2\mathcal A=\Im\{\frac{1}{Z_{,\alpha'}}&(I-\mathbb H)({\partial_{\alpha'}^2Z_{tt}})+[\frac{1}{Z_{,\alpha'}}, \mathbb H]({\partial_{\alpha'}^2Z_{tt}})\}\\&+\Im\{
   2(I-\mathbb H)(\partial_{\alpha'} Z_{tt}\partial_{\alpha'}\frac{1}{Z_{,\alpha'}})+ [Z_{tt}, \mathbb H]\partial_{\alpha'}^2\frac{1}{Z_{,\alpha'}} \}
     \end{aligned}
  \end{equation}
 We apply \eqref{3.20}, \eqref{3.21} and H\"older. This gives
  \begin{equation}\label{eq:0112}
  \nm{\partial_{\alpha'}^2\mathcal A-\Im\{\frac{1}{Z_{,\alpha'}}(I-\mathbb H)({\partial_{\alpha'}^2Z_{tt}})\}}_{L^2}\lesssim \nm{\partial_{\alpha'}\frac{1}{Z_{,\alpha'}}}_{L^2}\|Z_{tt,\alpha'}\|_{L^\infty},
  \end{equation}
  so
  \begin{equation}\label{eq:114}
  \nm{(\partial_{\alpha'}^2\mathcal A)\bar Z_{t,\alpha'}}_{L^2}\lesssim \|D_{\alpha'}\bar Z_t\|_{L^\infty}\|\partial_{\alpha'}^2Z_{tt}\|_{L^2}+\nm{\partial_{\alpha'}\frac{1}{Z_{,\alpha'}}}_{L^2}\|\bar Z_{t,\alpha'}\|_{L^\infty}\|Z_{tt,\alpha'}\|_{L^\infty}.
  \end{equation}
  By \eqref{eq:48}, \eqref{eq:61}, \eqref{eq:80} and Appendix~\ref{quantities}, 
  $$\nm{(\partial_{\alpha'}^2\mathcal A)\bar Z_{t,\alpha'}}_{L^2}^2\le C(\frak E)E_2.$$
  This completes the proof for 
  \begin{equation}\label{eq:115}
\int {|Z_{,\alpha'}U_h^{-1}G_{2,2}|^2}\,d\alpha
\le C(\frak E)E_2. 
  \end{equation}
  \subsubsection*{Step 4. Controlling  $\|Z_{,\alpha'}U_h^{-1}G_{2,0}\|_{L^2}$.} We are left with controlling $\|Z_{,\alpha'}U_h^{-1}G_{2,0}\|_{L^2}$.
 By  \eqref{eq:551}, we must show 
  \begin{equation}\label{eq:118}
  \int|\partial_{\alpha'}^2(\frac{\frak a_t}{\frak a}\circ h^{-1}(\bar Z_{tt}-i))|^2\,d\alpha'\le C(\frak E) E_2.
  \end{equation}
   We expand $\partial_{\alpha'}^2(\frac{\frak a_t}{\frak a}\circ h^{-1}(\bar Z_{tt}-i))$ by the product rule
   \begin{equation}\label{eq:119}
   \partial_{\alpha'}^2(\frac{\frak a_t}{\frak a}\circ h^{-1}(\bar Z_{tt}-i))=\frac{\frak a_t}{\frak a}\circ h^{-1}\partial_{\alpha'}^2\bar Z_{tt}+2\partial_{\alpha'}(\frac{\frak a_t}{\frak a}\circ h^{-1})\bar Z_{tt,\alpha'}+\partial_{\alpha'}^2(\frac{\frak a_t}{\frak a}\circ h^{-1})(\bar Z_{tt}-i)
   \end{equation}
   where we estimate the $L^2$ norm of $\partial_{\alpha'}(\frac{\frak a_t}{\frak a}\circ h^{-1})$ by \eqref{at}, \eqref{3.16}, \eqref{3.20}, \eqref{3.21} and \eqref{3.17}
   \begin{equation}\label{eq:120}
   \begin{aligned}
   \nm{\partial_{\alpha'}(\frac{\frak a_t}{\frak a}\circ h^{-1})}_{L^2}\lesssim &\|Z_{t,\alpha'}\|_{L^\infty}\|Z_{tt,\alpha'}\|_{L^2}+\|Z_{t,\alpha'}\|_{L^\infty}\|Z_{t,\alpha'}\|_{L^2}\|D_{\alpha'}\bar Z_t\|_{L^\infty}\\&+
   \|Z_{t,\alpha'}\|_{L^\infty}\|Z_{t,\alpha'}\|_{L^2}\nm{\frac{\frak a_t}{\frak a}}_{L^\infty}
      \end{aligned}
   \end{equation}
   so by Appendix~\ref{quantities}, \eqref{eq:48} and \eqref{eq:61}, \eqref{eq:80},
   \begin{equation}\label{eq:121}
   \int |\frac{\frak a_t}{\frak a}\circ h^{-1}\partial_{\alpha'}^2\bar Z_{tt}+2\partial_{\alpha'}(\frac{\frak a_t}{\frak a}\circ h^{-1})\bar Z_{tt,\alpha'}|^2\,d\alpha'\le C(\frak E)E_2.
   \end{equation}
What remains is the term $\int |\partial_{\alpha'}^2(\frac{\frak a_t}{\frak a}\circ h^{-1})(\bar Z_{tt}-i)|^2\,d\alpha'$. We begin with 
     \eqref{eq:44}, together with \eqref{eq:45} and \eqref{eq:55}, \eqref{eq:56}, \eqref{eq:57}:
  \begin{equation}\label{eq:116}
 (\partial_t^2+i\frak a\partial_\alpha)D_\alpha(\frac{\partial_\alpha}{h_\alpha})\bar z_t=D_\alpha\frac{\partial_\alpha}{h_\alpha}(-i\frak a_t\bar z_\alpha)+G_{2,1}+G_{2,2}.
  \end{equation}
 Precomposing with $h^{-1}$ then multiply $Z_{,\alpha'}$ gives, using $\bar z_{tt}-i=-i\frak a \bar z_\alpha$ \eqref{interface-e2},
   \begin{equation}\label{eq:117}
Z_{,\alpha'}U_h^{-1}(\partial_t^2+i\frak a\partial_\alpha)D_\alpha(\frac{\partial_\alpha}{h_\alpha})\bar z_t=\partial_{\alpha'}^2(\frac{\frak a_t}{\frak a}\circ h^{-1}(\bar Z_{tt}-i))+Z_{,\alpha'}U_h^{-1}(G_{2,1}+G_{2,2}).
  \end{equation} 
By commuting  $(\partial_t^2+i\frak a\partial_\alpha)$ with $\frac{h_\alpha}{z_{\alpha}}$ we rewrite the left hand side as
\begin{equation}\label{eq:122}
U_h^{-1}(\partial_t^2+i\frak a\partial_\alpha)(\frac{\partial_\alpha}{h_\alpha})^2\bar z_t +Z_{,\alpha'}U_h^{-1}[(\partial_t^2+i\frak a\partial_\alpha),  \frac{h_\alpha}{z_{\alpha}}  ] (\frac{\partial_\alpha}{h_\alpha})^2\bar z_t;  \end{equation}
\eqref{eq:117} now yields
\begin{equation}\label{eq:125}
 U_h^{-1}(\partial_t^2+i\frak a\partial_\alpha)(\frac{\partial_\alpha}{h_\alpha})^2\bar z_t =\partial_{\alpha'}^2\big(\frac{\frak a_t}{\frak a}\circ h^{-1}\big)(\bar Z_{tt}-i)+e
 \end{equation}
 where
 \begin{equation}
 \begin{aligned}
 e:&=-Z_{,\alpha'}U_h^{-1}[(\partial_t^2+i\frak a\partial_\alpha),  \frac{h_\alpha}{z_{\alpha}}  ] (\frac{\partial_\alpha}{h_\alpha})^2\bar z_t+ Z_{,\alpha'}U_h^{-1}(G_{2,1}+G_{2,2})\\&+\frac{\frak a_t}{\frak a}\circ h^{-1}\partial_{\alpha'}^2\bar Z_{tt}+2\partial_{\alpha'}(\frac{\frak a_t}{\frak a}\circ h^{-1})\bar Z_{tt,\alpha'}.
 \end{aligned}
 \end{equation}
 Observe  $(I-\mathbb H) U_h^{-1}(\frac{\partial_\alpha}{h_\alpha})^2\bar z_t=(I-\mathbb H)\partial_{\alpha'}^2\bar Z_t=0$. We want to use the "almost holomorphicity" of the LHS of \eqref{eq:125} and the fact that $\partial_{\alpha'}^2\big(\frac{\frak a_t}{\frak a}\circ h^{-1}\big)$ is real valued to estimate $\int |\partial_{\alpha'}^2(\frac{\frak a_t}{\frak a}\circ h^{-1})(\bar Z_{tt}-i)|^2\,d\alpha$. We first show that the error term $e$ is well behaved.
By \eqref{eq:c16}, 
\begin{equation}\label{eq:123}
\begin{aligned}
&Z_{,\alpha'}U_h^{-1}[(\partial_t^2+i\frak a\partial_\alpha),  \frac{h_\alpha}{z_{\alpha}}  ] (\frac{\partial_\alpha}{h_\alpha})^2\bar z_t
=2((h_t\circ h^{-1})_{\alpha'}-D_{\alpha'} Z_t)U_h^{-1}\partial_t U_h\partial_{\alpha'}^2 \bar Z_t\\&
+((h_t\circ h^{-1})_{\alpha'}-D_{\alpha'} Z_t)^2 \partial_{\alpha'}^2 \bar Z_t +(U_h^{-1}\partial_t U_h(h_t\circ h^{-1})_{\alpha'}-U_h^{-1}\partial_tU_hD_{\alpha'} Z_t)\partial_{\alpha'}^2 \bar Z_t\\&+ 
(Z_{tt}+i)\partial_{\alpha'}\big(\frac{1}{Z_{,\alpha'}}\big) \partial_{\alpha'}^2 \bar Z_t;
\end{aligned}
 \end{equation}
from \eqref{eq:c1}, $U_h^{-1}\partial_tU_hD_{\alpha'} Z_t=D_{\alpha'} Z_{tt}-(D_{\alpha'} Z_t)^2$,     therefore by Appendix~\ref{quantities}, \eqref{eq:48}, \eqref{eq:50} and \eqref{eq:92},
\begin{equation}\label{eq:124}
\int \abs{Z_{,\alpha'}U_h^{-1}[(\partial_t^2+i\frak a\partial_\alpha),  \frac{h_\alpha}{z_{\alpha}}  ] (\frac{\partial_\alpha}{h_\alpha})^2\bar z_t}^2\,d\alpha'\le C(\frak E) E_2.
\end{equation}
The estimates \eqref{eq:124}, \eqref{eq:121}, \eqref{eq:115} and \eqref{eq:51} give that
 \begin{equation}\label{eq:126}
 \int |e|^2\,d\alpha'\le C(\frak E) E_2.
 \end{equation}

Now we apply $(I-\mathbb H)$ to both sides of \eqref{eq:125}, then rewrite $(I-\mathbb H)( \partial_{\alpha'}^2\big(\frac{\frak a_t}{\frak a}\circ h^{-1}\big)(\bar Z_{tt}-i))$ by commuting out $(\bar Z_{tt}-i)$:
 \begin{equation}\label{eq:127}
 \begin{aligned}
 (I-\mathbb H) \big(U_h^{-1}(\partial_t^2+i\frak a\partial_\alpha)(\frac{\partial_\alpha}{h_\alpha})^2\bar z_t\big)& =(\bar Z_{tt}-i)(I-\mathbb H)(\partial_{\alpha'}^2\big(\frac{\frak a_t}{\frak a}\circ h^{-1}\big))\\&+[\bar Z_{tt}, \mathbb H](\partial_{\alpha'}^2\big(\frac{\frak a_t}{\frak a}\circ h^{-1}\big))        +(I-\mathbb H)e
 \end{aligned}
 \end{equation}
 Since $\mathbb H$ is purely imaginary, $|\partial_{\alpha'}^2\big(\frac{\frak a_t}{\frak a}\circ h^{-1}\big)|\le |(I-\mathbb H)(\partial_{\alpha'}^2\big(\frac{\frak a_t}{\frak a}\circ h^{-1}\big))|$ hence
 \begin{equation}\label{eq:128}
 \begin{aligned}
 &\|(\bar Z_{tt}-i)\partial_{\alpha'}^2\big(\frac{\frak a_t}{\frak a}\circ h^{-1}\big)\|_{L^2}\le \|(\bar Z_{tt}-i)(I-\mathbb H)(\partial_{\alpha'}^2\big(\frac{\frak a_t}{\frak a}\circ h^{-1}\big))\|_{L^2}\\&\lesssim 
 \| (I-\mathbb H) \big(U_h^{-1}(\partial_t^2+i\frak a\partial_\alpha)(\frac{\partial_\alpha}{h_\alpha})^2\bar z_t\big)\|_{L^2}+ \|[\bar Z_{tt}, \mathbb H](\partial_{\alpha'}^2\big(\frac{\frak a_t}{\frak a}\circ h^{-1}\big))\|_{L^2}        +\|e\|_{L^2}.
 \end{aligned}
 \end{equation}
 By \eqref{3.21} and \eqref{eq:120}, \eqref{eq:80}, \eqref{eq:61} and Appendix~\ref{quantities},
 \begin{equation}
 \|[\bar Z_{tt}, \mathbb H](\partial_{\alpha'}^2\big(\frac{\frak a_t}{\frak a}\circ h^{-1}\big))\|_{L^2}\lesssim \|Z_{tt,\alpha'}\|_{L^\infty}\|\partial_{\alpha'}\big(\frac{\frak a_t}{\frak a}\circ h^{-1}\big)\|_{L^2}\le C(\frak E)E_2^{1/2}
 \end{equation}
 therefore
 \begin{equation}\label{eq:129}
 \|(\bar Z_{tt}-i)\partial_{\alpha'}^2\big(\frac{\frak a_t}{\frak a}\circ h^{-1}\big)\|_{L^2}\lesssim  \| (I-\mathbb H) (U_h^{-1}(\partial_t^2+i\frak a\partial_\alpha)(\frac{\partial_\alpha}{h_\alpha})^2\bar z_t)\|_{L^2}+C(\frak E)E_2^{1/2}.
 \end{equation}
 In what follows we will show that  $$\| (I-\mathbb H) (U_h^{-1}(\partial_t^2+i\frak a\partial_\alpha)(\frac{\partial_\alpha}{h_\alpha})^2\bar z_t)\|_{L^2}\le C(\frak E)E_2^{1/2}$$
  and complete the proof for Proposition~\ref{step1}.
  
 \subsubsection*{Step 4.1. Controlling $\|(I-\mathbb H) (U_h^{-1}(\partial_t^2+i\frak a\partial_\alpha)(\frac{\partial_\alpha}{h_\alpha})^2\bar z_t)\|_{L^2} $}   
We introduce the following notations.  We write $f_1\equiv f_2$, if $(I-\mathbb H)(f_1-f_2)=0$.  We define $\mathbb P_H:=\frac{(I+\mathbb H)}2$ and $\mathbb P_A:=\frac{(I-\mathbb H)}2$, so $\mathbb P_H+\mathbb P_A=I$, and $\mathbb P_H-\mathbb P_A=\mathbb H$.
By Proposition~\ref{prop:hilbe}, $\mathbb P_H$ is the projection onto the space of holomorphic functions in the lower half plane $P_-$, and $\mathbb P_A$ is the projection onto the space of anti-holomorphic functions in $P_-$.
 
We want to derive an estimate of  $\|(I-\mathbb H) (U_h^{-1}(\partial_t^2+i\frak a\partial_\alpha)U\circ h\|_{L^2}$ for a generic $U$ satisfying $U=\mathbb H U$, i.e. $U\equiv 0$.  
Observe $D_{\alpha'}U\equiv 0$. By \eqref{c8} of Proposition~\ref{prop:basic-iden}, $U$ satisfies
\begin{equation}\label{eq:134}
\begin{aligned}
U_h^{-1}(\partial_t^2 +i\frak a\partial_\alpha)U\circ h &\equiv
 2\frac{Z_t} {Z_{,\alpha'}}\partial_{\alpha'} (U_h^{-1}\partial_tU_h-\frac{Z_t } {Z_{,\alpha'}}\partial_{\alpha'} )U
\\& +  Z_t^2D^2_{\alpha'}U +2(Z_{tt}+i)D_{\alpha'}U.
 \end{aligned}
\end{equation}
 What we will do first is to use \eqref{eq:134} to rewrite $(I-\mathbb H) (U_h^{-1}(\partial_t^2+i\frak a\partial_\alpha)U\circ h$ into a favorable form so that desired estimate will follow. 
 
  We expand on the RHS of \eqref{eq:134} the term 
  $$ Z_t^2D^2_{\alpha'}U=\big(\frac{Z_t } {Z_{,\alpha'}}\big)^2\partial_{\alpha'}^2U+\frac{Z_t^2} {Z_{,\alpha'}} \partial_{\alpha'}(\frac1{Z_{,\alpha'}})\partial_{\alpha'}U$$
 by the product rule, and decompose $\frac{Z_t } {Z_{,\alpha'}}=\mathbb P_A(\frac{Z_t } {Z_{,\alpha'}})+\mathbb P_H(\frac{Z_t } {Z_{,\alpha'}})$. We have, because 
 $\partial_{\alpha'} (U_h^{-1}\partial_tU_h-\frac{Z_t } {Z_{,\alpha'}}\partial_{\alpha'} )U\equiv 0$ by \eqref{c5}, 
 \begin{equation}\label{eq:135}
 \begin{aligned}
U_h^{-1}(\partial_t^2 +i\frak a\partial_\alpha)U\circ h& \equiv
 2\mathbb P_A(\frac{Z_t} {Z_{,\alpha'}})\partial_{\alpha'} (U_h^{-1}\partial_tU_h-(\mathbb P_A+\mathbb P_H)(\frac{Z_t } {Z_{,\alpha'}})\partial_{\alpha'} )U
\\& +  \big((\mathbb P_A+\mathbb P_H)(\frac{Z_t } {Z_{,\alpha'}})\big)^2\partial_{\alpha'}^2 U\\&+\frac{Z_t ^2} {Z_{,\alpha'}}\partial_{\alpha'}(\frac1{Z_{,\alpha'}})\partial_{\alpha'} U
+2(Z_{tt}+i)D_{\alpha'} U.
  \end{aligned}
 \end{equation}
 We expand further the factor $\partial_{\alpha'} (\mathbb P_H(\frac{Z_t } {Z_{,\alpha'}})\partial_{\alpha'} U)$ on the RHS by the product rule. After cancelation we obtain
 \begin{equation}\label{eq:136}
 \begin{aligned}
&U_h^{-1}(\partial_t^2 +i\frak a\partial_\alpha)U\circ h\equiv
 2\mathbb P_A(\frac{Z_t} {Z_{,\alpha'}})\partial_{\alpha'} (U_h^{-1}\partial_tU_h-\mathbb P_A(\frac{Z_t } {Z_{,\alpha'}})\partial_{\alpha'} )U
\\& -2\mathbb P_A ( \frac{Z_t } {Z_{,\alpha'}} )  \{\mathbb P_H( \frac{Z_{t,\alpha'} } {Z_{,\alpha'}}  ) +\mathbb P_H( Z_t\partial_{\alpha'}\frac{1 } {Z_{,\alpha'}}  ) \}\partial_{\alpha'}U+  \\&\big(\mathbb P_A(\frac{Z_t } {Z_{,\alpha'}})\big)^2\partial_{\alpha'}^2 U+\frac{Z_t^2 } {Z_{,\alpha'}} \partial_{\alpha'}(\frac1{Z_{,\alpha'}})\partial_{\alpha'} U
+2(Z_{tt}+i)D_{\alpha'} U.
  \end{aligned}
 \end{equation}
Now because $\mathbb P_H( Z_t\partial_{\alpha'}\frac{1 } {Z_{,\alpha'}}  )$ and $\partial_{\alpha'}U$ are holomorphic, $$2\mathbb P_A ( \frac{Z_t } {Z_{,\alpha'}} )  \mathbb P_H( Z_t\partial_{\alpha'}\frac{1 } {Z_{,\alpha'}}  ) \partial_{\alpha'}U\equiv 2 \frac{Z_t } {Z_{,\alpha'}}   \mathbb P_H( Z_t\partial_{\alpha'}\frac{1 } {Z_{,\alpha'}}  ) \partial_{\alpha'}U,$$
 moreover
 $$
 -2 \frac{Z_t } {Z_{,\alpha'}}   \mathbb P_H( Z_t\partial_{\alpha'}\frac{1 } {Z_{,\alpha'}}  )+\frac{Z_t^2 } {Z_{,\alpha'}} \partial_{\alpha'}(\frac1{Z_{,\alpha'}})=-\frac{Z_t } {Z_{,\alpha'}}\mathbb H(Z_t \partial_{\alpha'}\frac1{Z_{,\alpha'}});\qquad\text{and}
 $$
 $$
 \begin{aligned}
 -\frac{Z_t } {Z_{,\alpha'}}\mathbb H(Z_t \partial_{\alpha'}\frac1{Z_{,\alpha'}})\partial_{\alpha'}U&=
 -Z_t\mathbb H(Z_t \partial_{\alpha'}\frac1{Z_{,\alpha'}})\partial_{\alpha'}(\frac1{Z_{,\alpha'}}U)\\&+
Z_t \partial_{\alpha'}(\frac1{Z_{,\alpha'}})\mathbb H(Z_t \partial_{\alpha'}\frac1{Z_{,\alpha'}})U;
 \end{aligned}
 $$
 and by straightforward expansion, 
 $$[Z_t, [Z_t, \mathbb H]] \partial_{\alpha'}\frac1{Z_{,\alpha'}}\equiv-2Z_t\mathbb H(Z_t\partial_{\alpha'}\frac1{Z_{,\alpha'}});$$
 and
 $$
 Z_t \partial_{\alpha'}(\frac1{Z_{,\alpha'}})\mathbb H(Z_t \partial_{\alpha'}\frac1{Z_{,\alpha'}})=\{\mathbb P_H\big(
 Z_t \partial_{\alpha'}(\frac1{Z_{,\alpha'}})\big)\}^2-\{\mathbb P_A\big(
 Z_t \partial_{\alpha'}(\frac1{Z_{,\alpha'}})\big)\}^2.
 $$
 Therefore
  \begin{equation}\label{eq:137}
 \begin{aligned}
 &U_h^{-1}(\partial_t^2 +i\frak a\partial_\alpha)U\circ h\equiv
 2\mathbb P_A(\frac{Z_t} {Z_{,\alpha'}})\partial_{\alpha'} (U_h^{-1}\partial_tU_h-\mathbb P_A(\frac{Z_t } {Z_{,\alpha'}})\partial_{\alpha'} )U
\\& -2\mathbb P_A ( \frac{Z_t } {Z_{,\alpha'}} )  \mathbb P_H( \frac{Z_{t,\alpha'} } {Z_{,\alpha'}}  )\partial_{\alpha'}U+  \big(\mathbb P_A(\frac{Z_t } {Z_{,\alpha'}})\big)^2\partial_{\alpha'}^2 U\\&+\frac12\{[Z_t,[Z_t,\mathbb H]]\partial_{\alpha'}\frac1{Z_{,\alpha'}}\}\partial_{\alpha'} (\frac1{Z_{,\alpha'}} U)-\{\mathbb P_A\big(
 Z_t \partial_{\alpha'}(\frac1{Z_{,\alpha'}})\big)\}^2 U\\&
+2\frac{(Z_{tt}+i)}{Z_{,\alpha'}}\partial_{\alpha'} U.
\end{aligned}
 \end{equation}
We further rewrite
$$2\mathbb P_A ( \frac{Z_t } {Z_{,\alpha'}} )  \mathbb P_H( \frac{Z_{t,\alpha'} } {Z_{,\alpha'}}  )\partial_{\alpha'}U\equiv \mathbb P_H( \frac{Z_{t,\alpha'} } {Z_{,\alpha'}}  )(I-\mathbb H)(\mathbb P_A ( \frac{Z_t } {Z_{,\alpha'}} ) \partial_{\alpha'}U).
$$
Apply $(I-\mathbb H)$ to both sides of \eqref{eq:137}, and rewrite terms of the form $(I-\mathbb H)(g_1 g_2)$ with $g_2=\mathbb H g_2$ as $[g_1,\mathbb H]g_2$. We obtain
  \begin{equation}\label{eq:0137}
 \begin{aligned}
 &(I-\mathbb H)U_h^{-1}(\partial_t^2 +i\frak a\partial_\alpha)U\circ h=
 2[\mathbb P_A(\frac{Z_t} {Z_{,\alpha'}}),\mathbb H]\partial_{\alpha'} (U_h^{-1}\partial_tU_h-\mathbb P_A(\frac{Z_t } {Z_{,\alpha'}})\partial_{\alpha'} )U
\\& -(I-\mathbb H)\{ \mathbb P_H( \frac{Z_{t,\alpha'} } {Z_{,\alpha'}}  )[\mathbb P_A ( \frac{Z_t } {Z_{,\alpha'}} ),\mathbb H]  \partial_{\alpha'}U\}+  [\big(\mathbb P_A(\frac{Z_t } {Z_{,\alpha'}})\big)^2,\mathbb H]\partial_{\alpha'}^2 U\\&+[\frac12[Z_t,[Z_t,\mathbb H]]\partial_{\alpha'}\frac1{Z_{,\alpha'}},\mathbb H]\partial_{\alpha'} (\frac1{Z_{,\alpha'}} U)\\&-(I-\mathbb H)\{\big(\mathbb P_A\big(
 Z_t \partial_{\alpha'}\frac1{Z_{,\alpha'}}\big)\big)^2 U\}
+2[\frac{Z_{tt}+i}{Z_{,\alpha'}},\mathbb H]\partial_{\alpha'} U.
\end{aligned}
 \end{equation}
We further use the identity\footnote{It is an easy consequence of integration by parts.}
$$-2[g_1,\mathbb H]\partial_{\alpha'}(g_1g_2)+[g_1^2,\mathbb H]\partial_{\alpha'}g_2=-\frac1{\pi i}\int\big(\frac{g_1(\alpha')-g_1(\beta')}{\alpha'-\beta'}\big)^2g_2(\beta')\,d\beta':=-[g_1,g_1; g_2]
$$
to rewrite the sum of second part of the first and the third terms on the right:
\begin{equation}\label{eq:138}
\begin{aligned}
- 2&[\mathbb P_A(\frac{Z_t} {Z_{,\alpha'}}),\mathbb H]\partial_{\alpha'}\paren{\mathbb P_A(\frac{Z_t } {Z_{,\alpha'}})\partial_{\alpha'} U}
+  [\big(\mathbb P_A(\frac{Z_t } {Z_{,\alpha'}})\big)^2,\mathbb H]\partial_{\alpha'}^2 U\\&=
-[\mathbb P_A(\frac{Z_t} {Z_{,\alpha'}}), \mathbb P_A(\frac{Z_t} {Z_{,\alpha'}}); \partial_{\alpha'} U].
\end{aligned}
\end{equation}
We are now ready to give the estimate for $(I-\mathbb H)U_h^{-1}(\partial_t^2 +i\frak a\partial_\alpha)U\circ h$.
We have, by \eqref{3.20}, \eqref{eq:b11} and H\"older's inequality,
  \begin{equation}\label{eq:139}
 \begin{aligned}
 &\|(I-\mathbb H)U_h^{-1}(\partial_t^2 +i\frak a\partial_\alpha)U\circ h\|_{L^2} \lesssim 
 \|\partial_{\alpha'}\mathbb P_A(\frac{Z_t} {Z_{,\alpha'}})\|_{L^\infty}\|U_h^{-1}\partial_t  U_h U\|_{L^2}+
\\& \| \mathbb P_H( \frac{Z_{t,\alpha'} } {Z_{,\alpha'}})\|_{L^\infty} \|\partial_{\alpha'}\mathbb P_A ( \frac{Z_t } {Z_{,\alpha'}} )\|_{L^\infty} \| U\|_{L^2}+  \|\partial_{\alpha'}\mathbb P_A(\frac{Z_t } {Z_{,\alpha'}})\|_{L^\infty}^2 \| U\|_{L^2}\\&+\|\partial_{\alpha'}\big([Z_t,[Z_t,\mathbb H]]\partial_{\alpha'}\frac1{Z_{,\alpha'}}\big)\|_{L^2}\|\frac1{Z_{,\alpha'}} U\|_{\dot H^{1/2}}\\&+ \|\mathbb P_A\big(
 Z_t \partial_{\alpha'}\frac1{Z_{,\alpha'}}\big)\|_{L^\infty}^2 \|U\|_{L^2}+
+\|\partial_{\alpha'}(\frac{Z_{tt}+i}{Z_{,\alpha'}})\|_{L^\infty} \|U\|_{L^2}.
\end{aligned}
 \end{equation}
Now by \eqref{eq:b14},
$$\|\partial_{\alpha'}\big([Z_t,[Z_t,\mathbb H]]\partial_{\alpha'}\frac1{Z_{,\alpha'}}\big)\|_{L^2}\lesssim \|Z_{t,\alpha'}\|_{L^2}^2\|\partial_{\alpha'}\frac1{Z_{,\alpha'}}\|_{L^2}
$$
and because $$\partial_{\alpha'}(\frac{Z_{tt}+i}{Z_{,\alpha'}})=D_{\alpha'}Z_{tt}+(Z_{tt}+i)\partial_{\alpha'}\frac{1}{Z_{,\alpha'}},$$
$$\|\partial_{\alpha'}(\frac{Z_{tt}+i}{Z_{,\alpha'}})\|_{L^\infty}\le \|D_{\alpha'}Z_{tt}\|_{L^\infty}+\|(Z_{tt}+i)\partial_{\alpha'}\frac{1}{Z_{,\alpha'}}\|_{L^\infty}.$$
We can conclude now by Appendix~\ref{quantities} that for any $U$ satisfying $U=\mathbb H U$, 
\begin{equation}\label{eq:140}
\|(I-\mathbb H)U_h^{-1}(\partial_t^2 +i\frak a\partial_\alpha)U\circ h\|_{L^2}\le C(\frak E)(\|U\|_{L^2}+\|U_h^{-1}\partial_t U_h U\|_{L^2}+\|\frac{1}{Z_{,\alpha'}}U\|_{\dot H^{1/2}}).
\end{equation}
As a consequence of \eqref{eq:140} and \eqref{eq:48}, \eqref{eq:50}, 
\begin{equation}\label{eq:142}
\| (I-\mathbb H) (U_h^{-1}(\partial_t^2+i\frak a\partial_\alpha)(\frac{\partial_\alpha}{h_\alpha})^2\bar z_t)\|_{L^2}\le C(\frak E)E_2^{1/2}.
\end{equation}
 \eqref{eq:129} then gives
\begin{equation}\label{eq:143}
 \|(\bar Z_{tt}-i)\partial_{\alpha'}^2\big(\frac{\frak a_t}{\frak a}\circ h^{-1}\big)\|_{L^2}^2\lesssim  C(\frak E)E_2.
 \end{equation}
Sum up \eqref{eq:115}, \eqref{eq:51}, \eqref{eq:121} and \eqref{eq:143}, 
$$\int |Z_{,\alpha'}U_h^{-1}G_2|^2\,d\alpha'\le C(\frak E)E_2.$$
This finishes the proof of Proposition~\ref{step1}.

   \end{proof}

\subsection{The proof of Proposition~\ref{step2}}\label{proof-prop2}

\begin{proof}
We prove Proposition~\ref{step2} by applying Lemma~\ref{basic-e} to \eqref{eq:44} for $k=3$, notice that
$(I-\mathbb H)U_h^{-1} D_\alpha(\frac{\partial_\alpha}{h_\alpha})^{2}\bar z_t=
(I-\mathbb H)D_{\alpha'}\partial_{\alpha'}^2\bar Z_t=0$. For $k=3$, the right hand side of \eqref{eq:44} is
\begin{equation}\label{eq:211}
G_3:= D_\alpha(\frac{\partial_\alpha}{h_\alpha})^{2}(-i\frak a_t\bar z_\alpha)+ [\partial_t^2+i\frak a\partial_\alpha, D_\alpha(\frac{\partial_\alpha}{h_\alpha})^{2}]\bar z_t
\end{equation}
Similar to the proof for Proposition~\ref{step1}, we only need to show that 
\begin{equation}\label{eq:201}
\int |Z_{,\alpha'}U_h^{-1}G_3|^2\,d\alpha'\le C(\frak E, E_2) E_3.
\end{equation}
We expand $Z_{,\alpha'}U_h^{-1}G_3$ by \eqref{eq:c12}, \eqref{eq:c5}, \eqref{eq:c10}. We have
\begin{equation}\label{eq:202}
\begin{aligned}
Z_{,\alpha'}U_h^{-1}G_3&=\partial_{\alpha'}^3\big(\frac{\frak a_t}{\frak a}\circ h^{-1} (\bar Z_{tt}-i)\big)+
Z_{,\alpha'}U_h^{-1}[\partial_t^2+i\frak a\partial_{\alpha}, D_\alpha](\frac{\partial_\alpha}{h_\alpha})^2\bar z_t\\&+\partial_{\alpha'}U_h^{-1}[\partial_t^2+i\frak a\partial_{\alpha}, \frac{\partial_\alpha}{h_\alpha} ]\frac{\partial_\alpha\bar z_{t}}{h_\alpha}+ \partial_{\alpha'}^2U_h^{-1}[\partial_t^2+i\frak a\partial_{\alpha}, \frac{\partial_\alpha}{h_\alpha} ]\bar z_t\\&:=Z_{,\alpha'}U_h^{-1}G_{3,0}+Z_{,\alpha'}U_h^{-1}G_{3,1}+Z_{,\alpha'}U_h^{-1}G_{3,2}+Z_{,\alpha'}U_h^{-1}G_{3,3}
\end{aligned}
\end{equation}
where
\begin{equation}\label{eq:206}
\begin{aligned}
Z_{,\alpha'}U_h^{-1}G_{3,0}:&=\partial_{\alpha'}^3\big(\frac{\frak a_t}{\frak a}\circ h^{-1} (\bar Z_{tt}-i)\big)\\&
=\partial_{\alpha'}^3\big(\frac{\frak a_t}{\frak a}\circ h^{-1}\big) (\bar Z_{tt}-i)+3\partial_{\alpha'}^2\big(\frac{\frak a_t}{\frak a}\circ h^{-1}\big) \bar Z_{tt,\alpha'}\\&+3\partial_{\alpha'}\big(\frac{\frak a_t}{\frak a}\circ h^{-1}\big)\partial_{\alpha'}^2 \bar Z_{tt}+\frac{\frak a_t}{\frak a}\circ h^{-1}\partial_{\alpha'}^3 \bar Z_{tt}
;
\end{aligned}
\end{equation}
\begin{equation}\label{eq:203}
\begin{aligned}
Z_{,\alpha'}U_h^{-1}G_{3,1}:&=Z_{,\alpha'}U_h^{-1}[\partial_t^2+i\frak a\partial_{\alpha}, D_\alpha](\frac{\partial_\alpha}{h_\alpha})^2\bar z_t\\&=-2D_{\alpha'}Z_{tt}\partial_{\alpha'}^3\bar Z_t-2(D_{\alpha'}Z_t)Z_{,\alpha'}U_h^{-1}\partial_t U_h \frac1{Z_{,\alpha'}}\partial_{\alpha'}^3\bar Z_t;
\end{aligned}
\end{equation}
\begin{equation}\label{eq:204}
\begin{aligned}
Z_{,\alpha'}&U_h^{-1}G_{3,2}:=\partial_{\alpha'}U_h^{-1}[\partial_t^2+i\frak a\partial_{\alpha}, \frac{\partial_\alpha}{h_\alpha} ]\frac{\partial_\alpha\bar z_{t}}{h_\alpha}
\\&=-\partial_{\alpha'}U_h^{-1}\partial_tU_h\{(h_t\circ h^{-1})_{\alpha'}\partial_{\alpha'}^2\bar Z_t\}-\partial_{\alpha'}\{(h_t\circ h^{-1})_{\alpha'}\partial_{\alpha'}U_{h^{-1}}\partial_tU_h\bar Z_{t,\alpha'}\}
\\&-i\partial_{\alpha'}\{\mathcal A_{\alpha'}\partial_{\alpha'}^2\bar Z_t\}\\&=
-(\partial_{\alpha'}U_h^{-1}\partial_tU_h(h_t\circ h^{-1})_{\alpha'})\partial_{\alpha'}^2\bar Z_t
-(U_h^{-1}\partial_tU_h(h_t\circ h^{-1})_{\alpha'})\partial_{\alpha'}^3\bar Z_t\\&
-\partial_{\alpha'}(h_t\circ h^{-1})_{\alpha'}(U_h^{-1}\partial_tU_h\partial_{\alpha'}^2\bar Z_t+ 
\partial_{\alpha'}U_h^{-1}\partial_tU_h\partial_{\alpha'}\bar Z_t)
\\&-(h_t\circ h^{-1})_{\alpha'}(\partial_{\alpha'}U_h^{-1}\partial_tU_h\partial_{\alpha'}^2\bar Z_t+\partial_{\alpha'}^2U_h^{-1}\partial_tU_h\partial_{\alpha'}\bar Z_t)\\&
-i(\partial_{\alpha'}\mathcal A_{\alpha'})\partial_{\alpha'}^2\bar Z_t-i\mathcal A_{\alpha'}\partial_{\alpha'}^3\bar Z_t;
\end{aligned}
\end{equation}
and
\begin{equation}\label{eq:205}
\begin{aligned}
Z_{,\alpha'}&U_h^{-1}G_{3,3}:=\partial_{\alpha'}^2U_h^{-1}[\partial_t^2+i\frak a\partial_{\alpha}, \frac{\partial_\alpha}{h_\alpha} ]\bar z_{t}
\\&=-\partial_{\alpha'}^2U_h^{-1}\partial_tU_h\{(h_t\circ h^{-1})_{\alpha'}\partial_{\alpha'}\bar Z_t\}-\partial_{\alpha'}^2\{(h_t\circ h^{-1})_{\alpha'}\partial_{\alpha'}\bar Z_{tt}\}
\\&-i\partial_{\alpha'}^2\{\mathcal A_{\alpha'}\partial_{\alpha'}\bar Z_t\}.
\end{aligned}
\end{equation}

\subsubsection*{Step 1. Quantities controlled by $E_3$ and a polynomial of $\frak E$ and $E_2$} 
By the definition of $E_3$, and the fact that $\|A_1\|_{L^\infty}\le C(\frak E)$ (cf. Appendix~\ref{quantities}), 
\begin{equation}\label{eq:207}
\|\partial_{\alpha'}^3\bar Z_t\|_{L^2}^2,\quad   \nm{Z_{,\alpha'}U_h^{-1}\partial_t U_h \frac1{Z_{,\alpha'}}\partial_{\alpha'}^3\bar Z_t}_{L^2}^2,\quad \nm{\frac1{Z_{,\alpha'}}\partial_{\alpha'}^3\bar Z_t}_{\dot H^{1/2}}^2 \le C(\frak E) E_3.
\end{equation}
We commute $Z_{,\alpha'}$ with $U_h^{-1}\partial_t U_h$ of the second quantity in \eqref{eq:207}:
\begin{equation}
U_h^{-1}\partial_t U_h\partial_{\alpha'}^3\bar Z_t= Z_{,\alpha'}U_h^{-1}\partial_t U_h \frac1{Z_{,\alpha'}}\partial_{\alpha'}^3\bar Z_t-[Z_{,\alpha'}, U_h^{-1}\partial_t U_h] \frac1{Z_{,\alpha'}}\partial_{\alpha'}^3\bar Z_t
\end{equation}
By \eqref{eq:c13} and Appendix~\ref{quantities}, we have
\begin{equation}\label{eq:208}
\abs{\nm{U_h^{-1}\partial_t U_h\partial_{\alpha'}^3\bar Z_t}_{L^2}-\nm{Z_{,\alpha'}U_h^{-1}\partial_t U_h \frac1{Z_{,\alpha'}}\partial_{\alpha'}^3\bar Z_t}_{L^2}}
\le C(\frak E)\|\partial_{\alpha'}^3\bar Z_t\|_{L^2},
\end{equation}
so
\begin{equation}\label{eq:209}
\nm{U_h^{-1}\partial_t U_h\partial_{\alpha'}^3\bar Z_t}_{L^2}^2\le C(\frak E)E_3
\end{equation}
By \eqref{eq:20}, 
$$\partial_{\alpha'}U_h^{-1}\partial_t U_h\partial_{\alpha'}^2\bar Z_t-U_h^{-1}\partial_t U_h\partial_{\alpha'}^3\bar Z_t=[\partial_{\alpha'},U_h^{-1}\partial_t U_h]\partial_{\alpha'}^2\bar Z_t=(h_t\circ h^{-1})_{\alpha'}\partial_{\alpha'}^3\bar Z_t,$$
so
\begin{equation}\label{eq:212}
\|\partial_{\alpha'}U_h^{-1}\partial_t U_h\partial_{\alpha'}^2\bar Z_t\|_{L^2}^2\le C(\frak E)E_3.
\end{equation}
As a consequence of \eqref{eq:sobolev}, \eqref{eq:48}, \eqref{eq:207}, \eqref{eq:50} and \eqref{eq:212},
\begin{equation}\label{eq:213}
\|\partial_{\alpha'}^2\bar Z_t\|_{L^\infty}^2\le C(\frak E, E_2)E_3^{1/2},\quad \|U_h^{-1}\partial_t U_h\partial_{\alpha'}^2\bar Z_t\|_{L^\infty}^2\le C(\frak E, E_2)E_3^{1/2}.
\end{equation}

By \eqref{eq:20} again,
$$
\begin{aligned}
\partial_{\alpha'}^2U_h^{-1}\partial_t U_h\partial_{\alpha'}\bar Z_t&-\partial_{\alpha'}U_h^{-1}\partial_t U_h\partial_{\alpha'}^2\bar Z_t=\partial_{\alpha'}
[\partial_{\alpha'},U_h^{-1}\partial_t U_h]\partial_{\alpha'}\bar Z_t\\&=\partial_{\alpha'}(h_t\circ h^{-1})_{\alpha'}\partial_{\alpha'}^2\bar Z_t+(h_t\circ h^{-1})_{\alpha'}\partial_{\alpha'}^3\bar Z_t,
\end{aligned}
$$
so by \eqref{eq:212}, \eqref{eq:207}, and \eqref{eq:78}, \eqref{eq:213}, \eqref{eq:48}, \eqref{eq:61} and Appendix~\ref{quantities}, 
\begin{equation}\label{eq:214}
\|\partial_{\alpha'}^2U_h^{-1}\partial_t U_h\partial_{\alpha'}\bar Z_t\|_{L^2}^2\le C(\frak E, E_2)E_3+C(\frak E, E_2);
\end{equation}
and consequently by \eqref{eq:sobolev} and \eqref{eq:60}, 
\begin{equation}\label{eq:215}
\|\partial_{\alpha'}U_h^{-1}\partial_t U_h\partial_{\alpha'}\bar Z_t\|_{L^\infty}^2\le C(\frak E, E_2)E_3^{1/2}+C(\frak E, E_2).
\end{equation}

\subsubsection*{Step 2. Controlling $G_{3,1}$}
By \eqref{eq:203}, Appendix~\ref{quantities} and \eqref{eq:207},
\begin{equation}\label{eq:210}
\int {|Z_{,\alpha'}U_h^{-1}G_{3,1}|^2}\,d\alpha
\le C(\frak E)E_3. 
\end{equation}

\subsubsection*{Step 3. Controlling $G_{3,2}$}
By \eqref{eq:93}, \eqref{eq:101}, \eqref{eq:78}, \eqref{eq:95}, \eqref{eq:60}, \eqref{eq:48}, \eqref{eq:213}, \eqref{eq:92}, \eqref{eq:207}, \eqref{eq:215}, \eqref{eq:212}, \eqref{eq:214}, \eqref{eq:0104},
\eqref{eq:0112}, \eqref{eq:80} and Appendix~\ref{quantities}, we can control each of the terms in \eqref{eq:204}. Sum up, we have
\begin{equation}\label{eq:216}
\int {|Z_{,\alpha'}U_h^{-1}G_{3,2}|^2}\,d\alpha
\le C(\frak E, E_2)E_3+C(\frak E, E_2). 
\end{equation}

\subsubsection*{Step 4. Controlling $G_{3,3}$} Expanding $G_{3,3}$ in \eqref{eq:205} by the product rule, 
we find that the additional types of terms that have not already appeared in \eqref{eq:204} and controlled in the previous step are
$$
\begin{aligned}
&(\partial_{\alpha'}^2U_h^{-1}\partial_tU_h(h_t\circ h^{-1})_{\alpha'})\partial_{\alpha'}\bar Z_t,\quad (\partial_{\alpha'}^2(h_t\circ h^{-1})_{\alpha'})U_h^{-1}\partial_tU_h\partial_{\alpha'}\bar Z_t,\\&\partial_{\alpha'}^2\{(h_t\circ h^{-1})_{\alpha'}\partial_{\alpha'}\bar Z_{tt}\}, \quad \text{and }\quad(\partial_{\alpha'}^3 \mathcal A)  \partial_{\alpha'}\bar Z_t.
\end{aligned}
$$
\subsubsection*{Step 4.1. Controlling $\partial_{\alpha'}^2(h_t\circ h^{-1})_{\alpha'}$ and $\partial_{\alpha'}^3Z_{tt}$} 
We begin with controlling $\partial_{\alpha'}A_1$, $\partial_{\alpha'}^2A_1$ and $\partial_{\alpha'}^2\frac{1}{Z_{,\alpha'}}$. Differentiating \eqref{a1} gives
\begin{equation}\label{eq:104}
 \partial_{\alpha'}A_1  
 =-\Im [Z_{t,\alpha'},\mathbb H]\bar Z_{t,\alpha'}-\Im [Z_t,\mathbb H]\partial_{\alpha'}\bar Z_{t,\alpha'}
  \end{equation}
  so by \eqref{eq:b13}, Appendix~\ref{quantities} and \eqref{eq:48},
  \begin{equation}\label{eq:105}
  \nm{\partial_{\alpha'}A_1}_{L^\infty}\lesssim \|Z_{t,\alpha'}\|_{L^2}\|\partial_{\alpha'}^2Z_{t}\|_{L^2}\lesssim C(\frak E)E_2^{1/2}. 
    \end{equation}
Differentiating again with respect to $\alpha'$ then apply \eqref{3.20}, \eqref{3.22} and \eqref{eq:sobolev} gives
\begin{equation}\label{eq:219}
\|\partial_{\alpha'}^2A_1\|_{L^2}\lesssim \|Z_{t,\alpha'}\|_{L^\infty}\|\partial_{\alpha'}^2Z_{t}\|_{L^2}\le C(\frak E, E_2).
\end{equation}
To estimate $\partial_{\alpha'}^2\frac{1}{Z_{,\alpha'}}$ we begin with 
 \eqref{interface-a1}:
$$-i\frac1{Z_{,\alpha'}}=\frac{\bar Z_{tt}-i}{A_1}.$$
Taking two derivatives with respect to $\alpha'$ gives
\begin{equation}\label{eq:220}
-i\partial_{\alpha'}^2\frac{1}{Z_{,\alpha'}}=\frac{\partial_{\alpha'}^2\bar Z_{tt}}{A_1}-2\bar Z_{tt,\alpha'}\frac{\partial_{\alpha'}A_1}{A_1^2}+(\bar Z_{tt}-i)(-\frac{\partial_{\alpha'}^2A_1}{A_1^2}+2\frac{(\partial_{\alpha'}A_1)^2}{A_1^3});
\end{equation}
therefore
\begin{equation}\label{eq:221}
\begin{aligned}
\|\partial_{\alpha'}^2\frac{1}{Z_{,\alpha'}}\|_{L^2}\lesssim &\|\partial_{\alpha'}^2\bar Z_{tt}\|_{L^2}+\|\partial_{\alpha'}\bar Z_{tt}\|_{L^2}\|\partial_{\alpha'}A_1\|_{L^\infty}\\&+\|\frac{1}{Z_{,\alpha'}}\|_{L^\infty}(\|\partial_{\alpha'}^2A_1\|_{L^2}+\|\partial_{\alpha'}A_1\|_{L^2}\|\partial_{\alpha'}A_1\|_{L^\infty})\le C(\frak E, E_2),
\end{aligned}
\end{equation}
and consequently by \eqref{eq:sobolev},
\begin{equation}\label{eq:225}
\|\partial_{\alpha'}\frac{1}{Z_{,\alpha'}}\|_{L^\infty} \le C(\frak E, E_2).
\end{equation}

We are now ready to give the estimates for $ \|\partial_{\alpha'}^2(h_t\circ h^{-1})_{\alpha'}\|_{L^2}  $ and $\| \partial_{\alpha'}^3\bar Z_{tt}  \|_{L^2}$.   Rewriting the first term on the right of \eqref{eq:74} as a commutator then differentiating yields, 
\begin{equation}\label{eq:217}
\begin{aligned}
\partial_{\alpha'}^2(h_t\circ h^{-1})_{\alpha'}&-2\Re (\frac{\partial_{\alpha'}^3Z_{t}}{Z_{,\alpha'}}+ \partial_{\alpha'}^2Z_{t}\partial_{\alpha'}\frac{1}{Z_{,\alpha'}})
=\Re\{2\partial_{\alpha'} [Z_{t,\alpha'}, \mathbb H]\partial_{\alpha'}\frac{1}{Z_{,\alpha'}})    
\\&- \partial_{\alpha'}[\frac{1}{\bar Z_{,\alpha'}}, \mathbb H]\partial_{\alpha'}^2\bar Z_{t}+\partial_{\alpha'}[Z_{t},\mathbb H]\partial_{\alpha'}^2\frac{1}{Z_{,\alpha'}}
 \}.
\end{aligned}
\end{equation}
Expanding the right hand side of \eqref{eq:217} by the product rule. By \eqref{3.20}, \eqref{3.21}, 
\begin{equation}\label{eq:218}
\begin{aligned}
\|\partial_{\alpha'}^2(h_t\circ h^{-1})_{\alpha'}\|_{L^2}&\lesssim \|\frac1{Z_{,\alpha'}}\|_{L^\infty}\|\partial_{\alpha'}^3Z_{t}\|_{L^2}+\|\partial_{\alpha'}^2Z_{t}\|_{L^\infty}\|\partial_{\alpha'}\frac{1}{Z_{,\alpha'}}\|_{L^2}\\&+\|Z_{t,\alpha'}\|_{L^\infty}\|\partial_{\alpha'}^2\frac{1}{Z_{,\alpha'}}\|_{L^2}\le C(\frak E, E_2)E_3^{1/2}+C(\frak E, E_2).
\end{aligned}
\end{equation}

For $\| \partial_{\alpha'}^3\bar Z_{tt}  \|_{L^2}$,  we differentiate \eqref{eq:75} with respect to $\alpha'$:
\begin{equation}\label{eq:230}
\begin{aligned}
\partial_{\alpha'}^3\bar Z_{tt}-\partial_{\alpha'}U_h^{-1}\partial_t U_h\partial_{\alpha'}^2\bar Z_t
&=3\partial_{\alpha'}(h_t\circ h^{-1})_{\alpha'} \partial_{\alpha'}^2\bar Z_t +2(h_t\circ h^{-1})_{\alpha'} \partial_{\alpha'}^3\bar Z_t\\&+\partial_{\alpha'}^2(h_t\circ h^{-1})_{\alpha'} \partial_{\alpha'}\bar Z_t,
\end{aligned}
\end{equation}
therefore by \eqref{eq:78}, \eqref{eq:207}, \eqref{eq:212}, \eqref{eq:213},
\begin{equation}\label{eq:222}
\| \partial_{\alpha'}^3\bar Z_{tt}  \|_{L^2}^2\le C(\frak E, E_2)E_3+C(\frak E, E_2),
\end{equation}
and as a consequence of \eqref{eq:sobolev},
\begin{equation}\label{eq:223}
\| \partial_{\alpha'}^2\bar Z_{tt}  \|_{L^\infty}^2\le C(\frak E, E_2)E_3^{1/2}+C(\frak E, E_2).
\end{equation}
\subsubsection*{Step 4.2. Controlling $\partial_{\alpha'}^3\mathcal A$.} We differentiate \eqref{eq:112} with respect to $\alpha'$ and use the product rule to expand. We have,
\begin{equation}\label{eq:224}
\begin{aligned}
\|\partial_{\alpha'}^3\mathcal A\|_{L^2}&\le \|\frac1{Z_{,\alpha'}}\|_{L^\infty}\|\partial_{\alpha'}^3\bar Z_{tt}\|_{L^2}+ \|\partial_{\alpha'}\frac1{Z_{,\alpha'}}\|_{L^\infty}\|\partial_{\alpha'}^2\bar Z_{tt}\|_{L^2}\\&+\|\partial_{\alpha'}^2\frac1{Z_{,\alpha'}}\|_{L^2}\|\partial_{\alpha'}\bar Z_{tt}\|_{L^\infty}\le C(\frak E, E_2)E_3^{1/2}+C(\frak E, E_2).
\end{aligned}
\end{equation}
\subsubsection*{Step 4.3. Controlling $\partial_{\alpha'}^2U_h^{-1}\partial_t U_h (h_t\circ h^{-1})_{\alpha'}$.} By \eqref{eq:c12} and \eqref{eq:20}, 
\begin{equation}\label{eq:226}
\begin{aligned}
\partial_{\alpha'}^2U_h^{-1}\partial_t U_h (h_t\circ h^{-1})_{\alpha'}&=U_h^{-1}\partial_t U_h\partial_{\alpha'}^2 (h_t\circ h^{-1})_{\alpha'}\\&+(\partial_{\alpha'} (h_t\circ h^{-1})_{\alpha'})^2+2 (h_t\circ h^{-1})_{\alpha'}\partial_{\alpha'}^2 (h_t\circ h^{-1})_{\alpha'}
\end{aligned}
\end{equation}
where
\begin{equation}\label{eq:227}
\begin{aligned}
&\|(\partial_{\alpha'} (h_t\circ h^{-1})_{\alpha'})^2+2 (h_t\circ h^{-1})_{\alpha'}\partial_{\alpha'}^2 (h_t\circ h^{-1})_{\alpha'}\|_{L^2}\\&\lesssim \|\partial_{\alpha'} (h_t\circ h^{-1})_{\alpha'}\|_{L^2}\|\partial_{\alpha'} (h_t\circ h^{-1})_{\alpha'}\|_{L^\infty}+\|\partial_{\alpha'}^2 (h_t\circ h^{-1})_{\alpha'}\|_{L^2}\| (h_t\circ h^{-1})_{\alpha'}\|_{L^\infty}\\&\lesssim  \|\partial_{\alpha'} (h_t\circ h^{-1})_{\alpha'}\|_{L^2}^{3/2}\|\partial_{\alpha'}^2(h_t\circ h^{-1})_{\alpha'}\|_{L^2}^{1/2}+\|\partial_{\alpha'}^2 (h_t\circ h^{-1})_{\alpha'}\|_{L^2}\| (h_t\circ h^{-1})_{\alpha'}\|_{L^\infty}\\&\le C(\frak E, E_2)E_3^{1/2}+C(\frak E, E_2).
\end{aligned}
\end{equation}
For $U_h^{-1}\partial_t U_h\partial_{\alpha'}^2 (h_t\circ h^{-1})_{\alpha'}$, we differentiate \eqref{eq:217} and use the product rule and \eqref{eq:c14} to expand the derivatives,
\begin{equation}\label{eq:228}
\begin{aligned}
&U_h^{-1}\partial_t U_h\partial_{\alpha'}^2(h_t\circ h^{-1})_{\alpha'}-2\Re (U_h^{-1}\partial_t U_h(\frac{\partial_{\alpha'}^3Z_{t}}{Z_{,\alpha'}})+ U_h^{-1}\partial_t U_h(\partial_{\alpha'}^2Z_{t}\partial_{\alpha'}\frac{1}{Z_{,\alpha'}}))
\\&=\Re U_h^{-1}\partial_t U_h\{2\partial_{\alpha'} [Z_{t,\alpha'}, \mathbb H]\partial_{\alpha'}\frac{1}{Z_{,\alpha'}})    
- \partial_{\alpha'}[\frac{1}{\bar Z_{,\alpha'}}, \mathbb H]\partial_{\alpha'}^2\bar Z_{t}+\partial_{\alpha'}[Z_{t},\mathbb H]\partial_{\alpha'}^2\frac{1}{Z_{,\alpha'}}
 \};
\end{aligned}
\end{equation} 
we then use \eqref{3.20}, \eqref{3.21}, \eqref{3.22}, \eqref{eq:b12} and H\"older's inequality to do the estimates. We have
\begin{equation}\label{eq:229}
\begin{aligned}
&\|U_h^{-1}\partial_t U_h\partial_{\alpha'}^2(h_t\circ h^{-1})_{\alpha'}\|_{L^2}\lesssim \|U_h^{-1}\partial_t U_h\partial_{\alpha'}^3Z_{t}\|_{L^2}\|\frac1{Z_{,\alpha'}}\|_{L^\infty}\\&+\|\partial_{\alpha'}^3Z_{t}\|_{L^2}\|U_h^{-1}\partial_t U_h\frac1{Z_{,\alpha'}}\|_{L^\infty}+ \|U_h^{-1}\partial_t U_h\partial_{\alpha'}^2Z_{t}\|_{L^\infty}\|\partial_{\alpha'}\frac{1}{Z_{,\alpha'}}\|_{L^2}\\&+\|\partial_{\alpha'}^2Z_{t}\|_{L^\infty}(\|U_h^{-1}\partial_t U_h\partial_{\alpha'}\frac{1}{Z_{,\alpha'}}\|_{L^2}+\|\partial_{\alpha'}U_h^{-1}\partial_t U_h\frac{1}{Z_{,\alpha'}}\|_{L^2})\\&
+\|\partial_{\alpha'}^2Z_{t}\|_{L^\infty}\|(h_t\circ h^{-1})_{\alpha'}\|_{L^\infty}\|\partial_{\alpha'}\frac{1}{Z_{,\alpha'}}\|_{L^2}+\|U_h^{-1}\partial_t U_h\partial_{\alpha'}Z_{t}\|_{L^\infty}\|\partial_{\alpha'}^2\frac{1}{Z_{,\alpha'}}\|_{L^2}\\&+\|\partial_{\alpha'}Z_{t}\|_{L^\infty}\|(h_t\circ h^{-1})_{\alpha'}\|_{L^\infty}\|\partial_{\alpha'}^2\frac{1}{Z_{,\alpha'}}\|_{L^2}+\|Z_{tt,\alpha'}\|_{L^\infty}\|\partial_{\alpha'}^2\frac{1}{Z_{,\alpha'}}\|_{L^2}\\&+\|Z_{t,\alpha'}\|_{L^\infty}\|U_h^{-1}\partial_t U_h\partial_{\alpha'}^2\frac{1}{Z_{,\alpha'}}\|_{L^2}.
\end{aligned}
\end{equation}
Now by \eqref{eq:20}, \eqref{eq:c12},
\begin{equation}\label{eq:231}
U_h^{-1}\partial_t U_h\partial_{\alpha'}^2\frac{1}{Z_{,\alpha'}}=\partial_{\alpha'}^2 U_h^{-1}\partial_t U_h\frac{1}{Z_{,\alpha'}}-\partial_{\alpha'}(h_t\circ h^{-1})_{\alpha'}\partial_{\alpha'}\frac{1}{Z_{,\alpha'}}-2(h_t\circ h^{-1})_{\alpha'}\partial_{\alpha'}^2\frac{1}{Z_{,\alpha'}}
\end{equation}
and
\begin{equation}\label{eq:232}
U_h^{-1}\partial_t U_h\frac{1}{Z_{,\alpha'}}=\frac{1}{Z_{,\alpha'}}( (h_t\circ h^{-1})_{\alpha'}-D_{\alpha'}Z_t);
\end{equation}
\begin{equation}\label{eq:233}
\begin{aligned}
&\partial_{\alpha'}^2 U_h^{-1}\partial_t U_h\frac{1}{Z_{,\alpha'}}=(\partial_{\alpha'}^2\frac{1}{Z_{,\alpha'}})( (h_t\circ h^{-1})_{\alpha'}-D_{\alpha'}Z_t)\\&+2(\partial_{\alpha'}\frac{1}{Z_{,\alpha'}})( \partial_{\alpha'}(h_t\circ h^{-1})_{\alpha'}-\partial_{\alpha'}D_{\alpha'}Z_t)+\frac{1}{Z_{,\alpha'}}( \partial_{\alpha'}^2(h_t\circ h^{-1})_{\alpha'}-\partial_{\alpha'}^2D_{\alpha'}Z_t);
\end{aligned}
\end{equation}
we further expand
$$\partial_{\alpha'}D_{\alpha'}Z_t=\frac{1}{Z_{,\alpha'}}\partial_{\alpha'}^2Z_t+ \partial_{\alpha'}\frac{1}{Z_{,\alpha'}}\partial_{\alpha'}Z_t;$$
and
$$\partial_{\alpha'}^2D_{\alpha'}Z_t= \frac{1}{Z_{,\alpha'}}\partial_{\alpha'}^3Z_t+ 2\partial_{\alpha'}\frac{1}{Z_{,\alpha'}}\partial_{\alpha'}^2Z_t+\partial_{\alpha'}^2\frac{1}{Z_{,\alpha'}} Z_{t,\alpha'}.$$
Therefore
\begin{equation}\label{eq:234}
\|U_h^{-1}\partial_t U_h\frac{1}{Z_{,\alpha'}}\|_{L^\infty} \le C(\frak E),\quad \|U_h^{-1}\partial_t U_h\partial_{\alpha'}^2\frac{1}{Z_{,\alpha'}}\|_{L^2}\le C(\frak E, E_2)E_3^{1/2}+C(\frak E, E_2).
\end{equation}
By \eqref{eq:229},
\begin{equation}\label{eq:235}
\|U_h^{-1}\partial_t U_h\partial_{\alpha'}^2(h_t\circ h^{-1})_{\alpha'}\|_{L^2}\le C(\frak E, E_2)E_3^{1/2}+C(\frak E, E_2).
\end{equation}
\subsubsection*{Step 4.4. Conclusion for $G_{3,3}$} We expand $G_{3,3}$ by product rules. Sum up the estimates in Steps 4.1-4.3, we have
\begin{equation}\label{eq:236}
\int |Z_{,\alpha'}U_h^{-1}G_{3,3}|^2\,d\alpha'\le C(\frak E, E_2)E_3+C(\frak E, E_2).
\end{equation}
\subsubsection*{Step 5. Controlling $G_{3,0}$} We estimate $\|Z_{,\alpha'}U_h^{-1}G_{3,0}\|_{L^2}$ using similar ideas as that in Step 4 for Proposition~\ref{step1}. By \eqref{eq:206}, we must control 
$\|\partial_{\alpha'}^3(\frac{\frak a_t}{\frak a}\circ h^{-1}) (\bar Z_{tt}-i)\|_{L^2}$, $\|\partial_{\alpha'}^2(\frac{\frak a_t}{\frak a}\circ h^{-1}) \bar Z_{tt,\alpha'}\|_{L^2}$, $\|\partial_{\alpha'}(\frac{\frak a_t}{\frak a}\circ h^{-1}) \partial_{\alpha'}^2\bar Z_{tt}\|_{L^2}$ and $\|(\frac{\frak a_t}{\frak a}\circ h^{-1}) \partial_{\alpha'}^3\bar Z_{tt}\|_{L^2}$. First by \eqref{eq:222} and Appendix~\ref{quantities}, 
\begin{equation}\label{eq:237}
\|(\frac{\frak a_t}{\frak a}\circ h^{-1}) \partial_{\alpha'}^3\bar Z_{tt}\|_{L^2}^2\le C(\frak E, E_2)E_3+C(\frak E, E_2).
\end{equation}
By \eqref{eq:120} and \eqref{eq:sobolev},
\begin{equation}\label{eq:238}
\|\partial_{\alpha'}(\frac{\frak a_t}{\frak a}\circ h^{-1}) \partial_{\alpha'}^2\bar Z_{tt}\|_{L^2}^2\le C(\frak E, E_2)E_3+C(\frak E, E_2).
\end{equation}
By \eqref{at},
\begin{equation}\label{eq:239}
\begin{aligned}
\|\partial_{\alpha'}^2(\frac{\frak a_t}{\frak a}\circ h^{-1}) \|_{L^2}&\lesssim \|\partial_{\alpha'}^2Z_t\|_{L^2}(\|Z_{tt,\alpha'}\|_{L^\infty}+\|\mathbb H Z_{tt,\alpha'}\|_{L^\infty})+\|\partial_{\alpha'}^2Z_{tt}\|_{L^2}\|Z_{t,\alpha'}\|_{L^\infty}\\&+(\|\partial_{\alpha'}^2Z_t\|_{L^\infty}\|Z_{t,\alpha'}\|_{L^2}+\|\partial_{\alpha'}^2Z_t\|_{L^2}\|Z_{t,\alpha'}\|_{L^\infty}) \|D_{\alpha'}Z_t\|_{L^\infty}\\&+\|\partial_{\alpha'}Z_t\|_{L^\infty}^2 \|\partial_{\alpha'}D_{\alpha'}Z_t\|_{L^2}+\|\partial_{\alpha'}(\frac{\frak a_t}{\frak a}\circ h^{-1}) \|_{L^2}\|\partial_{\alpha'}A_1\|_{L^\infty}\\&+\|\frac{\frak a_t}{\frak a} \|_{L^\infty}\|\partial_{\alpha'}^2A_1\|_{L^2}.
\end{aligned}
\end{equation}
so
\begin{equation}\label{eq:240}
\|\partial_{\alpha'}^2(\frac{\frak a_t}{\frak a}\circ h^{-1}) \|_{L^2}\le C(\frak E, E_2)E_3^{1/4}+C(\frak E, E_2),
\end{equation}
therefore
\begin{equation}\label{eq:241}
\|\partial_{\alpha'}^2(\frac{\frak a_t}{\frak a}\circ h^{-1}) \partial_{\alpha'}\bar Z_{tt}\|_{L^2}^2\le C(\frak E, E_2)E_3+C(\frak E, E_2).
\end{equation}
Now similar to \eqref{eq:123} and \eqref{eq:124}, we compute $Z_{,\alpha'}U_h^{-1}[(\partial_t^2+i\frak a\partial_\alpha),  \frac{h_\alpha}{z_{\alpha}}  ](\frac{\partial_\alpha}{h_\alpha})^3\bar z_t$  by \eqref{eq:c17} and have 
\begin{equation}\label{eq:243}
\int \abs{Z_{,\alpha'}U_h^{-1}[(\partial_t^2+i\frak a\partial_\alpha),  \frac{h_\alpha}{z_{\alpha}}  ] (\frac{\partial_\alpha}{h_\alpha})^3\bar z_t}^2\,d\alpha'\le C(\frak E) E_3.
\end{equation}

Now we begining with \eqref{eq:44} for $k=3$. After expansion, commuting and precomposing with $h^{-1}$,   and using the above estimates, we arrive at
\begin{equation}\label{eq:244}
U_h^{-1}(\partial_t^2+i\frak a\partial_\alpha)U_h \partial_{\alpha'}^3\bar Z_t=(\bar Z_{tt}-i)\partial_{\alpha'}^3(\frac{\frak a_t}{\frak a}\circ h^{-1})+e_1
\end{equation}
with
\begin{equation}\label{eq:245}
\int |e_1|^2\,d\alpha'\le C(\frak E, E_2)E_3+C(\frak E, E_2).
\end{equation}
Going through similar calculations as in \eqref{eq:127} to \eqref{eq:129}, 
then applying \eqref{eq:140} to $U=\partial_{\alpha'}^3\bar Z_t$, we obtain
\begin{equation}\label{eq:246}
\|(\bar Z_{tt}-i)\partial_{\alpha'}^3(\frac{\frak a_t}{\frak a}\circ h^{-1})\|_{L^2}^2\le C(\frak E, E_2)E_3+C(\frak E, E_2).
\end{equation}
This finishes the proof for Proposition~\ref{step2}.

\end{proof}

\subsection{Completing the proof for Theorem~\ref{blow-up}}\label{complete1}
\begin{proof}
Let $s\ge 4$. Let the initial interface $Z(\cdot,0)=Z(0)$, the initial velocity $Z_t(\cdot,0)=Z_t(0)$ be given and satisfy \eqref{interface-holo} and $\bar Z_t(0)=\mathbb H \bar Z_t(0)$; let $A_1(0)$ satisfy \eqref{a1} and the initial acceleration $Z_{tt}(0)$ satisfy \eqref{interface-a1}. 
Assume  $Z_{,\alpha'}(0)-1\in L^\infty(\mathbb R)$, $Z_t(0)\in H^{s+1/2}(\mathbb R)$, 
and  $Z_{tt}(0)\in H^s(\mathbb R)$. It is clear that $E_2(0)+E_3(0)<\infty$. Assume $Z=Z(\cdot, t)$, for $t\in [0, T^*)$ is a solution of \eqref{interface-r}-\eqref{interface-holo}, such that $(Z_t, Z_{tt}, Z_{,\alpha'}-1)\in C([0, T^*), H^{s+1/2}(\mathbb R)\times H^s(\mathbb R)\times H^s(\mathbb R))$, and $T^*$ is the maximum existence time as defined in Theorem~\ref{blow-up}. Assume $T^*<\infty$, for otherwise we are done; and assume $\sup_{t\in [0, T^*)}\frak E(t):=M<\infty$. We want to show
$\sup_{t\in [0, T^*)} (\|Z_{tt}(t)\|_{H^3}+\|Z_t(t)\|_{H^{3+1/2}})<\infty$. 

\subsubsection*{ Step 1. Controlling $\|Z_{tt}(t)\|_{L^2}$ and $\|Z_t(t)\|_{L^2}$ by $\frak E$ and the initial data} We start with $\|Z_{tt}(t)\|_{L^2}$. By a change of the variables,
\begin{equation}\label{eq:200}
\frac d{dt}\|Z_{tt}(t)\|_{L^2}^2=\frac d{dt}\int |z_{tt}|^2 h_\alpha\,d\alpha=2\Re \int z_{tt}\bar{z}_{ttt} h_\alpha\,d\alpha+ 2\int |z_{tt}|^2 \frac{h_{t\alpha}}{h_\alpha}h_\alpha\,d\alpha;
\end{equation}
we estimate
\begin{equation}\label{eq:25}
\int |z_{tt}|^2 \frac{h_{t\alpha}}{h_\alpha}h_\alpha\,d\alpha\le \nm{\frac{ h_{t\alpha}}{h_\alpha}}_{L^\infty}\|Z_{tt}(t)\|_{L^2}^2.
\end{equation}
Switching back to the Riemann mapping variable and using \eqref{quasi-r} gives
\begin{equation}\label{eq:21}
\begin{aligned}
\int & z_{tt}\bar{z}_{ttt} h_\alpha\,d\alpha= \int Z_{tt}\bar{Z}_{ttt} \,d\alpha'\\&=
-i\int Z_{tt}\mathcal A \bar Z_{t,\alpha'}\,d\alpha'+\int Z_{tt}\frac{\frak a_t}{\frak a}\circ h^{-1} (\bar Z_{tt}-i)\,d\alpha'=I+II
\end{aligned}
\end{equation}
Replacing $\mathcal A:=\frac{A_1}{|Z_{,\alpha'}|^2}$, we estimate $I$ by
\begin{equation}\label{eq:22}
|I|\le \|A_1\|_{L^\infty}\nm{\frac1{Z_{,\alpha'}}}_{L^\infty}^2  \|Z_{tt}\|_{L^2} \|Z_{t,\alpha'}\|_{L^2}
\end{equation}
In $II$ we estimate $\nm{\frac{\frak a_t}{\frak a}\circ h^{-1} }_{L^2}$ by \eqref{at}, where we rewrite $D_{\alpha'}\bar Z_t:=\frac1{Z_{,\alpha'}}\bar Z_{t,\alpha'}$. Using \eqref{3.21}, \eqref{3.22}
and \eqref{eq:b12} yields
\begin{equation}\label{eq:23}
\nm{\frac{\frak a_t}{\frak a}\circ h^{-1} }_{L^2}\lesssim \|Z_{t,\alpha'}\|_{L^2}\|Z_{tt}\|_{L^\infty}+\|Z_{t,\alpha'}\|_{L^2}^3\nm{\frac1{Z_{,\alpha'}}}_{L^\infty},
\end{equation}
so
\begin{equation}\label{eq:24}
|II|
\lesssim
(\|Z_{t,\alpha'}\|_{L^2}\|Z_{tt}\|_{L^\infty}+\|Z_{t,\alpha'}\|_{L^2}^3\nm{\frac1{Z_{,\alpha'}}}_{L^\infty}) \|Z_{tt}\|_{L^2}(\|Z_{tt}\|_{L^\infty}+1). 
\end{equation}
Sum up the above estimates and apply Appendix~\ref{quantities}, we arrive at
$$\frac d{dt}\|Z_{tt}(t)\|_{L^2}^2\le c(\frak E(t))\|Z_{tt}(t)\|_{L^2}^2+c(\frak E(t)).$$
Consequently by Gronwall, 
\begin{equation}\label{eq:249}
\sup_{[0, T^*)}\|Z_{tt}(t)\|_{L^2}\le c(\|Z_{tt}(0)\|_{L^2}, M)<\infty.
\end{equation}

Changing to the Lagrangian coordinate, we have
$$\int |Z_t(\alpha',t)|^2\,d\alpha' =\int |z_t(\alpha, t)|^2 h_\alpha(\alpha,t)\,d\alpha,$$
so
\begin{equation}\label{eq:247}
\frac d{dt}\int |z_t|^2 h_\alpha\,d\alpha = 2\Re \int z_t\bar z_{tt}h_\alpha\,d\alpha+\int |z_t|^2 h_\alpha\frac{h_{t\alpha}}{h_\alpha}\,d\alpha.
\end{equation}
Using Cauchy-Schwarz and  changing back to the Riemann mapping variable, 
\begin{equation}\label{eq:248}
\frac d{dt}\int |z_t|^2 h_\alpha\,d\alpha \le  2\|Z_t(t)\|_{L^2}\|Z_{tt}(t)\|_{L^2}+\|(h_t\circ h^{-1})_{\alpha'}\|_{L^\infty}\|Z_t(t)\|_{L^2}^2,
\end{equation}
therefore  
\begin{equation}\label{eq:250}
\frac d{dt}\|Z_{t}(t)\|_{L^2}^2\le C(\frak E(t)) \|Z_{t}(t)\|_{L^2}^2+\|Z_{tt}(t)\|_{L^2}^2,
\end{equation}
by Appendix~\ref{quantities}. Consequently by Gronwall's inequality and \eqref{eq:249}, 
\begin{equation}\label{eq:251}
\sup_{t\in [0, T^*)} \|Z_{t}(t)\|_{L^2}^2\le C(\|Z_t(0)\|_{L^2}, \|Z_{tt}(0)\|_{L^2}, M)<\infty.
\end{equation}

\subsubsection*{Step 2. Controlling $\|Z_{,\alpha'}\|_{L^\infty}$} We know
$$Z_{,\alpha'}\circ h=\frac {z_\alpha}{h_\alpha},$$
  and
$$\frac d{dt} \abs{\frac {z_\alpha}{h_\alpha}}^2=2\abs{\frac {z_\alpha}{h_\alpha}}^2\Re (D_\alpha z_t-\frac{h_{t\alpha}}{h_\alpha}),$$
so by Appendix~\ref{quantities},
$$\frac d{dt} \abs{\frac {z_\alpha}{h_\alpha}}^2\le C(\frak E) \abs{\frac {z_\alpha}{h_\alpha}}^2$$ 
therefore 
\begin{equation}\label{eq:252}
\sup_{t\in [0, T^*)}\|Z_{,\alpha'}(t)\|_{L^\infty}^2\le \|Z_{,\alpha'}(0)\|_{L^\infty}^2 e^{C(M)T^*}<\infty.
\end{equation}

\subsubsection*{Step 3. Controlling $\|Z_t(t)\|_{H^{3+1/2}}+\|Z_{tt}(t)\|_{H^3}$} Taking $\sup$ over $[0, T^*)$ on \eqref{step1-2} gives
\begin{equation}
\begin{aligned}
\sup_{t\in [0, T^*)} E_2(t)&\le E_2(0)e^{ p_1(M)T^*}:=M_2<\infty;\\
\sup_{t\in [0, T^*)} E_3(t)&\le (E_3(0)+ p_3(M, M_2) T^*)e^{ p_2(M, M_2) T^*}:=M_3<\infty,
\end{aligned}
\end{equation}
By \eqref{eq:222}, \eqref{eq:249},
\begin{equation}\label{eq:253}
\sup_{[0, T^*)}\|Z_{tt}(t)\|_{H^3}\lesssim \sup_{[0, T^*)}(\|\partial_{\alpha'}^3 Z_{tt}(t)\|_{L^2}+ \|Z_{tt}(t)\|_{L^2})<\infty.
\end{equation}

Now by \eqref{Hhalf},
$$\|\partial_{\alpha'}^3 \bar Z_{t}\|_{\dot H^{1/2}}\lesssim \|Z_{,\alpha'}\|_{L^\infty}(\|\frac1{Z_{,\alpha'}}\partial_{\alpha'}^3 \bar Z_{t}\|_{\dot H^{1/2}}+\|\partial_{\alpha'}\frac1{Z_{,\alpha'}}\|_{L^2}\|\partial_{\alpha'}^3 \bar Z_{t}\|_{L^2}).
$$
We know by \eqref{eq:207} and Appendix~\ref{quantities}, 
$$\|\partial_{\alpha'}^3 \bar Z_{t}\|_{L^2},\ \  \|\frac1{Z_{,\alpha'}}\partial_{\alpha'}^3 \bar Z_{t}\|_{\dot H^{1/2}}\le C(\frak E)E_3,\qquad \|\partial_{\alpha'}\frac1{Z_{,\alpha'}}\|_{L^2}\le C(\frak E);$$
so using \eqref{eq:252} we have
\begin{equation}
\sup_{[0, T^*)}\|\partial_{\alpha'}^3 \bar Z_{t}\|_{\dot H^{1/2}}\le \|Z_{,\alpha'}(0)\|_{L^\infty}^2 e^{C(M)T^*} C(M)M_3<\infty
\end{equation}
Combine with \eqref{eq:251}, we have
\begin{equation}
\sup_{[0, T^*)}\| Z_{t}(t)\|_{H^{3+1/2}}<\infty.
\end{equation}
By Proposition~\ref{prop:local-s} this brings us a contradiction. This finishes the proof for Theorem~\ref{blow-up}.

\end{proof}

\section{The proof of Theorem~\ref{th:local}}\label{proof2}
We prove Theorem~\ref{th:local} by mollifying the initial data by the Poisson Kernel and approximating.
We denote $z'=x'+iy'$, where $x', y'\in\mathbb R$. $f\ast g$ is the convolution in the spatial variable.
 
\subsection{The initial data}\label{ID}
 Let $F(z', 0)$ be the initial fluid velocity in the Riemann mapping coordinate, $\Psi(z',0):P_-\to\Omega(0)$ be the Riemann mapping as given in \S\ref{id} with $Z(\alpha', 0)=\Psi(\alpha', 0)$  the initial interface. We note that by the assumption 
$$
\begin{aligned}&\sup_{y'<0}\|\partial_{z'}(\frac1{\Psi_{z'}(z',0)})\|_{L^2(\mathbb R, dx')}\le \mathcal E_1(0)<\infty,\quad \sup_{y'<0}\|\frac1{\Psi_{z'}(z',0)}-1\|_{L^2(\mathbb R, dx')}\le c_0<\infty;\\&
\sup_{y'<0}\|F_{z'}(z',0)\|_{L^2(\mathbb R, dx')}\le \mathcal E_1(0)<\infty,\quad \sup_{y'<0}\|F(z',0)\|_{L^2(\mathbb R, dx')}\le c_0<\infty,
\end{aligned}
$$
$ \frac1{\Psi_{z'}}(\cdot, 0)$, $F(\cdot, 0)$ can be extended continuously onto $\bar P_-$. We denote their boundary values by  $\frac1{\Psi_{z'}}(\alpha', 0)$ and $F(\alpha',0)$.
So $Z(\cdot,0)= \Psi(\cdot, 0)$ is continuous differentiable on the open set where $\frac1{\Psi_{z'}}(\alpha', 0)\ne 0$, and $\frac1{\Psi_{z'}}(\alpha', 0)=\frac1{Z_{,\alpha'}(\alpha', 0)}$ where
$\frac1{\Psi_{z'}}(\alpha', 0)\ne 0$.
By $\frac1{\Psi_{z'}}(\cdot, 0)-1\in H^1(\mathbb R)$ and Sobolev embedding, there is $N>0$ sufficiently large, such that for $|\alpha'|\ge N$, $|\frac1{\Psi_{z'}}(\alpha', 0)-1| \le 1/2$, so $Z=Z(\cdot, 0)$ is continuous differentiable on $(-\infty, -N)\cup (N, \infty)$, with $|Z_{,\alpha'}(\alpha', 0)|\le 2$, for all $ |\alpha'|\ge N$. Moreover, $Z_{,\alpha'}(\cdot, 0)-1\in H^1\{(-\infty, -N)\cup (N, \infty)\}$.

\subsection{The mollified data and the approximate solutions}\label{mo-ap}
Let $\epsilon>0$. We take
\begin{equation}\label{m-id}
\begin{aligned}
Z^\epsilon(\alpha', 0)&=\Psi(\alpha'-\epsilon i, 0),\quad \bar Z^\epsilon_t(\alpha', 0)=F(\alpha'-\epsilon i, 0),\quad
h^\epsilon(\alpha,0)=\alpha,\\& F^\epsilon(z',0)=F(z'-\epsilon i, 0),\quad \Psi^\epsilon(z',0)=\Psi(z'-\epsilon i,0).
\end{aligned}
\end{equation}
Notice that $F^\epsilon(\cdot, 0)$, $\Psi^\epsilon(\cdot, 0)$ are holomorphic on $P_-$, $Z^\epsilon(0)$ satisfies \eqref{interface-holo} and $\bar Z^\epsilon_t(0)=\mathbb H \bar Z^\epsilon_t( 0)$. Let $Z_{tt}^\epsilon(0)$ be given by \eqref{interface-a1}. 
It is clear $Z^\epsilon(0)$, $ Z_t^\epsilon(0)$ and $Z_{tt}^\epsilon(0)$ satisfy the assumption of Theorem~\ref{blow-up}. Let
$Z^\epsilon(t):=Z^\epsilon(\cdot,t)$ be the solution as given by Theorem~\ref{blow-up},  with the homeomorphism $h^\epsilon(t)=h^\epsilon(\cdot,t):\mathbb R\to\mathbb R$, and $z^\epsilon(\alpha,t)=Z^\epsilon(h^\epsilon(\alpha,t),t)$. We know $z_t^\epsilon(\alpha,t)=Z_t^\epsilon(h^\epsilon(\alpha,t),t)$.
Let  $$F^\epsilon(x'+iy', t)=K_{y'}\ast \bar Z^\epsilon_t(x', t),\quad \Psi_{z'}^\epsilon(x'+iy', t) =K_{y'}\ast Z^\epsilon_{,\alpha'}(x',t),\quad \Psi^\epsilon(\cdot, t)$$
be the holomorphic functions on $P_-$ with boundary values $\bar Z^\epsilon_t( t)$, $Z^\epsilon_{,\alpha'}(t)$ and $Z^\epsilon(t)$;
$$\frac1{\Psi_{z'}^\epsilon}(x'+iy', t) =K_{y'}\ast \frac1{Z^\epsilon_{,\alpha'}}(x',t)$$
by uniqueness.\footnote{By the maximum principle, $\big(K_{y'}\ast \frac1{Z^\epsilon_{,\alpha'}}\big)\big(K_{y'}\ast {Z^\epsilon_{,\alpha'}}\big)\equiv 1$ on $P_-$. }
We denote the energy functional $\mathcal E$  for $Z^\epsilon(t)$, $\bar Z_t^\epsilon(t)$ by $\mathcal E^\epsilon(t)$ and the energy functional $\mathcal E_1$ for $F^\epsilon(t)$, $\Psi^\epsilon(t)$ by $\mathcal E^\epsilon_1(t)$. It is clear $\mathcal E^\epsilon(0)=\mathcal E^\epsilon_1(0)\le \mathcal E_1(0)$.
By Theorem~\ref{blow-up}, Theorem~\ref{prop:a priori} and Proposition~\ref{prop:energy-eq}, there exists $T_0>0$, $T_0$ depends only on $\mathcal E_1(0)$, such that on $[0, T_0]$, the system \eqref{interface-r}-\eqref{interface-holo}-\eqref{b}-\eqref{a1} has a unique solution $Z^\epsilon=Z^\epsilon(\cdot,t)$, satisfying $(Z^\epsilon_t, Z^\epsilon_{tt}, \frac1{Z_{,\alpha'}^\epsilon}-1)\in C([0, T_0], H^{s+1/2}(\mathbb R)\times H^{s}(\mathbb R)\times H^{s}(\mathbb R))$ for $s>4$, and 
\begin{equation}\label{eq:400}
\sup_{[0, T_0]}\mathcal E_1^\epsilon(t)=\sup_{[0, T_0]}\mathcal E^\epsilon(t)\le M(\mathcal E_1(0))<\infty.
\end{equation}
Moreover by \eqref{interface-a1}, \eqref{eq:249} and \eqref{eq:251},  
\begin{equation}\label{eq:401}
\sup_{[0, T_0]}(\|Z^\epsilon_t(t)\|_{L^2}+\|Z^\epsilon_{tt}(t)\|_{L^2}+\|\frac1{Z^\epsilon_{,\alpha'}(t)}-1\|_{L^2})\le c(c_0, \mathcal E_1(0)),
\end{equation}
so there is a constant $C_0:=C(c_0, \mathcal E_1(0))>0$, such that 
\begin{equation}\label{eq:402}
\sup_{[0,T_0]}\{\sup_{y'<0}\|F^\epsilon(x'+iy', t)\|_{L^2(\mathbb R, dx')}+\sup_{y'<0}\|\frac1{\Psi^\epsilon_{z'}(x'+iy',t)}-1\|_{L^2(\mathbb R,dx')}\}<C_0<\infty.
\end{equation}

\subsection{Uniformly bounded quantities}\label{ubound}
We would like to apply some compactness results to pass to the limits of the various quantities for the water waves. It is necessary to understand  the boundedness properties of these quantities. 

Let $b^\epsilon:=h_t^\epsilon\circ (h^\epsilon)^{-1}=2\Re Z_t^\epsilon+\Re[Z_t^\epsilon,\mathbb H](\frac1{Z^\epsilon_{,\alpha'}}-1)$ be as given by \eqref{b}. By \eqref{eq:b13}, 
\begin{equation}\label{eq:408}
\|b^\epsilon(t)\|_{L^\infty}=\|h_t^\epsilon(t)\|_{L^\infty}\lesssim \| Z_t^\epsilon(t)\|_{L^\infty}+\|Z_{t,\alpha'}^\epsilon(t)\|_{L^2}\|\frac1{Z^\epsilon_{,\alpha'}}(t)-1\|_{L^2}.
\end{equation}
Using \eqref{interface-e-2} to rewrite $b^\epsilon= \Re(I-\mathbb H)(Z_t^\epsilon\frac1{Z^\epsilon_{,\alpha'}})$, differentiating to get  
\begin{equation}\label{eq:412}
\|b^\epsilon_{\alpha'}(t)\|_{L^2}\lesssim \|Z_{t,\alpha'}^\epsilon(t)\|_{L^2}\|\frac1{Z^\epsilon_{,\alpha'}}(t)\|_{L^\infty}+\|Z_t^\epsilon(t)\|_{L^\infty}\|\partial_{\alpha'}\frac1{Z^\epsilon_{,\alpha'}}(t)\|_{L^2}.
\end{equation}
We know $h^\epsilon$ satisfies
\begin{equation}\label{eq:405}
\begin{cases}
\frac d{dt} h^\epsilon=b^\epsilon(h^\epsilon, t);\\
h^\epsilon(\alpha, 0)=\alpha.
\end{cases}
\end{equation}
Differentiating \eqref{eq:405} gives
\begin{equation}\label{eq:406}
  \begin{cases}
\frac d{dt} h^\epsilon_{\alpha}=b^\epsilon_{\alpha'}(h^\epsilon, t)h^\epsilon_{\alpha};\\
h^\epsilon_{\alpha}(\alpha, 0)=1
\end{cases}
\end{equation}
therefore
\begin{equation}\label{eq:407}
 e^{- t\sup_{[0, t]}\|b^\epsilon_{\alpha'}(s)\|_{L^\infty} }\le     h^\epsilon_{\alpha}(\alpha, t)=e^{\int_0^t b^\epsilon_{\alpha'}(h^\epsilon, s)\,ds}\le e^{ t\sup_{[0, t]}\|b^\epsilon_{\alpha'}(s)\|_{L^\infty} }.
\end{equation}
Now  by \eqref{b}, \eqref{eq:c14'}, with an application of \eqref{eq:b13} and \eqref{eq:b15},
\begin{equation}\label{eq:409}
\begin{aligned}
&\|U_{h^\epsilon}^{-1}\partial_t U_{h^\epsilon} b^\epsilon(t)\|_{L^\infty}\lesssim \|Z^\epsilon_{tt}(t)\|_{L^\infty}+\|Z^\epsilon_{tt,\alpha'}(t)\|_{L^2}\|\frac1{Z^\epsilon_{,\alpha'}}(t)-1\|_{L^2}\\&+\|Z^\epsilon_{t,\alpha'}(t)\|_{L^2}(\|U_{h^\epsilon}^{-1}\partial_tU_{h^\epsilon} \frac1{Z^\epsilon_{,\alpha'}}(t)\|_{L^2}+\|b^\epsilon_{\alpha'}(t)\|_{L^\infty}\|\frac1{Z^\epsilon_{,\alpha'}}(t)-1\|_{L^2});
\end{aligned}
\end{equation}
where $U_{h^\epsilon}^{-1}\partial_tU_{h^\epsilon} \frac1{Z^\epsilon_{,\alpha'}}=\frac1{Z^\epsilon_{,\alpha'}}((h_t^\epsilon\circ (h^\epsilon)^{-1})_{\alpha'}-D_{\alpha'}Z^\epsilon_{t})$ gives
\begin{equation}\label{eq:411}
\begin{aligned}
&\|U_{h^\epsilon}^{-1}\partial_tU_{h^\epsilon} \frac1{Z^\epsilon_{,\alpha'}}(t)\|_{L^2}\le \|\frac1{Z^\epsilon_{,\alpha'}}(t)\|_{L^\infty}(\|b^\epsilon_{\alpha'}(t)\|_{L^2}+\|\frac1{Z^\epsilon_{,\alpha'}}(t)\|_{L^\infty}\|Z^\epsilon_{t,\alpha'}(t)\|_{L^2})\\&
\|U_{h^\epsilon}^{-1}\partial_tU_{h^\epsilon} \frac1{Z^\epsilon_{,\alpha'}}(t)\|_{L^\infty}\le \|\frac1{Z^\epsilon_{,\alpha'}}(t)\|_{L^\infty}(\|b^\epsilon_{\alpha'}(t)\|_{L^\infty}+\|D_{\alpha'}Z^\epsilon_{t}(t)\|_{L^\infty})
\end{aligned}
\end{equation}
and $U_{h^\epsilon}^{-1}\partial_t U_{h^\epsilon} =\partial_t+b^\epsilon \partial_{\alpha'}$ gives
\begin{equation}\label{eq:410}
\|\partial_tb^\epsilon(t)\|_{L^\infty}\le \|U_{h^\epsilon}^{-1}\partial_t U_{h^\epsilon} b^\epsilon(t)\|_{L^\infty}+
\|b^\epsilon(t)\|_{L^\infty}\| b^\epsilon_{\alpha'}(t)\|_{L^\infty}.
\end{equation}
Finally, differentiating \eqref{interface-e2}  gives
$z^\epsilon_{ttt}=(z^\epsilon_{tt}+i)(D_{\alpha}z^\epsilon_t+\frac{\frak a^\epsilon_t}{\frak a^\epsilon})$, so 
\begin{equation}\label{eq:4400}
\|z^\epsilon_{ttt}(t)\|_{L^\infty}\le \|z^\epsilon_{tt}(t)+i\|_{L^\infty}(\|D_{\alpha}z^\epsilon_t(t)\|_{L^\infty}+\|\frac{\frak a^\epsilon_t}{\frak a^\epsilon}(t)\|_{L^\infty}).
\end{equation}

Let $M(\mathcal E_1(0))$, $c(c_0,\mathcal E_1(0))$, $C_0$ be the bounds in \eqref{eq:400}, \eqref{eq:401} and \eqref{eq:402}. 
By Proposition~\ref{prop:energy-eq}, Sobolev embedding, Appendix~\ref{quantities} and \eqref{eq:411},  the following quantities are uniformly bounded with bounds depending only on $M(\mathcal E_1(0))$, $c(c_0,\mathcal E_1(0))$, $C_0$:
\begin{equation}\label{eq:404}
\begin{aligned}
&\sup_{[0, T_0]}\|Z^\epsilon_t(t)\|_{L^\infty}, \sup_{[0, T_0]}\|Z^\epsilon_{t,\alpha'}(t)\|_{L^2}, \sup_{[0, T_0]}\|Z^\epsilon_{tt}(t)\|_{L^\infty}, \sup_{[0, T_0]}\|Z^\epsilon_{tt,\alpha'}(t)\|_{L^2},\\&
\sup_{[0, T_0]}\|\frac1{Z^\epsilon_{,\alpha'}}(t)\|_{L^\infty}, \sup_{[0, T_0]}\|\partial_{\alpha'}(\frac1{Z^\epsilon_{,\alpha'}})(t)\|_{L^2}, \sup_{[0, T_0]}\|U_{h^\epsilon}^{-1}\partial_t U_{h^\epsilon}\frac1{Z^\epsilon_{,\alpha'}}(t)\|_{L^\infty};
\end{aligned}
\end{equation}
and with a change of the variables and \eqref{eq:407}, \eqref{eq:4400} and Appendix~\ref{quantities},
\begin{equation}\label{eq:414}
\begin{aligned}
&\sup_{[0, T_0]}\|z^\epsilon_t(t)\|_{L^\infty}+ \sup_{[0, T_0]}\|z^\epsilon_{t\alpha}(t)\|_{L^2}+ \sup_{[0, T_0]}\|z^\epsilon_{tt}(t)\|_{L^\infty}\le C(c_0, \mathcal E_1(0)), \\&
\sup_{[0, T_0]}\|\frac{h^\epsilon_{\alpha}}{z^\epsilon_{\alpha}}(t)\|_{L^\infty}+ \sup_{[0, T_0]}\|\partial_{\alpha}(\frac {h^\epsilon_{\alpha}}{z^\epsilon_{\alpha}})(t)\|_{L^2}+ \sup_{[0, T_0]}\|\partial_t \frac{h^\epsilon_\alpha}{z^\epsilon_{\alpha}}(t)\|_{L^\infty}\le C(c_0, \mathcal E_1(0)),\\&
\sup_{[0, T_0]}\|z^\epsilon_{tt}(t)\|_{L^\infty}+ \sup_{[0, T_0]}\|z^\epsilon_{tt\alpha}(t)\|_{L^2}+ \sup_{[0, T_0]}\|z^\epsilon_{ttt}(t)\|_{L^\infty}\le C(c_0, \mathcal E_1(0)).
\end{aligned}
\end{equation}
Furthermore, by the estimates in \eqref{eq:408}--\eqref{eq:410}, using \eqref{eq:401} \eqref{eq:404} and Appendix~\ref{quantities},  the following quantities are uniformly bounded: 
\begin{equation}\label{eq:415}
\begin{aligned}
&\sup_{[0, T_0]}\|b^\epsilon(t)\|_{L^\infty}+ \sup_{[0, T_0]}\|b^\epsilon_{\alpha'}(t)\|_{L^\infty}+ \sup_{[0, T_0]}\|b^\epsilon_t(t)\|_{L^\infty}\le C(c_0, \mathcal E_1(0))\\&
 \sup_{[0, T_0]}\|h^\epsilon_{\alpha}(t)\|_{L^\infty}+ \sup_{[0, T_0]}\|h^\epsilon_t(t)\|_{L^\infty}\le C(c_0, \mathcal E_1(0))
\end{aligned}
\end{equation}
In particular, by \eqref{eq:407} and Appendix~\ref{quantities}, there are $c_1, c_2>0$, depending only on $c_0$ and $\mathcal E_1(0)$, such that
\begin{equation}\label{eq:416}
0<c_1\le\frac{h^\epsilon(\alpha,t)-h^\epsilon(\beta,t)}{\alpha-\beta}\le c_2<\infty,\qquad \forall \alpha,\beta\in \mathbb R, \ t\in [0, T_0].
\end{equation}

\subsection{Some useful compactness results} Here we give two compactness results that we will use to pass to the limits.

\begin{lemma}\label{lemma1} Let $\{f_n\}$ be a sequence of smooth functions on $\mathbb R\times [0, T]$. Let $1<p\le\infty$. Assume that there is a constant $C$, independent of $n$, such that 
\begin{equation}
\sup_{[0, T]}\|f_n(t)\|_{L^\infty}+ \sup_{[0, T]}\|{\partial_x f_n}(t)\|_{L^p}+ \sup_{[0, T]}\|\partial_t f_n(t)\|_{L^\infty}\le C.
\end{equation}
Then there is a function $f$, continuous and bounded on $\mathbb R\times [0, T]$,  and a subsequence $\{f_{n_j}\}$, such that $f_{n_j}\to f$ uniformly on compact subsets of $\mathbb R\times [0, T]$.
\end{lemma}
Lemma~\ref{lemma1} is an easy consequence of Arzela-Ascoli Theorem, we omit the proof.

\begin{lemma}\label{lemma2} 
Assume that $f_n\to f$ uniformly on compact subsets of $\mathbb R\times [0, T]$, and assume there is a constant $C$, such that $\sup_{n}\|f_n\|_{L^\infty(\mathbb R\times [0, T])}\le C$. Then $K_{y'}\ast f_n$ converges uniformly  to $K_{y'}\ast f$ on compact subsets of $\bar P_-\times [0, T]$.
\end{lemma}
The proof follows easily by considering the convolution on two sets $|x'|<N$, and $|x'|\ge N$. We omit the proof.

\begin{definition} We write
\begin{equation}\label{unif-notation}
f_n\Rightarrow f\qquad \text{on }E
\end{equation}
if $f_n$ converge uniformly to $f$ on compact subsets of $E$.
\end{definition}

\subsection{Passing to the limit} Notice that 
$h^\epsilon(\alpha, t)-\alpha=\int_0^t h^\epsilon_t(\alpha, s)\,ds$, so 
\begin{equation}\label{eq:425}
\sup_{\mathbb R\times [0, T_0]} |h^\epsilon(\alpha, t)-\alpha|\le T_0\sup_{[0, T_0]}\|h^\epsilon_t(t)\|_{L^\infty}\le T_0C(c_0,\mathcal E_1(0))<\infty.
\end{equation}
By Lemma~\ref{lemma1}, there is a subsequence $\epsilon_j\to 0$, which we still write as $\epsilon$ instead of $\epsilon_j$, and  functions $b$, $h-\alpha$, $w$, $u$, $q:=w_t$,  continuous and bounded on $\mathbb R\times [0, T_0]$, such that
\begin{equation}\label{eq:417}
b^\epsilon\Rightarrow b, \quad h^\epsilon\Rightarrow h, \quad z^\epsilon_t\Rightarrow w,\quad \frac{h^\epsilon_{\alpha}}{z^\epsilon_\alpha}\Rightarrow u,\quad z^\epsilon_{tt}\Rightarrow q,\qquad \text{on } \mathbb R\times [0, T_0],
\end{equation}
as $ \epsilon=\epsilon_j\to 0$. Moreover by \eqref{eq:416}, 
\begin{equation}\label{eq:420}
0<c_1\le\frac{h(\alpha,t)-h(\beta,t)}{\alpha-\beta}\le c_2<\infty,\qquad \forall \alpha,\beta\in \mathbb R, \ t\in [0, T_0];
\end{equation}
hence $h(\cdot,t):\mathbb R\to\mathbb R$ is a homeomorphism, and 
\begin{equation}\label{eq:418}
(h^\epsilon)^{-1}\Rightarrow h^{-1}\qquad\text{on }\mathbb R\times [0, T_0], \quad \text{ as } \epsilon=\epsilon_j\to 0.
\end{equation}
This gives
\begin{equation}\label{eq:419}
\bar Z_t^\epsilon\Rightarrow w\circ h^{-1},\qquad \frac1{Z^\epsilon_{,\alpha'}}\Rightarrow u\circ h^{-1}, \quad\bar Z_{tt}^\epsilon\Rightarrow w_t\circ h^{-1},\qquad\text{on }\mathbb R\times [0, T_0]
\end{equation}
as  $ \epsilon=\epsilon_j\to 0$. 
Now 
 \begin{equation}\label{eq:421}
F^\epsilon(z',t)=K_{y'}\ast \bar Z_t^\epsilon,\qquad \frac1{\Psi^\epsilon_{z'}}(z',t)=K_{y'}\ast \frac1{Z^\epsilon_{,\alpha'}}.
\end{equation}
Let $F(z',t)=   K_{y'}\ast (w\circ h^{-1})(x', t)$, $\Lambda(z',t)=K_{y'}\ast (u\circ h^{-1})(x',t)$.
By Lemma~\ref{lemma2}, 
\begin{equation}\label{eq:422}
F^\epsilon(z',t)\Rightarrow F(z',t),\qquad \frac1{\Psi^\epsilon_{z'}}(z',t)\Rightarrow \Lambda(z',t)\qquad\text{on }\bar P_-\times [0, T_0];
\end{equation}
as  $ \epsilon=\epsilon_j\to 0$. Moreover $F(\cdot,t)$, $\Lambda(\cdot,t)$ are holomorphic on $P_-$ for each $t\in [0, T_0]$, and continuous on $\bar P_-\times [0, T]$. Furthermore applying the Cauchy integral formula to the first limit in \eqref{eq:422} yields
\begin{equation}\label{eq:430}
F^\epsilon_{z'}(z',t)\Rightarrow F_{z'}(z',t) \qquad\text{on } P_-\times [0, T_0].
\end{equation}
as  $ \epsilon=\epsilon_j\to 0$.

\subsubsection*{Step 1. The limit of $\Psi^\epsilon$}\label{step4.1}
We consider the limit of $\Psi^{\epsilon}$, as $\epsilon=\epsilon_j\to 0$.  We know
\begin{equation}\label{eq:423}
\begin{aligned}
z^\epsilon(\alpha,t)&=z^\epsilon(\alpha,0)+\int_0^t z_t^\epsilon(\alpha, s)\,ds\\&
=\Psi(\alpha-\epsilon i, 0)+\int_0^t z_t^\epsilon(\alpha, s)\,ds,
\end{aligned}
\end{equation}
therefore
\begin{equation}\label{eq:424}
\begin{aligned}
Z^\epsilon(\alpha',t)-Z^\epsilon(\alpha',0)&
=\Psi((h^\epsilon)^{-1}(\alpha',t)-\epsilon i, 0)-\Psi(\alpha'-\epsilon i, 0)\\&+\int_0^t z_t^\epsilon((h^\epsilon)^{-1}(\alpha',t)   , s)\,ds.
\end{aligned}
\end{equation}
Let
\begin{equation}\label{eq:431}
W^\epsilon(\alpha',t):=\Psi((h^\epsilon)^{-1}(\alpha',t)-\epsilon i, 0)-\Psi(\alpha'-\epsilon i, 0)+\int_0^t z_t^\epsilon((h^\epsilon)^{-1}(\alpha',t)   , s)\,ds.
\end{equation}
Observe $Z^\epsilon(\alpha',t)-Z^\epsilon(\alpha',0)$ is the boundary value of the holomorphic function
$\Psi^\epsilon(z', t)-\Psi^\epsilon(z', 0)$. By \eqref{eq:417} and \eqref{eq:418}, $\int_0^t z_t^\epsilon((h^\epsilon)^{-1}(\alpha',t), s)\,ds\to \int_0^t w(h^{-1}(\alpha',t), s)\,ds$ uniformly on compact subsets of $\mathbb R\times [0, T_0]$, and by \eqref{eq:414}, $\int_0^t z_t^\epsilon((h^\epsilon)^{-1}(\alpha',t), s)\,ds$ is continuous and uniformly bounded in $L^\infty(\mathbb R\times [0, T_0])$.  
By the assumptions $\lim_{z'\to 0}\Psi_{z'}(z',0)=1$, $\Psi(\cdot, 0)$ is continuous on $\bar P_-$ and \eqref{eq:425}, \eqref{eq:418}, $$\Psi((h^\epsilon)^{-1}(\alpha',t)-\epsilon i, 0)-\Psi(\alpha'-\epsilon i, 0)$$ is continuous and uniformly bounded in $L^\infty(\mathbb R\times [0, T_0])$ for $0<\epsilon<1$, and converges uniformly on compact subsets of $\mathbb R\times [0, T_0]$, as $\epsilon=\epsilon_j\to 0$. This gives\footnote{Because $W^\epsilon(\cdot,t)$ and $\partial_{\alpha'}W^\epsilon(\cdot,t):=Z^\epsilon_{,\alpha'}(\alpha',t)-Z^\epsilon_{,\alpha'}(\alpha',0)$ are continuous and bounded on $\mathbb R$, $\Psi^\epsilon_{z'}(z',t)-\Psi^\epsilon_{z'}(z',0)=K_{y'}\ast (\partial_{\alpha'}W^\epsilon)(x',t)=\partial_{z'}K_{y'}\ast W^\epsilon(x',t)$.}
\begin{equation}\label{eq:426}
\Psi^\epsilon(z', t)-\Psi^\epsilon(z', 0)=K_{y'}\ast W^\epsilon(x',t)
\end{equation}
and by Lemma~\ref{lemma2}, $\Psi^\epsilon(z', t)-\Psi^\epsilon(z', 0)$ converges uniformly on compact subsets of $\bar P_-\times [0, T_0]$ to a function that is holomorphic on $P_-$ for every $t\in [0, T_0]$ and continuous on $\bar P_-\times [0, T_0]$. Therefore there is a function $\Psi(\cdot,t)$, holomorphic on $P_-$ for every $t\in [0, T_0]$  and continuous on $\bar P_-\times [0, T_0]$, such that
\begin{equation}\label{eq:427}
\Psi^\epsilon(z',t)\Rightarrow \Psi(z',t)\qquad\text{on }\bar P_-\times [0, T_0]
\end{equation}
as $\epsilon=\epsilon_j\to 0$; as a consequence of the Cauchy integral formula, 
\begin{equation}\label{eq:428}
\Psi^\epsilon_{z'}(z',t)\Rightarrow \Psi_{z'}(z',t)\qquad\text{on } P_-\times [0, T_0]
\end{equation}  
as $\epsilon=\epsilon_j\to 0$. Combining with \eqref{eq:422}, we have $\Lambda(z',t)=\frac1{\Psi_{z'}(z',t)}$, so $\Psi_{z'}(z',t)\ne 0$ for all $(z',t)\in P_-\times [0, T_0]$ and
\begin{equation}\label{eq:429}
\frac1{\Psi^\epsilon_{z'}(z',t)}\Rightarrow \frac1{\Psi_{z'}(z',t)}\qquad\text{on }\bar P_-\times [0, T_0]
\end{equation}
as $\epsilon=\epsilon_j\to 0$.

Denote $Z(\alpha', t):=\Psi(\alpha', t)$, $\alpha'\in \mathbb R$, and $z(\alpha,t)=Z(h(\alpha,t),t)$. \eqref{eq:427} gives $Z^\epsilon(\alpha',t)\Rightarrow Z(\alpha',t)$, and with \eqref{eq:417} it gives
$z^\epsilon(\alpha,t)\Rightarrow z(\alpha,t)$ 
on $\mathbb R\times [0, T_0]$, as $\epsilon=\epsilon_j\to 0$.
Moreover by \eqref{eq:423}, 
$$z(\alpha', t)=z(\alpha',0)+\int_0^t w(\alpha,s)\,ds,$$ so $w=z_t$. We denote $Z_t=z_t\circ h^{-1}$. 

\subsubsection*{Step 2. The limits of $\Psi^\epsilon_t$ and $F_t^\epsilon$} Observe that by \eqref{eq:431}, for fixed $\epsilon>0$,  $\partial_t W^\epsilon(\cdot, t)$ is a bounded function on $\mathbb R\times [0, T_0]$, so by \eqref{eq:426}, $\Psi^\epsilon_t=K_{y'}\ast \partial_t W^\epsilon$ is bounded on $P_-\times [0, T_0]$. However we will not use this to pass to the limit for $\Psi^\epsilon_t$, instead, we use \eqref{c3}.

By \eqref{c3} and the above observation, since $\frac{\Psi^\epsilon_t}{\Psi^\epsilon_{z'}}$ is bounded and holomorphic on $P_-$, 
\begin{equation}\label{eq:432}
\frac{\Psi^\epsilon_t}{\Psi^\epsilon_{z'}}=K_{y'}\ast(\frac{Z^\epsilon_t}{Z^\epsilon_{,\alpha'}}-b^\epsilon).
\end{equation}
By \eqref{eq:417}, \eqref{eq:419} and Lemma~\ref{lemma2}, $\frac{\Psi^\epsilon_t}{\Psi^\epsilon_{z'}}$ converges uniformly on compact subsets of $\bar P_-\times [0, T_0]$ to a function that is holomorphic on $P_-$ for each $t\in [0, T_0]$ and continuous on $\bar P_-\times [0, T_0]$. By \eqref{eq:427}, \eqref{eq:428}, we can conclude that $\Psi$ is continuously differentiable and 
\begin{equation}\label{eq:433}
\Psi^\epsilon_t\Rightarrow \Psi_t\qquad \text{on }P_-\times [0, T_0]
\end{equation}
 as $\epsilon=\epsilon_j\to 0$.
 
Now we consider the limit of $F^\epsilon_t$ as $\epsilon=\epsilon_j\to 0$. Since for fixed $\epsilon>0$, 
$\partial_t Z_t^\epsilon=Z_{tt}^\epsilon-b^\epsilon Z_{t,\alpha'}^\epsilon$ is in $L^\infty(\mathbb R\times [0, T_0])$, by \eqref{eq:421},
\begin{equation}\label{eq:434}
F^\epsilon_t(z',t)=K_{y'}\ast  \partial_t \bar Z_t^\epsilon=K_{y'}\ast (\bar Z_{tt}^\epsilon-b^\epsilon \bar Z_{t,\alpha'}^\epsilon).
\end{equation}
By Lemma~\ref{lemma2}, $K_{y'}\ast \bar Z_{tt}^\epsilon$ converges uniformly on compact subsets of $\bar P_-\times [0, T_0]$. With a change of variables
\begin{equation}\label{eq:436}
K_{y'}\ast (b^\epsilon \bar Z_{t,\alpha'}^\epsilon)=\frac{-1}{\pi} \int\frac{y'}{(x'-h^\epsilon(\alpha,t))^2+y'^2}b^\epsilon\circ h^\epsilon(\alpha,t) \bar z^\epsilon_{t\alpha}(\alpha,t)\,d\alpha.
\end{equation}
Because \eqref{eq:417}: $z_t^\epsilon\to z_t$, $z_{tt}^\epsilon\to z_{tt}$ uniform on compact subsets of $\mathbb R\times [0, T_0]$,  and \eqref{eq:404}: $ \sup_{[0, T_0]}\|\bar z^\epsilon_{t\alpha}(t)\|_{L^2}\le C(c_0, \mathcal E_1(0))$, $ \sup_{[0, T_0]}\|\bar z^\epsilon_{tt\alpha}(t)\|_{L^2}\le C(c_0, \mathcal E_1(0))$, $\bar z_{t\alpha}$,
$\bar z_{tt\alpha}$ exist in $L^2(\mathbb R)$ for each $t\in [0, T_0]$, with $ \sup_{[0, T_0]}\|\bar z_{t\alpha}(t)\|_{L^2}\le C(c_0, \mathcal E_1(0))$, $ \sup_{[0, T_0]}\|\bar z_{tt\alpha}(t)\|_{L^2}\le C(c_0, \mathcal E_1(0))$, and by
\eqref{eq:417},  \eqref{eq:425} and $K_{y'}\ast (b^\epsilon \bar Z_{t,\alpha'}^\epsilon)$ converges point-wise on $P_-\times [0, T_0]$ to the continuous function 
$$\frac{-1}{\pi} \int\frac{y'}{(x'-h(\alpha,t))^2+y'^2}b\circ h(\alpha,t) \bar z_{t\alpha}(\alpha,t)\,d\alpha$$
as $\epsilon=\epsilon_j\to 0$ and by \eqref{eq:404}, \eqref{eq:415}, 
$$\sup_{[0, T_0]}\|F_t^\epsilon(z', t)\|_{L^\infty(\mathbb R, dx')}\le (1+\frac1{|y'|^{1/2}})C(c_0, \mathcal E_1(0)).$$
Therefore $F$ is continuously differentiable with respect to $t$,
 with $\sup_{[0, T_0]}\|F_t(z', t)\|_{L^\infty(\mathbb R, dx')}\le (1+\frac1{|y'|^{1/2}})C(c_0, \mathcal E_1(0))$ and 
\begin{equation}\label{eq:435}
F^\epsilon_t(z',t)\to F_t(z',t),\qquad \text{as }\epsilon=\epsilon_j\to 0
\end{equation}
point-wise on $P_-\times [0, T_0]$.

\subsubsection*{Step 3. The limit of $\mathfrak P^\epsilon$} By the calculation in \S\ref{general-soln}, we know $Z^\epsilon_{,\alpha}(\bar Z^\epsilon_{tt}-i)$ is the boundary value of the function $\Psi^\epsilon_{z'} 
F^\epsilon_t- {\Psi^\epsilon_t}F^\epsilon_{z'}+{\bar F^\epsilon}F^\epsilon_{z'} -i\Psi^\epsilon_{z'}$  on $\partial P_-$. Since $\Psi^\epsilon_{z'} 
F^\epsilon_t- {\Psi^\epsilon_t}F^\epsilon_{z'} -i\Psi^\epsilon_{z'}$ is holomorphic and ${\bar F^\epsilon}F^\epsilon_{z'}=\partial_{z'}({\bar F^\epsilon}F^\epsilon)$, where $\partial_{z'}=\frac12(\partial_{x'}-i\partial_{y'})$, there is a real valued function $\frak P^\epsilon$, such that
\begin{equation}\label{eq:437}
\Psi^\epsilon_{z'} 
F^\epsilon_t- {\Psi^\epsilon_t}F^\epsilon_{z'}+{\bar F^\epsilon}F^\epsilon_{z'} -i\Psi^\epsilon_{z'}=-(\partial_{x'}-i\partial_{y'})\frak P^\epsilon,\qquad\text{in }P_-;
\end{equation}
and by $Z^\epsilon_{,\alpha}(\bar Z^\epsilon_{tt}-i)=iA_1^\epsilon$, which is pure imaginary, we know 
\begin{equation}\label{eq:438}
\frak P^\epsilon=constant,\qquad \text{on }\partial P_-.
\end{equation}
Without loss of generality we take the $constant=0$. We now explore a few other properties of $\frak P^\epsilon$. Moving ${\bar F^\epsilon}F^\epsilon_{z'}=\partial_{z'}({\bar F^\epsilon}F^\epsilon)$ to the right of  \eqref{eq:437} gives
\begin{equation}\label{eq:440}
\Psi^\epsilon_{z'} 
F^\epsilon_t- {\Psi^\epsilon_t}F^\epsilon_{z'} -i\Psi^\epsilon_{z'}=-(\partial_{x'}-i\partial_{y'})(\frak P^\epsilon+ \frac12 |F^\epsilon|^2),\qquad\text{in }P_-;
\end{equation}
Applying $(\partial_{x'}+i\partial_{y'})=2\bar \partial_{z'}$ to \eqref{eq:440} yields
\begin{equation}\label{eq:439}
-\Delta  (\frak P^\epsilon+ \frac12 |F^\epsilon|^2) =0,\qquad\text{in } P_-.
\end{equation}
So $\frak P^\epsilon+ \frac12 |F^\epsilon|^2$ is a harmonic function on $P_-$ with boundary value $\frac12 |\bar Z_t^\epsilon|^2$. On the other hand, it is easy to check that $\lim_{y'\to-\infty} (\Psi^\epsilon_{z'} 
F^\epsilon_t- {\Psi^\epsilon_t}F^\epsilon_{z'} -i\Psi^\epsilon_{z'})=-i$. Therefore 
\begin{equation}\label{eq:441}
\frak P^\epsilon(z',t)=- \frac12 |F^\epsilon(z',t)|^2-y + \frac12 K_{y'}\ast (|\bar Z_t^\epsilon|^2)(x',t).
\end{equation}
By \eqref{eq:422}, \eqref{eq:419} and Lemma~\ref{lemma2}, 
\begin{equation}\label{eq:442}
\frak P^\epsilon(z',t)\Rightarrow - \frac12 |F(z',t)|^2-y + \frac12 K_{y'}\ast (|\bar Z_t|^2)(x',t),\qquad \text{on }\bar P_-\times [0, T_0] 
\end{equation}
as $\epsilon=\epsilon_j\to 0$. We write $$\frak P:=- \frac12 |F(z',t)|^2-y + \frac12 K_{y'}\ast (|\bar Z_t|^2)(x',t).$$ $\frak P$ is continuous on $\bar P_-\times [0, T_0]$ with $\frak P \in C([0, T_0], C^\infty(P_-))$, and 
\begin{equation}\label{eq:443}
\frak P=0,\qquad \text{on }\partial P_-.
\end{equation}
Moreover, since $K_{y'}\ast (|\bar Z_t^\epsilon|^2)(x',t)$ is harmonic on $P_-$, by interior derivative estimate for harmonic functions and by \eqref{eq:422}, 
\begin{equation}\label{eq:4450}
(\partial_{x'}-i\partial_{y'})\frak P^\epsilon\Rightarrow (\partial_{x'}-i\partial_{y'})\frak P\qquad\text{on }P_-\times [0, T_0]
\end{equation}
as $\epsilon=\epsilon_j\to 0$. 

\subsubsection*{Step 4. Conclusion} We now sum up Steps 1-3. We have shown that there are
functions $\Psi(\cdot, t)$ and $F(\cdot, t)$,  holomorphic on $P_-$ for each fixed $t\in [0, T_0]$,  continuous on $\bar P_-\times [0, T_0]$, and continuous differentiable on $P_-\times [0, T_0]$, with $ \frac1{\Psi_{z'}}$ continuous on $\bar P_-\times [0, T_0]$, 
 such that 
$\Psi^\epsilon \to \Psi$, $\frac1{\Psi^\epsilon_{z'}} \to \frac1{\Psi_{z'}}$, $ F^\epsilon\to F$ uniform on compact subsets of $\bar P_-\times [0, T_0]$, $\Psi^\epsilon_t \to \Psi_t$, $\Psi^\epsilon_{z'}\to \Psi_{z'}$, $F^\epsilon_{z'}\to F_{z'}$ uniform on compact subsets of $P_-\times [0, T_0]$, and $F^\epsilon_t\to F_t$ point-wise on $P_-\times [0, T_0]$, as  $\epsilon=\epsilon_j\to 0$. We have also shown there is $\frak P$,
continuous on $\bar P_-\times [0, T_0]$ with $\frak P=0$ on $\partial P_-$ and $(\partial_{x'}-i\partial_{y'})\frak P$ continuous on $P_-\times [0, T_0]$, such that $(\partial_{x'}-i\partial_{y'})\frak P^\epsilon\to (\partial_{x'}-i\partial_{y'})\frak P$ uniformly on compact subsets of $ P_-\times [0, T_0]$, as $\epsilon=\epsilon_j\to 0$. Let $\epsilon=\epsilon_j\to 0$ in equation \eqref{eq:437}, we have 
\begin{equation}\label{eq:444}
\Psi_{z'} 
F_t- {\Psi_t}F_{z'}+{\bar F}F_{z'} -i\Psi_{z'}=-(\partial_{x'}-i\partial_{y'})\frak P,\qquad\text{on }P_-\times [0, T_0].
\end{equation}
This shows $\Psi$, $F$ is a generalized solution of the water wave equation in the sense given in \S\ref{general-soln}. Furthermore because of \eqref{eq:400}, \eqref{eq:402}, letting $\epsilon=\epsilon_j\to 0$ gives 
\begin{equation}
\sup_{[0, T_0]}\mathcal E_1(t)\le M(\mathcal E_1(0))<\infty.
\end{equation}
and
\begin{equation}
\sup_{[0,T_0]}\{\sup_{y'<0}\|F(x'+iy', t)\|_{L^2(\mathbb R, dx')}+\sup_{y'<0}\|\frac1{\Psi_{z'}(x'+iy',t)}-1\|_{L^2(\mathbb R,dx')}\}<C_0<\infty.
\end{equation}

\subsection{The invertability of $\Psi(\cdot,t)$} If in addition $\Sigma(t)=\{Z=\Psi(\alpha',t):=Z(\alpha',t)\ | \ \alpha'\in\mathbb R\}$ is a Jordan curve, then because $\lim_{|\alpha'|\to\infty}\Psi_{z'}(\alpha', t)=1$,\footnote{By a similar argument as in \S\ref{ID}.} the domain $\Omega(t)$ bounded above by $\Sigma(t)$ is winded by $\Sigma(t)$ exactly once. By the argument principle, $\Psi(\cdot,t):\bar P_-\to \Omega(t)$ is one-to-one and onto, $\Psi^{-1}(\cdot,t):\Omega(t)\to P_-$ exists and is holomorphic. By the chain rule, it is easy to check \eqref{eq:444} is equivalent to
\begin{equation}\label{eq:445}
(F\circ \Psi^{-1})_t+ {\bar F\circ \Psi^{-1}}(F\circ \Psi^{-1})_{z} -i=-(\partial_{x}-i\partial_{y})(\frak P\circ \Psi^{-1}),\qquad\text{on }\Omega(t).
\end{equation}
This is the Euler equation, the first equation of \eqref{euler} in complex form. Let $\bar{\bold v}=F\circ \Psi^{-1}$, $P=\frak P\circ \Psi^{-1}$. Then $(\bold v, P)$ is a solution of the water wave equation \eqref{euler} in $\Omega(t)$, with fluid interface $\Sigma(t): Z=Z(\alpha', t)$, $\alpha'\in\mathbb R$.

\subsection{The chord-arc interfaces} Now assume at time $t=0$, the interface $Z=\Psi(\alpha',0):=Z(\alpha',0)$, $\alpha'\in\mathbb R$ is chord-arc, that is,  there is $0<\delta<1$, such that
$$\delta \int_{\alpha'}^{\beta'} |Z_{,\alpha'}(\gamma,0)|\,d\gamma\le |Z(\alpha', 0)-Z(\beta', 0)|\le \int_{\alpha'}^{\beta'} |Z_{,\alpha'}(\gamma,0)|\,d\gamma,\quad \forall -\infty<\alpha'<  \beta'<\infty.$$
 We want to show there is $T_1>0$, depending only on $\mathcal E_1(0)$, such that for $t\in [0, \min\{T_0, \frac\delta{T_1}\}]$, the interface $Z=Z(\alpha',t):=\Psi(\alpha',t)$ remains chord-arc. We begin with
 \begin{equation}\label{eq:446}
- z^\epsilon(\alpha,t)+z^\epsilon(\beta,t)+z^\epsilon(\alpha,0)-z^\epsilon(\beta,0)=\int_0^t\int_{\alpha}^\beta z^\epsilon_{t\alpha}(\gamma,s)\,d\gamma\,ds
 \end{equation}
for $\alpha<\beta$. Because 
 \begin{equation}\label{eq:447}
 \frac d{dt} |z^\epsilon_{\alpha}|^2=2|z^\epsilon_{\alpha}|^2 \Re D_{\alpha}z^\epsilon_t
 \end{equation}
 by  Gronwall, for $t\in [0, T_0]$, 
 \begin{equation}\label{eq:448}
 |z^\epsilon_{\alpha}(\alpha,t)|^2\le |z^\epsilon_{\alpha}(\alpha,0)|^2 e^{2\int_0^t |D_{\alpha}z^\epsilon_t(\alpha,\tau)|\,d\tau };
 \end{equation}
 so 
 \begin{equation}\label{eq:449}
 |z^\epsilon_{t\alpha}(\alpha,t)|\le |z^\epsilon_{\alpha}(\alpha,0)| |D_{\alpha}z^\epsilon_t(\alpha,t)|e^{\int_0^t |D_{\alpha}z^\epsilon_t(\alpha,\tau)|\,d\tau };
 \end{equation}
 by Appendix~\ref{quantities}, \eqref{eq:400} and Proposition~\ref{prop:energy-eq},
  \begin{equation}\label{eq:450}
\sup_{[0, T_0]} |z^\epsilon_{t\alpha}(\alpha,t)|\le |z^\epsilon_{\alpha}(\alpha,0)| C(\mathcal E_1(0)).
 \end{equation}
 therefore for $t\in [0, T_0]$, 
 \begin{equation}\label{eq:451}
 \int_0^{t}\int_{\alpha}^\beta |z^\epsilon_{t\alpha}(\gamma,s)|\,d\gamma\,ds\le t C(\mathcal E_1(0))\int_{\alpha}^\beta |z^\epsilon_{\alpha}(\gamma,0)| \,d\gamma
 \end{equation}
 Now $z^\epsilon(\alpha,0)=Z^\epsilon(\alpha,0)=\Psi(\alpha-\epsilon i, 0)$. Because $Z_{,\alpha'}(\cdot,0)\in L^1_{loc}(\mathbb R)$, and $Z_{,\alpha'}(\cdot,0)-1\in H^1(\mathbb R\setminus [-N, N])$ for some large $N$, 
 \begin{equation}\label{eq:452}
 \overline{\lim_{\epsilon\to 0}}\int_{\alpha}^\beta |\Psi_{z'}(\gamma-\epsilon i,0)|\,d\gamma\le \int_{\alpha}^\beta |Z_{,\alpha'}(\gamma, 0)|\,d\gamma
 \end{equation}
 Let $\epsilon=\epsilon_j\to 0$ in \eqref{eq:446}. We get, for $t\in [0, T_0]$, 
 \begin{equation}\label{eq:453}
| |z(\alpha,t)-z(\beta,t)| -|Z(\alpha,0)-Z(\beta,0)||\le 
tC(\mathcal E_1(0))\int_{\alpha}^\beta |Z_{,\alpha'}(\gamma, 0)|\,d\gamma
 \end{equation}
 hence for all $\alpha<\beta$ and $0\le t\le \min\{T_0, \frac{\delta}{2C(\mathcal E_1(0))}\}$, 
 \begin{equation}\label{eq:454}
\frac12\delta \int_\alpha^\beta |Z_{,\alpha'}(\gamma,0)|\,d\gamma\le  |z(\alpha,t)-z(\beta,t)|\le 2 \int_\alpha^\beta |Z_{,\alpha'}(\gamma,0)|\,d\gamma
 \end{equation}
 This show that for $\le t\le \min\{T_0, \frac{\delta}{2C(\mathcal E_1(0))}\}$,
 $z=z(\cdot,t)$ is absolute continuous on compact intervals of $\mathbb R$, with $z_{\alpha}(\cdot,t)\in L_{loc}^1(\mathbb R)$, and is chord-arc.  So $\Sigma(t)=\{z(\alpha,t) \ | \ \alpha\in \mathbb R\}$ is Jordan. This finishes the proof of Theorem~\ref{th:local}.
 
\begin{appendix}

\section{Basic analysis preparations}\label{ineq}
Let $\Omega\subset \mathbb C$ be a domain with boundary $\Sigma: z=z(\alpha)$, $\alpha\in  I$, oriented clockwise. Let $\mathfrak H$ be the Hilbert transform associated to $\Omega$:
\begin{equation}\label{hilbert-t}
\frak H f(\alpha)=\frac1{\pi i}\, \text{pv.}\int\frac{z_\beta(\beta)}{z(\alpha)-z(\beta)}f(\beta)\,d\beta
\end{equation}
We have the following characterization of the trace of a holomorphic function on $\Omega$.

\begin{proposition}\cite{jour}\label{prop:hilbe}
a.  Let $g \in L^p$ for some $1<p <\infty$. 
  Then $g$ is the boundary value of a holomorphic function $G$ on $\Omega$ with $G(z)\to 0$ at infinity if and only if 
  \begin{equation}
    \label{eq:1571}
    (I-\mathfrak H) g = 0.
  \end{equation}

b. Let $ f \in L^p$ for some $1<p<\infty$. Then $ \mathbb P_H  f:=\frac12(I+\mathfrak H)  f$ is the boundary value of a holomorphic function $\frak G$ on $\Omega$, with $\frak G(z)\to 0$ as $|z|\to \infty$.

c. $\mathfrak H1=0$.
\end{proposition}
Observe  that Proposition~\ref{prop:hilbe} gives  $\frak H^2=I$ in $L^p$.

We next present the basic estimates we will rely on for this paper.  We start with the Sobolev inequality.

\begin{proposition}[Sobolev inequality]\label{sobolev}
Let $f\in C^1_0(\mathbb R)$. Then
\begin{equation}\label{eq:sobolev}
\|f\|_{L^\infty}^2\le 2\|f\|_{L^2}\|f'\|_{L^2}
\end{equation}
\end{proposition}

\begin{proposition}[Hardy's Inequality]
\label{hardy-inequality}  
Let $f \in C^1(\mathbb R)$, 
with $f' \in L^2(\mathbb R)$. Then there exists $C > 0$ independent of $f$ such that for any $x \in \mathbb R$,
\begin{equation}
  \label{eq:77}
\abs{\int \f{(f(x) - f(y))^2}{(x-y)^2} dy} \le C \nm{f'}_{L^2}^2.
\end{equation}
\end{proposition}

Let $$\mathbb Hf(x)=\frac1{\pi i}\text{pv.}\int\frac1{x-y}f(y)\,dy.$$
be the Hilbert transform associated with $P_-$.
Let $f :\mathbb R\to\mathbb C$ be a function in $\dot H^{1/2}$, we note that
\begin{equation}\label{def-hhalf}
\|f\|_{\dot H^{1/2}}^2= \int i\mathbb H \partial_x f(x) \bar f(x)\,dx=\frac1{2\pi}\iint\frac{|f(x)-f(y)|^2}{(x-y)^2}\,dx\,dy.
\end{equation}
We have the following result on $\dot H^{1/2}$ functions. 

\begin{proposition}\label{prop:Hhalf}
Let $f,\ g\in C^1(\mathbb R)$. Then
\begin{equation}\label{Hhalf}
\|g\|_{\dot H^{1/2}}\lesssim \|f^{-1}\|_{L^\infty}(\|fg\|_{\dot H^{1/2}}+\|f'\|_{L^2}\|g\|_{L^2}).
\end{equation}

\end{proposition}
The proof is straightforward from the definition of $\dot H^{1/2}$ and the Hardy's inequality. We omit the
details.

Let $ A_i\in C^1(\mathbb R)$, $i=1,\dots m$. Define
\begin{equation}\label{3.15}
C_1(A_1,\dots, A_m, f)(x)=\text{pv.}\int \frac{\Pi_{i=1}^m(A_i(x)-A_i(y))}{(x-y)^{m+1}}f(y)\,dy.
\end{equation}

\begin{proposition}\label{B1} There exist  constants $c_1>0$, $c_2>0$, such that 

1. For any $f\in L^2,\ A_i'\in L^\infty, \ 1\le i\le m, $
\begin{equation}\label{3.16}
\|C_1(A_1,\dots, A_m, f)\|_{L^2}\le c_1\|A_1'\|_{L^\infty}\dots\|A_m'\|_{L^\infty}\|f\|_{L^2}. 
\end{equation}
2. For any $ f\in L^\infty, \ A_i'\in L^\infty, \ 2\le i\le m,\ A_1'\in L^2$, 
\begin{equation}\label{3.17}
\|C_1(A_1,\dots, A_m, f)\|_{L^2}\le c_2\|A_1'\|_{L^2}\|A'_2\|_{L^\infty}\dots\|A_m'\|_{L^\infty}\|f\|_{L^\infty}.
\end{equation}
\end{proposition}
\eqref{3.16} is a result of Coifman, McIntosh and Meyer \cite{cmm}. \eqref{3.17} is a consequence of the Tb Theorem, a proof  is given in \cite{wu3}.

Let  $A_i$ satisfies the same assumptions as in \eqref{3.15}. Define
\begin{equation}\label{3.19}
C_2(A, f)(x)=\int \frac{\Pi_{i=1}^m(A_i(x)-A_i(y))}{(x-y)^m}\partial_y f(y)\,dy.
\end{equation}
We have the following inequalities.
\begin{proposition}\label{B2} There exist constants $c_3$, $c_4$ and $c_5$, such that 

1. For any $f\in L^2,\ A_i'\in L^\infty, \ 1\le i\le m, $
\begin{equation}\label{3.20}
\|C_2(A, f)\|_{L^2}\le c_3\|A_1'\|_{L^\infty}\dots\|A_m'\|_{L^\infty}\|f\|_{L^2}.
\end{equation}

2. For any $ f\in L^\infty, \ A_i'\in L^\infty, \ 2\le i\le m,\ A_1'\in L^2$,
\begin{equation}\label{3.21}
\|C_2(A, f)\|_{L^2}\le c_4\|A_1'\|_{L^2}\|A'_2\|_{L^\infty}\dots\|A_m'\|_{L^\infty}\|f\|_{L^\infty}.\end{equation}

3. For any $f'\in L^2, \ A_1\in L^\infty,\ \ A_i'\in L^\infty, \ 2\le i\le m, $
\begin{equation}\label{3.22}
\|C_2(A, f)\|_{L^2}\le c_5\|A_1\|_{L^\infty}\|A'_2\|_{L^\infty}\dots\|A_m'\|_{L^\infty}\|f'\|_{L^2}.\end{equation}

\end{proposition}

Using integration by parts, the operator $C_2(A, f)$ can be easily converted into a sum of operators of the form $C_1(A,f)$. \eqref{3.20} and \eqref{3.21} follow from \eqref{3.16} and \eqref{3.17}.  To get \eqref{3.22}, we rewrite $C_2(A,f)$ as the difference of the two terms $A_1C_1(A_2,\dots, A_m, f')$ and $C_1(A_2,\dots, A_m, A_1f')$ and apply \eqref{3.16} to each term.

\begin{proposition}\label{prop:half-dir}
 There exists a constant $C > 0$ such that for any $f, g \in C^1(\mathbb R)$ with $f' \in L^2$ and $g' \in L^2$,
\begin{align} 
 &\nm{[f,\HH]  g}_{L^2} \le C \nm{f}_{\dot{H}^{1/2}}\nm{g}_{L^2} \label{eq:b10}\\&
  \nm{[f,\HH] \partial_\aa g}_{L^2} \le C \nm{f'}_{L^2} \nm{g}_{\dot{H}^{1/2}}\label{eq:b11}
  \end{align}

\end{proposition}
\eqref{eq:b10} is straightforward  by Cauchy-Schwarz and the definition of $\dot H^{1/2}$. \eqref{eq:b11} follows from integration by parts,
then Cauchy-Schwarz, Hardy's inequality, the definition of $\dot H^{1/2}$ and \eqref{eq:b10}.

Recall $[f,g;h]$ as given in \eqref{eq:comm}.

\begin{proposition}
 There exists a constant $C > 0$ such that for any $f, g \in C^1(\mathbb R)$ with $f', g' \in L^2$ and $h \in L^2$,
  \begin{align}
    \label{eq:b12}
    \nm{[f,g;h]}_{L^2} &
        \le C \nm{f'}_{L^2} \nm{g'}_{L^2} \nm{h}_{L^2};\\ \label{eq:b15}
     \nm{[f,g;h]}_{L^\infty}&\le C \nm{f'}_{L^2} \nm{g'}_{L^\infty} \nm{h}_{L^2}.
  \end{align}
\end{proposition}
\eqref{eq:b12} follows directly from Cauchy-Schwarz, Hardy's inequality and Fubini Theorem; \eqref{eq:b15} follows from Cauchy-Schwarz, Hardy's inequality and the mean value Theorem.

\begin{proposition}
 There exists a constant $C > 0$ such that for any $f \in C^1(\mathbb R)$ with $f' \in L^2$, $g \in L^2$,
  \begin{equation}
    \label{eq:b13}
    \nm{[f,\HH] g}_{L^\infty} \le C \nm{f'}_{L^2} \nm{g}_{L^2}.
  \end{equation}
 \end{proposition}
 \eqref{eq:b13} is straightforward from Cauchy-Schwarz and Hardy's inequality.

\begin{proposition}
 There exists a constant $C > 0$ such that for any $f, g \in C^1(\mathbb R)$ with $f', g' \in L^2$, and $h \in L^2$,
\begin{equation}
  \label{eq:b14}
  \nm{\partial_\aa [f,[g,\HH]] h}_{L^2} \lec \nm{f'}_{L^2} \nm{g'}_{L^2} \nm{h}_{L^2}.
\end{equation}
\end{proposition}
Taking derivative under the integral $[f,[g,\HH]] h$, \eqref{eq:b14} directly follows from \eqref{eq:b12} and \eqref{eq:b13}.

\section{Identities}\label{iden}

\subsection{Basic identities}\label{basic-iden}
Here we derive a few basic identities from the system \eqref{interface-r}-\eqref{interface-holo}, without assuming $Z=Z(\cdot, t)$ being non-self-intersecting. Theses identities provide an alternative way of deriving the quasi-linearization of the system \eqref{interface-r}-\eqref{interface-holo} in this more general context, they also show that the argument in \cite{kw} can be modified, so that the a priori estimate of \cite{kw} and the characterization of the energy in \S10 of \cite{kw} hold for solutions of the system \eqref{interface-r}-\eqref{interface-holo} without the non-self-intersecting requirement. 

Let $Z=Z(\cdot, t)$ be sufficiently regular\footnote{Here we do not specify what precisely "sufficiently regular" means, but assume it is enough so that the calculations make sense.}  and satisfy \eqref{interface-r}-\eqref{interface-holo}:
\begin{equation}\label{c1}
\begin{cases}
Z_{tt}+i=i\mathcal AZ_{,\alpha'},\\
\bar{Z}_t=\mathbb H \bar{Z}_t,\\
Z_{,\alpha'}-1=\mathbb H(Z_{,\alpha'}-1),\qquad \frac1{Z_{,\alpha'}}-1=\mathbb H(\frac1{Z_{,\alpha'}}-1);
\end{cases}
\end{equation}
where $Z$ and $Z_t$ are related through \eqref{1001}, \eqref{1002}:
\begin{equation}\label{c2}
z(\alpha,t)=Z(h(\alpha,t),t),\qquad z_t(\alpha,t)=Z_t(h(\alpha,t),t)
\end{equation}
for some (sufficiently regular) homeomorphism $h(\cdot,t): \mathbb R\to \mathbb R$. 
Let $\frak a h_\alpha:=\mathcal A\circ h$, $A_1:= \mathcal A|Z_{,\alpha'}|^2$.  Precomposing the first equation of \eqref{c1} with $h$ gives \eqref{interface-l}:
\begin{equation}\label{c9}
z_{tt}+i=i\frak a z_{\alpha}
\end{equation}

We first show that \eqref{b} can be derived from \eqref{c1} and \eqref{c2}. Let $\Psi$ be a holomorphic function on $P_-$, continuously differentiable on $\overline P_-$, such that
$$\Psi(\alpha', t)=Z(\alpha',t),\qquad \Psi_{z'}(\alpha', t)=Z_{,\alpha'}(\alpha',t).$$
Therefore $z(\alpha,t)=\Psi(h(\alpha, t),t)$ and by the chain rule,
$z_t=\Psi_t\circ h+h_t\Psi_{z'}\circ h $. Precomposing with $h^{-1}$ then gives
$$Z_t=\Psi_t+ Z_{,\alpha'} h_t\circ h^{-1};$$
dividing by $Z_{,\alpha'}$ yields
\begin{equation}\label{c3}
h_t\circ h^{-1}(\alpha',t)= \frac{Z_t(\alpha', t)}{Z_{,\alpha'}(\alpha', t)} - \frac {\Psi_t}{\Psi_{z'}}(\alpha', t).
\end{equation}
Notice that $\frac {\Psi_t}{\Psi_{z'}}$ is a holomorphic function on $P_-$. By Proposition~\ref{prop:hilbe}, applying $(I-\mathbb H)$ to \eqref{c3} then taking the real parts and using the second and third equations of\eqref{c1} to rewrite into the commutator gives \eqref{b}. Conversely, if $h$ satisfies \eqref{b} for a function $Z$ satisfying the second and third equations of \eqref{c1}, then expanding the commutator yields 
\begin{equation}\label{c4}
h_t\circ h^{-1}=\Re (I-\mathbb H)(\frac{Z_t}{Z_{,\alpha'}})=\frac{Z_t}{Z_{,\alpha'}}+\frac12(I+\mathbb H)(\frac{\bar Z_t}{\bar Z_{,\alpha'}}-\frac{Z_t}{Z_{,\alpha'}}).
\end{equation}
By Proposition~\ref{prop:hilbe}, $\frac12(I+\mathbb H)(\frac{\bar Z_t}{\bar Z_{,\alpha'}}-\frac{Z_t}{Z_{,\alpha'}})$ is the boundary value of a holomorphic function on $P_-$, tending to zero at the spatial infinity.

In what follows we use the following notations. We write $U_1\equiv U_2$, if $(I-\mathbb H)(U_1-U_2)=0$; that is if $U_1-U_2$ is the boundary value of a holomorphic function on $P_-$ that tends to zero at infinity.

Assume $Z$ satisfies the second and third equations of \eqref{c1} and $h$ satisfies \eqref{b}, so \eqref{c4} holds. 

\begin{proposition}\label{prop:basic-iden}  Let $U(\cdot, t): \mathbb R\to \mathbb C$ be sufficiently regular, and $u=U\circ h$.  Assume $U\equiv 0$. We have
1. \begin{equation}\label{c5}
u_t\circ h^{-1}\equiv Z_t D_{\alpha'} U;
\end{equation}
2. \begin{equation}\label{c6}
u_{tt}\circ h^{-1}\equiv Z_{tt} D_{\alpha'} U+2 Z_t D_{\alpha'}( u_t\circ h^{-1}-Z_t  D_{\alpha'}U)+ Z_t^2D_{\alpha'}^2 U.
\end{equation}
3. \begin{equation}\label{c8}
U_h^{-1}(u_{tt}+i\frak a\partial_{\alpha}u)\equiv 2 Z_{tt} D_{\alpha'} U+2 Z_t D_{\alpha'}( u_t\circ h^{-1}-Z_t  D_{\alpha'}U)+ Z_t^2D_{\alpha'}^2 U.
\end{equation}
\end{proposition}

\begin{proof}
Applying the chain rule to $u=U\circ h$ and precompose with $h^{-1}$ gives  $$u_t\circ h^{-1}=\partial_t U+ \partial_{\alpha'} U h_t\circ h^{-1}.$$
Observe that $U\equiv 0$ gives $\partial_t U\equiv 0$ and $\partial_{\alpha'}U\equiv 0$.  \eqref{c5} follows from \eqref{c4} and the fact that product of holomorphic functions is holomorphic. 

Now we apply \eqref{c5} to $u_t\circ h^{-1}- Z_t D_{\alpha'} U$. This gives
\begin{equation}\label{c7}
U_{h}^{-1} \partial_t(u_t- z_tD_{\alpha}u)\equiv Z_t D_{\alpha'} (u_t\circ h^{-1}- Z_t D_{\alpha'} U).
\end{equation}
Expanding the left hand side by the product rule, and observe that $\partial_t D_{\alpha}u=D_{\alpha} (u_t- z_t D_{\alpha} u)+z_tD_{\alpha}^2u$, so 
\begin{align*}
\partial_t(u_t- z_tD_{\alpha}u)&=  u_{tt}-z_{tt}D_{\alpha}u-z_t\partial_t D_{\alpha}u\\&=u_{tt}-z_{tt}D_{\alpha}u-z_t  D_{\alpha} (u_t- z_t D_{\alpha} u)-z_t^2D_{\alpha}^2u.
\end{align*}
Precomposing with $h^{-1}$ and substituting in \eqref{c7} gives \eqref{c6}. 

\eqref{c8} follows from \eqref{c6} and the fact that $i\frak a \partial_{\alpha}u= (z_{tt}+i) D_{\alpha}u$ and $D_{\alpha'}U\equiv 0$. 

\end{proof} 
Now assume $Z$ satisfies \eqref{c1}\footnote{Here $Z=Z(\cdot,t)$ need not be non-self-intersecting.}. Applying \eqref{c5} to $\bar Z_t$ gives $\bar Z_{tt}\equiv Z_tD_{\alpha'}\bar Z_t$.  Following the rest of the argument in  
section 2.2.1 of \cite{wu6} gives 
 \eqref{a1}.  Similarly, applying \eqref{c8} to $\bar Z_t$ and following the rest of the argument in  
section 2.2.3 of \cite{wu6} gives  
\begin{equation}\label{c10}
\frac{\frak a_t}{\frak a}\circ h^{-1}=\frac{-\Im( 2[Z_t,\mathbb H]{\bar Z}_{tt,\alpha'}+2[Z_{tt},\mathbb H]\partial_{\alpha'} \bar Z_t-
[Z_t, Z_t; D_{\alpha'} \bar Z_t])}{A_1}
\end{equation}
where
\begin{equation}\label{c11}
[Z_t, Z_t; D_{\alpha'} \bar Z_t]:=\frac1{\pi i}\int\frac{(Z_t(\alpha',t)-Z_t(\beta',t))^2}{(\alpha'-\beta')^2} D_{\beta'} \bar Z_t(\beta',t)\,d\beta'
\end{equation}
For the periodic case studied in \cite{kw}, the same computations above  and Proposition~\ref{prop:basic-iden} hold, and the corresponding equations for \eqref{b}, \eqref{a1}, \eqref{c10} can be derived without the non-self-intersecting assumption. 
The periodic version of Proposition~\ref{prop:basic-iden}  shows that the argument in \cite{kw} can be modified so that the a priori estimate, Theorem 2 of \cite{kw} and the characterization of the energy in \S10 of \cite{kw} hold more generally without the non-self-intersecting assumption.  Proposition~\ref{prop:basic-iden} and a small modification of the argument in \cite{kw}  show that a similar a priori estimate and a similar characterization of the energy as in \cite{kw} hold in the whole line case for solutions of \eqref{interface-r}-\eqref{interface-holo}. 

\subsection{Commutator identities}\label{comm-iden}

We include here for reference the various commutator identities that are necessary. The first set: \eqref{eq:c1}-\eqref{eq:c5} has already appeared in \cite{kw}.
\begin{align}
  \label{eq:c1}
  [\partial_t,D_\a] &= - (D_\a z_t) D_\a;\\ 
  \label{eq:c2}
    \bracket{\partial_t,D_\a^2} &   
        = -2(D_\a z_t) D_\a^2 - (D_\a^2 z_t) D_\a;\\
  \label{eq:c3}
     \bracket{\partial_t^2,D_\a} &=(-D_\a z_{tt}) D_\a + 2(D_\a z_t)^2 D_\a - 2(D_\a z_t) D_\a \partial_t;\\
  \label{eq:c5}
   \bracket{\partial_t^2 + i \mathfrak{a} \partial_\a, D_\a} &=  (-2D_\a z_{tt}) D_\a -2(D_\a z_t)\partial_t D_\a.
   \end{align}

We need some additional commutator identities. In general for operators $A, B$ and $C$,
\begin{equation}\label{eq:c12}
[A, BC^k]=[A, B]C^k+ B[A, C^k]=[A, B]C^k+ \sum_{i=1}^k BC^{i-1}[A, C]C^{k-i}.
\end{equation}
We note that for $f=f(\cdot, t)$, 
$U_h \partial_{\alpha'} U_{h^{-1}}f=\frac{\partial_\alpha}{h_\alpha} f$.  So
\begin{align}
\label{eq:c7}
[\partial_t, \frac{\partial_\alpha}{h_\alpha}]f&=-\frac{h_{t\alpha}}{h_\alpha}\frac1{h_\alpha}\partial_\alpha f=-U_h\{(h_t\circ h^{-1})_{\alpha'}\partial_{\alpha'}U_{h^{-1}}f\};\\
\label{eq:20}
[U_h^{-1}\partial_t U_h, \partial_{\alpha'}]g&=U_h^{-1}[\partial_t, \frac{\partial_\alpha}{h_\alpha}]U_hg= -(h_t\circ h^{-1})_{\alpha'}\partial_{\alpha'}g.
\end{align}
Applying \eqref{eq:c12} yields
\begin{equation}
\label{eq:c11}
\begin{aligned}
\bracket{\partial_t, \big(\frac{\partial_\alpha}{h_\alpha}\big)^2}f&=\frac{\partial_\alpha}{h_\alpha} [\partial_t, \frac{\partial_\alpha}{h_\alpha}]f+[\partial_t, \frac{\partial_\alpha}{h_\alpha}]\frac{\partial_\alpha}{h_\alpha} f\\&=-2U_h\{(h_t\circ h^{-1})_{\alpha'}\partial_{\alpha'}^2U_{h^{-1}}f\}-U_h\{\partial_{\alpha'}^2(h_t\circ h^{-1})\partial_{\alpha'}U_{h^{-1}}f\};
\end{aligned}
\end{equation}
\begin{equation}
\label{eq:c8}
\begin{aligned}
\bracket{\partial_t^2, \frac{\partial_\alpha}{h_\alpha}}f&=\partial_t[\partial_t, \frac{\partial_\alpha}{h_\alpha}]f+[\partial_t, \frac{\partial_\alpha}{h_\alpha}]\partial_t f\\&=-\partial_tU_h\{(h_t\circ h^{-1})_{\alpha'}\partial_{\alpha'}U_{h^{-1}}f\}-U_h\{(h_t\circ h^{-1})_{\alpha'}\partial_{\alpha'}U_{h^{-1}}f_t\}.
\end{aligned}
\end{equation}
To calculate $[i\frak a\partial_\alpha, \frac{\partial_\alpha}{h_\alpha}]f$, we use the definition $\mathcal A\circ h:=\frak a h_\alpha$, and $i\frak a\partial_\alpha:=i\mathcal A\circ h \frac{\partial_\alpha}{h_\alpha}$. We have
\begin{equation}\label{eq:c9}
[i\frak a\partial_\alpha, \frac{\partial_\alpha}{h_\alpha}]f=[i\mathcal A\circ h \frac{\partial_\alpha}{h_\alpha}, 
\frac{\partial_\alpha}{h_\alpha}]f=-iU_h\{\mathcal A_{\alpha'}\partial_{\alpha'}U_{h^{-1}}f\}.
\end{equation}
Adding \eqref{eq:c8} and \eqref{eq:c9}, we conclude that
\begin{equation}\label{eq:c10}
\begin{aligned}
\bracket{\partial_t^2+i\frak a\partial_\alpha, \frac{\partial_\alpha}{h_\alpha}} f=-\partial_tU_h\{(h_t\circ h^{-1})_{\alpha'}\partial_{\alpha'}&U_{h^{-1}}f\}-U_h\{(h_t\circ h^{-1})_{\alpha'}\partial_{\alpha'}U_{h^{-1}}f_t\}
\\&-iU_h\{\mathcal A_{\alpha'}\partial_{\alpha'}U_{h^{-1}}f\}.
\end{aligned}
\end{equation}
We note that  $U_h^{-1}\partial_tU_h =\partial_t+b\partial_{\alpha'}$ where $b:=h_t\circ h^{-1}$. Therefore
\begin{equation}\label{eq:c21}
[U_h^{-1}\partial_tU_h, \mathbb H]=[h_t\circ h^{-1},\mathbb H]\partial_{\alpha'}
\end{equation}
A straightforward differentiation gives
\begin{equation}\label{eq:c14'}
\begin{aligned}
U_h^{-1}\partial_t U_h& [f,\mathbb H]g=[U_h^{-1}\partial_t U_h f,\mathbb H]g\\&+ [f,\mathbb H](U_h^{-1}\partial_t U_h g+(h_t\circ h^{-1})_{\alpha'} g)-[f, h_t\circ h^{-1}; g];
\end{aligned}
\end{equation}
with an application of \eqref{eq:20} yields
\begin{equation}\label{eq:c14}
\begin{aligned}
U_h^{-1}\partial_t U_h& [f,\mathbb H]\partial_{\alpha'}g=
[U_h^{-1}\partial_t U_h f,\mathbb H]\partial_{\alpha'}g\\&+ [f,\mathbb H]\partial_{\alpha'}U_h^{-1}\partial_t U_h g-[f, h_t\circ h^{-1}; \partial_{\alpha'}g].
\end{aligned}
\end{equation}
The following commutators are straightforward from the product rule. We have
\begin{equation}\label{eq:c13}
\begin{aligned}
&[Z_{,\alpha'}, U_h^{-1}\partial_t U_h] f=[U_h^{-1}\frac{z_\alpha}{h_\alpha}, U_h^{-1}\partial_t U_h] f\\&=-\{U_h^{-1}\partial_t\big(\frac{z_\alpha}{h_\alpha}\big)\}f=-Z_{,\alpha'}(D_{\alpha'}Z_t-(h_t\circ h^{-1})_{\alpha'})f;
\end{aligned}
\end{equation}
\begin{equation}\label{eq:c15}
[\partial_t, \frac{h_\alpha}{z_\alpha}]f=\partial_t\big(\frac{h_\alpha}{z_\alpha}\big)f=\frac{h_\alpha}{z_\alpha}(U_h(h_t\circ h^{-1})_{\alpha'}-D_\alpha z_t)f;
\end{equation}
by $i\frak a z_\alpha=z_{tt}+i$  \eqref{interface-l}, 
\begin{equation}\label{eq:c17}
[i\frak a \partial_\alpha, \frac{h_\alpha}{z_\alpha}]f=[(z_{tt}+i)D_\alpha, \frac{h_\alpha}{z_\alpha}]f=(z_{tt}+i)D_\alpha\big(\frac{h_\alpha}{z_\alpha}\big)f;
\end{equation}
by \eqref{eq:c12}, \eqref{eq:c15}, \eqref{eq:c17} and the product rule,
\begin{equation}\label{eq:c16}
\begin{aligned}
&[\partial_t^2+i\frak a\partial_\alpha, \frac{h_\alpha}{z_\alpha}]f=2\frac{h_\alpha}{z_\alpha}(U_h(h_t\circ h^{-1})_{\alpha'}-D_\alpha z_t)f_t+\frac{h_\alpha}{z_\alpha}(U_h(h_t\circ h^{-1})_{\alpha'}-D_\alpha z_t)^2f \\&+\frac{h_\alpha}{z_\alpha}(\partial_t U_h(h_t\circ h^{-1})_{\alpha'}-\partial_tD_\alpha z_t)f+ 
(z_{tt}+i)D_\alpha\big(\frac{h_\alpha}{z_\alpha}\big)f.
\end{aligned}
\end{equation}

\section{Main quantities controlled by $\frak E$} \label{quantities}
We list here the various quantities that we have shown in \cite{kw}  are controlled by polypromials of $\frak E(t)$. 
\footnote{The same proof for the symmetric periodic setting in \cite{kw} applies to the whole line setting. We leave it to the reader to check the details.}

\begin{equation}\label{eq:1550}
\begin{aligned}
&\nm{D_\aa^2 \bar{Z}_{tt}}_{L^2}, \nm{D_\aa^2 {Z}_{tt}}_{L^2},  \nm{D_\aa^2 \bar{Z}_t}_{L^2}, \nm{D_\aa^2 {Z}_t}_{L^2}, \nm{D_\a \partial_t D_\a \bar{z}_t}_{L^2({h_\a}d\a)},\\& \nm{\f{1}{Z_{,\aa}} D_\aa^2 \bar{Z}_t}_{\dot{H}^{1/2}}, \nm{D_\aa \bar{Z}_{tt}}_{L^\infty}, \nm{D_\aa {Z}_{tt}}_{L^\infty},  \nm{D_\aa \bar{Z}_t}_{L^\infty}, \nm{D_\aa {Z}_t}_{L^\infty},\\&  \nm{ \bar{Z}_{tt, \aa}}_{L^2},   \nm{ \bar{Z}_{t, \aa}}_{L^2},  \int \abs{D_\a \bar{z}_t}^2 \f{d\a}{\af},\  \int \abs{D_\a \bar{z}_{tt}}^2 \f{d\a}{\af},  \ \nm{\f{1}{Z_{,\aa}}}_{L^\infty}, \nm{Z_{tt}+i}_{L^\infty} , \nm{A_1}_{L^\infty};
\end{aligned}
\end{equation}

\begin{itemize}
 
\item $\nm{\f{\af_t}{\af}}_{L^\infty} = \nm{\f{\AAt}{\AA}}_{L^\infty}$;

\item $\nm{\partial_\aa \f{1}{Z_{,\aa}}}_{L^2}$;

\item $\nm{\f{h_{t\a}}{h_\a}}_{L^\infty}$;

\item $\nm{(I+\HH) D_\aa Z_t}_{L^\infty}$;

\item $\nm{D_\aa \f{1}{Z_{,\aa}}}_{L^\infty}$, $\nm{(Z_{tt} + i) \partial_\aa \f{1}{Z_{,\aa}}}_{L^\infty}$;

\item $\nm{\partial_\aa \P_A \f{Z_t}{Z_{,\aa}}}_{L^\infty}$, $\nm{\P_A \paren{Z_t \partial_\aa \f{1}{Z_{,\aa}}}}_{L^\infty}$.
\end{itemize}

In addition from (179), (186) of \cite{kw},
$$\nm{D_{\alpha'}(h_t\circ h^{-1})_{\alpha'}}_{L^2}\lesssim C(\frak E).$$

\end{appendix}

\end{document}